\documentclass[a4paper]{amsart}
\usepackage{amsfonts,amssymb}
\usepackage[all]{xy}
\usepackage[pdftex,pdfencoding=auto]{hyperref}
\DeclareMathOperator{\bN}{{\mathbb{N}}}
\DeclareMathOperator{\gp}{gp}
\DeclareMathOperator{\Res}{Res}
\DeclareMathOperator{\Spec}{Spec}
\DeclareMathOperator{\Zar}{Zar}
\DeclareMathOperator{\et}{\acute{e}t}
\DeclareMathOperator{\dR}{dR}
\DeclareMathOperator{\dRW}{dRW}
\DeclareMathOperator{\crys}{crys}
\DeclareMathOperator{\Rcrys}{Rcrys}
\DeclareMathOperator{\logcrys}{log-crys}
\DeclareMathOperator{\MW}{MW}
\DeclareMathOperator{\logMW}{log-MW}
\DeclareMathOperator{\rig}{rig}
\DeclareMathOperator{\ord}{ord}
\DeclareMathOperator{\Supp}{Supp}
\DeclareMathOperator{\Fil}{Fil}

\begin{document}
\title[On relative and overconvergent de Rham-Witt cohomology]{On relative and overconvergent de Rham-Witt cohomology for log schemes}
\author{Hironori Matsuue}
\thanks{Planning Dept., Treasury Unit,Sumitomo Mitsui Banking Corporation,1-1-2, Marunouchi, Chiyoda-ku, Tokyo, Japan}
\thanks{E-mail:\ \url{matsuue.hironori@gmail.com}}

\newtheorem{thm}{Theorem}[section]
\newtheorem{lem}[thm]{Lemma}
\newtheorem{prop}[thm]{Proposition}
\newtheorem{prop-def}[thm]{Proposition-Definition}
\newtheorem{cor}[thm]{Corollary}
\theoremstyle{definition}
\newtheorem{defn}[thm]{Definition}
\newtheorem{rem}[thm]{Remark}

\begin{abstract}
We construct the relative log de Rham-Witt complex.
This is a generalization of the relative de Rham-Witt complex of Langer-Zink to log schemes.
We prove the comparison theorem between the hypercohomology of the log de Rham-Witt complex
and the relative log crystalline cohomology in certain cases.
We construct the $p$-adic weight spectral sequence for relative proper strict semistable log schemes.
When the base log scheme is a log point, We show it degenerates at $E_2$
after tensoring with the fraction field of the Witt ring.
We also extend the definition of the overconvergent de Rham-Witt complex of Davis-Langer-Zink
to log schemes $(X,D)$ associated with smooth schemes with simple normal crossing divisor over a perfect field.
Finally, we compare its hypercohomology with the rigid cohomology of $X \setminus D$.
\end{abstract}

\maketitle
\tableofcontents

\section{Introduction}
The de Rham-Witt complex $\{ W_m \Omega_X^\bullet \}_{m\in \bN }$ was defined by Illusie \cite{ICrys} for a scheme $X$ of characteristic $p>0$.
He defined it as the initial object of $V$-pro-complexes.
When $X$ is smooth over a perfect scheme, the hypercohomology of the de Rham-Witt complex computes the crystalline cohomology.
Also, Illusie and Raynaud \cite{IR} remarked that
one can define the de Rham-Witt complex by using
the crystalline cohomology sheaf in the process of definition.

Langer and Zink \cite{LZDRW} extended Illusie's definition to relative situations.
Let $S$ be a $\mathbb{Z}_{(p)}$-scheme such that $p$ is nilpotent in $S$.
They defined the log de Rham-Witt complex $\{ W_m \Omega_{X/S}^\bullet \}_{m\in \bN }$
for a scheme $X$ over  $S$.
Their definition is close to that of Illusie: In fact, they
defined it as the initial object of $F$-$V$-pro-complexes.
The hypercohomology of Langer and Zink's de Rham-Witt complex also computes the  crystalline cohomology in smooth cases.

Olsson \cite{OlCrys} extended Langer-Zink's definition to the case of algebraic stacks.
He also gave another possible definition of the de Rham-Witt complex
via the crystalline cohomology sheaf
and compared two definitions, but it seems that they do not always coincide.

It is natural to extend the definition of the de Rham-Witt complex to the case of log schemes in the sense of Fontaine-Illusie-Kato (\cite{KatoLog}),
which is our main interest.
Hyodo and Kato \cite{HK} defined the log de Rham-Witt complex
for a log smooth log scheme of Cartier type over a perfect field of characteristic $p>0$
by using the log crystalline cohomology sheaf.
They also proved the comparison theorem to the log crystalline cohomology.

Nakkajima \cite{Nak} introduced a theory of formal de Rham-Witt complexes
as a kind of axiomatization of Hyodo-Kato's construction.
It also covers the cohomological construction of the de Rham-Witt complex
for smooth schemes with simple normal crossing divisor
over a perfect field of characteristic $p>0$.

In this paper, we construct the log de Rham-Witt complex
for a fine log scheme $X$ over a fine log scheme $S$ over $\mathbb{Z}_{(p)}$.
We follow the definition of Langer-Zink,
and construct the log de Rham-Witt complex as the initial object of log $F$-$V$-pro-complexes.
Note that we cannot apply the methods of Hyodo, Kato and Nakkajima
directly to our log de Rham-Witt complex
because their methods seem to be applicable only to the case of perfect base log schemes
and because the definition using the log crystalline cohomology sheaf seems not to be good in the case of non-perfect base log schemes.
We prove the comparison theorem between the hypercohomology of
the log de Rham-Witt complex
and the relative log crystalline cohomology
in case of relative semistable log schemes
and that of log schemes associated to smooth schemes with normal crossing divisor.

Mokrane \cite{Mo} used the de Rham-Witt complex of Hyodo-Kato to construct
the ($p$-adic) weight spectral sequence for the crystalline cohomology of
strictly semistable log schemes.
He proved its $E_2$-degeneration modulo torsion
when the base scheme is the spectrum of a finite field.
Nakkajima \cite{Nak} extended his result to the case where the base
scheme is the spectrum of any perfect field by using the specialization
argument of Illusie-Deligne (\cite{IReport}).
In this paper, we construct the $p$-adic
weight spectral sequence for the relative crystalline cohomology of
a relative strictly semistable log schemes and prove its $E_2$-degeneration
modulo torsion when the base scheme is the spectrum of a (not necessarily perfect) field.

Since our definition of the log de Rham-Witt complex follows that of
Langer-Zink and differs from that of Hyodo-Kato,
the proof of our results is similar to that of Langer-Zink and
differs from that of Hyodo-Kato and Mokrane.
The key ingredient is to find certain explicit basis of the log de Rham-Witt complex
called the log basic Witt differentials in explicit cases,
which are generalizations of the basic Witt differentials of Langer-Zink.

We also introduce the notion of the overconvergent log de Rham-Witt complex.
Davis, Langer and Zink \cite{DLZODRW} introduced the notion of the overconvergent de Rham-Witt complex
for smooth schemes over a perfect field of characteristic $p > 0$.
They proved the comparison theorem between its hypercohomology and the Monsky-Washnitzer cohomology in the affine case.
They also proved that its hypercohomology calculates the rigid cohomology
in the case of smooth quasi-projective varieties using Gro\ss e-Kl\"onne's theory of dagger spaces \cite{GKDagger}.
We first treat the case of smooth affine varieties with simple normal crossing divisor over a perfect field of characteristic $p > 0$
such that they admit global coordinates and divisors are defined by the coordinates.
We define the overconvergent log de Rham-Witt complex in this case and prove the comparison theorem
between its hypercohomology and the log Monsky-Washnitzer cohomology of Tsuzuki \cite{Tsu}.
More generally, we can extend the definition of the overconvergent log de Rham-Witt complex
to arbitrary log schemes obtained by smooth schemes with simple normal crossing divisor.
By combining the result of local cases with a result in \cite{Tsu}, we can prove
the comparison theorem with rigid cohomology.

The content of each section is as follows:
In \S \ref{sect:preliminaries}, we fix notations which we use in this paper
and give the definition of the crystalline cohomology over non-adic base, which we need later.

In \S \ref{sect:log-F-V-procpx-and-dRWcpx},
we define the log version of $F$-$V$-procomplexes and the de Rham-Witt complex defined by Langer and Zink.
We extend their fundamental results to our log cases.

In \S \ref{sect:basicwittdiffs},
we define the log $p$-basic elements and the log basic Witt differentials in specific cases.
They are generalizations of the $p$-basic elements
and the basic Witt differentials defined in \cite{LZDRW} \S 2.1, 2.2.
We prove that any element of the log de Rham-Witt complex is
written as a convergent sum of the log basic Witt differentials.
The notion of the basic differentials is a powerful tool for us and it plays a role in proofs in the later sections.

In \S \ref{sect:log-witt-lift-and-log-frobenius-lift},
we give the definition of log Witt lifts and log Frobenius lifts for log smooth log schemes.
We prove that there exists a log Frobenius lift \'etale locally.

In \S \ref{sect:comparison-morphism},
We construct the comparison morphism between the log crystalline cohomology
and the hypercohomology of the log de Rham-Witt complex
for log smooth log schemes using log Frobenius lifts.

In \S \ref{sect:comparison},
we prove the comparison theorem for smooth schemes with normal crossing divisor and semistable log schemes.

In \S \ref{sect:ssq-semistable},
we define the weight filtrations of the log de Rham-Witt complex
and construct the $p$-adic Steenbrink complex for proper semistable log schemes over arbitrary base.
The $p$-adic Steenbrink complex defines a spectral sequence,
which we call the $p$-adic weight spectral sequence.
When the base scheme is the spectrum of a (not necessarily perfect) field,
we prove $E_2$-degeneration after tensoring with the fractional field of the Witt ring
by using Nakkajima's specialization method.

In \S \ref{sect:ssq-open-smooth},
we construct the $p$-adic weight spectral sequence of
proper smooth schemes with simple normal crossing divisor and prove
its $E_2$-degeneration (after tensoring with the fraction field of the Witt ring)
when the base scheme is the spectrum of a (not necessarily perfect) field.

In \S \ref{sect:overconv},
we give the definition of the overconvergent log de Rham-Witt complex
for a log scheme $(X,D)$ defined by a smooth scheme $X$
with simple normal crossing divisor $D$ over a perfect field $k$ of characteristic $p >0$.
We see that the overconvergent log de Rham-Witt complex coincides
with the overconvergent de Rham Witt complex of Davis-Langer-Zink (\cite{DLZODRW})
when the log structure is trivial.
We compare the overconvergent log de Rham-Witt cohomology
with the log Monsky-Washnitzer cohomology in affine cases,
and with the rigid cohomology of $X \setminus D$ in general cases.

Finally, note that there exist several other variants of the de Rham-Witt complex:
When $p$ is odd, Hesselholt and Madsen defined the absolute de Rham-Witt complex
$\{ W_m \Omega_X^\bullet \}_{m\in \bN }$ for any $\mathbb{Z}_{(p)}$-scheme $X$ (\cite{HMKtheory}, \cite{HMDRW}).
They also proved the existance of the absolute de Rham-Witt complex for pre-log rings
(\cite{HMKtheory} Proposition 3.2.2).
Hesselholt studied the relation with the Langer and Zink's relative de Rham-Witt complex using $K$-theoretic methods,
with brief sketch also in the logarithmic setting (\cite{He}).
When $p$ is odd and nilpotent in $S$ and $X$ is $S$-scheme,
there is a canonical surjective map $\{ W_m \Omega_X^\bullet \}_{m\in \bN }\twoheadrightarrow \{ W_m \Omega_{X/S}^\bullet \}_{m\in \bN }$
from the absolute de Rham-Witt complex to the Langer-Zink's relative de Rham-Witt complex.
Cuntz and Deninger defined the relative de Rham-Witt complex in arbitrary truncated sets
by a different approach so that the big and the $p$-isotypical theories are covered (\cite{CD}).
It would be an interesting problem to generalize their constructions
to the case of log schemes, and compare them with our construction.

We also remark that there are other studies
to construct a ($p$-adic) weight filtration.
Nakkajima and Shiho \cite{NS} construct a theory of weights
on the log crystalline cohomology of a family of open smooth variety.
They used the log de Rham complex of a lift to define
a weight filtration.
Nakkajima \cite{Nakpreprint} applied their method to
a proper truncated simplicial SNCL
(=simple normal crossing log) scheme
having affine truncated simplicial open covering.
Tsuji \cite{Tsuji} used filtrations of sheaves of $\mathcal{D}$-modules
to construct a weight spectral sequence for a semistable log scheme
over a complete discrete valuation ring.
It is also an interesting problem to consider their situations
using our de Rham-Witt complex and to compare with their results.

\subsection*{Acknowledgments}
This paper is based on my master thesis in the University of Tokyo
under the guidance of my supervisor Atsushi Shiho.
I would like to express my sincere gratitude to
him for the helpful discussions,
reading the draft several times and providing valuable
suggestions for improvement. This work would not have been possible without his advice.
I would also like to thank Yukiyoshi Nakkajima for sending me his preprint \cite{Nakpreprint}.

\subsection*{Notations}
We fix a prime number $p$ throughout this paper.
All schemes are assumed to be defined and separated over $\mathbb{Z}_{(p)}$.

Let $R$ be a ring. For a $W(R)$-module $N$,
we write $N_{[F]}$ for the $W(R)$-module whose underlying set is $N$ and its module structure is obtained by the Frobenius map $F:W(R)\to W(R)$.

If $R$ is an $\mathbb{F}_p$-algebra and $L$ is a sheaf of $W(R)$-modules equipped with an endomorphism $\phi$
which is $F$-linear ($F$: Frobenius  map on $W(R)$) and $r$ is a \textit{negative} integer, the Tate twist $L(r)$ denotes a sheaf $L$ with the endomorphism $p^{-r}\phi$.

We use the convention of Nakkajima about signs. (\cite{Nak}, Conventions)

For a $\mathbb{Z}_{(p)}$-algebra $R$, the ring of Witt vectors of any length $W_m(R)$ has a canonical pd-structure on the ideal $I={}^V W_m(R)$ given by
\[
	\gamma_n({}^V\xi)=\frac{p^{n-1}}{n!}{}^{V}(\xi^n),
	\xi \in W_{m-1}(R), n\ge 1.
\]
We always consider this pd-structure on the ring of Witt vectors.

For a $\mathbb{Z}_{(p)}$-scheme $S$ and an $S$-scheme $X$, $\{ W_m\Omega_{X/S}^\bullet \}_{m\in \bN}$ denotes the de Rham-Witt complex constructed in \cite{LZDRW} \S 1.3.

For a complex $(E^\bullet,d^\bullet)$ and for an integer $n$,
$(E^\bullet \{ n \},d^\bullet \{ n \})$ denotes the following complex:
$(E^\bullet \{ n \})^q := E^{q+n}$ with the boundary morphism $d^\bullet \{ n \}:=d^{\bullet+n}$.

\section{Preliminaries}
\label{sect:preliminaries}

\subsection{Logarithmic geometry}

In this paper, we use freely the terminologies concerning logarithmic geometry in the sense of Fontaine-Illusie-Kato.
The basic reference is \cite{KatoLog}.
All log schemes are assumed to be fine and separated and defined over $\mathbb{Z}_{(p)}$.
If $X$ is a log scheme, we denote by $\mathring{X}$ the underlying scheme of $X$.

\begin{defn}
(1) A pre-log ring is a triple $(A,P,\alpha)$ consisting of a commutative ring $A$,
a commutative fine monoid $P$ and a morphism of monoids $P\to A$
where $A$ is regarded as a monoid by its multiplicative structure.
We usually suppress $\alpha$ in the notation.
We denote by $\{*\}$ the trivial monoid.

(2) If $(A,P)$ is a pre-log ring,
$\Spec (A,P)$ is the log scheme whose underlying scheme is $X = \Spec A$
with the log structure associated to the pre-log structure $P \to \mathcal{O}_X$
induced by the structure map $\alpha: P \to A$.

(3) We say $(Y,\mathcal{N})$ is a log scheme over a pre-log ring $(A,P)$
to mean that $(Y,\mathcal{N})$ comes equipped with a morphism of log schemes $(Y,\mathcal{N}) \to \Spec (A,P)$.
\end{defn}

\begin{defn}
A morphism $(A,P) \to (B,Q)$ of pre-log rings is said to be log smooth (resp. log \'etale)
if the kernel and the torsion part of the cokernel
(resp. the kernel and the cokernel) of $P^{\gp}\to Q^{\gp}$
are finite groups of orders invertible on $B$
and the induced morphism $A\otimes_{\mathbb{Z}[P]}\mathbb{Z}[Q]\to B$ is an \'etale ring map.
\end{defn}

We recall the toroidal characterization of log smoothness
(\cite{KatoLog} (3.5), \cite{KatoLogSmoothDeform} Theorem 4.1):

\begin{thm}
\label{thm:toroidal-characterization-log-smoothness}
Let $f:(X,\mathcal{M})\to (Y,\mathcal{N})$ be a morphism of fine log schemes
and $Q \to \mathcal{N}$ a chart of $\mathcal{N}$.
Then the following conditions are equivalent.

(1) $f$ is log smooth (resp. log \'{e}tale).

(2) There exists \'{e}tale locally a chart
$(P \to \mathcal{M}, Q \to \mathcal{N}, Q \to P)$ of $f$
 extending $Q\to \mathcal{N}$ such that

(a) The kernel and the torsion part of the cokernel
(resp. the kernel and the cokernel)
of $Q^{\gp} \to P^{\gp}$ are finite groups of orders invertible on $X$.

(b) The induced map $X \to Y\times_{\Spec \mathbb{Z}[Q]} \Spec \mathbb{Z}[P]$
of schemes is \'{e}tale (in the usual sense).
\end{thm}
We see if $(A,P)\to (B,Q)$ is a log smooth (resp. log \'etale) morphism of pre-log rings,
the induced map $\Spec (B,Q)\to \Spec (A,P)$ is a log smooth (resp. log \'etale) morphism of log schemes.

\begin{defn}
\label{defn:semistable}
(1) Let $f:X\to S$ be a smooth morphism of schemes and $D\subset X$ a reduced Cartier divisor.
Let $j:U:=X\setminus D\to X$ be the natural open immersion.
We call $D$ a simple normal crossing divisor (SNCD)
(resp. a normal crossing divisor (NCD))
if, for any point of $z$ of $D$, there exist a Zariski open neighbourhood $V$ of $z$ in $X$
(resp. an \'{e}tale morphism $V\to X$ such that the image of $V$ contains $z$)
and the following cartesian diagram of schemes
\[
	\xymatrix{
	D\times_X V \ar[r]^{\subset} \ar[d]
	&
	V \ar[d]^{g}
	\\
	\Spec(\mathcal{O}_S[T_1,\ldots,T_n]/(T_1\cdots T_d)) \ar[r]
	&
	\Spec(\mathcal{O}_S[T_1,\ldots,T_n]),
	}
\]
where $g$ is an \'{e}tale morphism.
Let $\mathcal{M}_X:=j^{\text{log}}_*(\mathcal{O}_U^\times)$ be the direct image of the trivial log structure on $U$.
The log scheme $(X,\mathcal{M}_X)$ is log smooth if $D$ is a NCD. By abuse of notation, we write $(X,D)$ instead of $(X,\mathcal{M}_X)$.

(2) Let $S=(\mathring{S},\bN)$ be a log scheme whose log structure is associated to a homomorphism
$\bN \to \Gamma(\mathring{S},\mathcal{O}_{\mathring{S}}); 1\mapsto 0$.
The map $\bN^d\to \mathcal{O}_{\mathring{S}}[T_1,\ldots,T_n]/(T_1 \cdots T_d)$
given by $e_i\mapsto T_i$ defines a fine log scheme
$\Spec(\mathcal{O}_{\mathring{S}}[T_1,\ldots,T_n]/(T_1\cdots T_d),\bN^d)$.

A fine log $S$-scheme $Y$ is called semistable (resp. strictly semistable)
if \'{e}tale locally (resp. Zariski locally) on $Y$,
the structure morphism $f :Y \to S$ has a decomposition
\[
	Y
	\xrightarrow{u}
	\Spec(\mathcal{O}_{\mathring{S}}[T_1,\ldots,T_n]/(T_1\cdots T_d),\bN^d)
	\xrightarrow{\delta}
	S
\]
with $u$ exact and \'{e}tale (in usual sense), $1 \le d \le n$,
and $\delta$ is induced by the diagonal map $\bN \to \bN^d$.
A semistable log $S$-scheme is log smooth and integral over $S$.
\end{defn}

\subsection{Witt scheme}
Let $X$ be a scheme such that $p$ is nilpotent in $X$ and $m$ be a positive integer.
\begin{defn}
(1) The Witt scheme $W_m(X)$ is the scheme $(|X|,W_m(\mathcal{O}_X))$
whose underlying topological space is that of $X$
and whose structure sheaf $W_m(\mathcal{O}_X)$ associates to an affine open subset $U=\Spec R \subset X$ the ring of Witt vectors $W_m(R)$.
We identify the underlying topological spaces of $X$ and $W_m(X)$.
See Appendix A.1 of \cite{LZDRW}.

(2) Let $\alpha : \mathcal{M}\to \mathcal{O}_X$ be a log structure of $X$.
Then the log Witt scheme of the log scheme $(X,\mathcal{M},\alpha)$ is the log scheme
$(W_m(X),W_m(\mathcal{M}),W_m(\alpha))$ whose underlying scheme $W_m(X)$ is
the Witt scheme of $X$, and whose sheaf of monoids $W_m(\mathcal{M})$ is defined by
$\mathcal{M} \oplus \ker(W_m(\mathcal{O}_X)^\times  \to \mathcal{O}_X^\times),$
and whose structure morphism $W_m(\alpha):W_m(\mathcal{M})\to W_m(\mathcal{O}_X)$ is induced by
$\mathcal{M} \to W_m(\mathcal{O}_X), q\mapsto [\alpha (q)]$,
where $[\alpha (q)]$ is the Teichm\"uller lift of $\alpha (q)$.
We sometimes write $W_m(X,\mathcal{M})$ instead of
$(W_m(X),W_m(\mathcal{M}),W_m(\alpha))$.
\end{defn}

\begin{defn}
Let $(A,P,\alpha)$ be a pre-log ring and $m$ a positive integer.

The Witt pre-log ring of $(A,P)$ is the pre-log ring $(W_m(A),P)$ where
the structure morphism is given by $P \to W_m(A), q \mapsto [\alpha (q)]$.
We denote by $W_m(A,P)$ this pre-log ring.
We see $\Spec W_m(A,P)$ is nothing but the Witt scheme of $\Spec (A,P)$.
We also define a pre-log ring $W(A,P)=(W(A),P)$ in a similar way.
\end{defn}

\subsection{Crystalline cohomology over non-adic base}
We give the definitions of crystalline cohomology
and log crystalline cohomology over non-adic base.
Let $A$ be a $\mathbb{Z}_{(p)}$-algebra,
$I_1 \subset A$ an ideal of $A$ equipped with a pd-structure compatible
with the canonical pd-structure on $p\mathbb{Z}_{(p)}\subset\mathbb{Z}_{(p)}$. Let
\[
	A\supset I_1 \supset I_2 \supset \cdots \supset I_m \supset I_{m+1} \supset \cdots
\]
be a decreasing filtration of sub pd-ideals satisfying the following condition:
For all $n \ge m$, there exists a positive integer $a$ (it may depend on $n$ and $m$)
such that $I_m^a \subset I_n$.

Set $A_m := A/I_m$ and we assume that $A$ is complete for the topology defined by $\{I_m\}_{m\in \bN}$, that is, $A\simeq \varprojlim_m A_m$.
We also assume that  $A_1$ is Noetherian and that $p$ is nilpotent in $A_1$.

A basic example is the Witt vector $A=W(R)$ of a Noetherian $\mathbb{Z}_{(p)}$-algebra $R$ in which $p$ is nilpotent
and the ideals $I_m={}^{V^m}W(R)$ equipped with the canonical pd-structure.

Let $X$ be a proper smooth scheme over $A_1$.
We have a canonical morphism of crystalline topoi
$i_{mn}:(X/A_m)_{\crys}\to (X/A_n)_{\crys}$ for $n\ge m$.
We say the system $\mathcal{E}=\{\mathcal{E}_m \}_m$ is a compatible system of locally free crystals of finite rank
if for each $m$, $\mathcal{E}_m$ is a locally free crystal of finite rank on the crystalline site $\text{Crys}(X/A_m)$
and for each $n \ge m$, $i_{mn}^*\mathcal{E}_n\simeq \mathcal{E}_m$.

\begin{defn}
\label{defn:strictly-perfect-cpx}
Let $R$ be a ring and $D(R)$ be the derived category of the category of complexes of $R$-modules.

(1) We call $K^\bullet\in D(R)$ is perfect
if $K^\bullet$ is quasi-isomorphic to a bounded above complex of finite free $R$ modules
and it has finite tor dimension.
This is equivalent to the condition that $K^\bullet$ is quasi-isomorphic to
a bounded complex of finite projective $R$-modules.

(2) We call $K^\bullet\in D(R)$ is strictly perfect
if $K^\bullet$ is quasi-isomorphic to a bounded complex of finite free $R$-modules.
\end{defn}

\begin{lem}
Suppose given $K_m\in D(A_m)$ and a map $K_{m+1}\to K_m$ in $D(A_{m+1})$ for each $m\ge 1$. We assume

(1) $K_1$ is a perfect object, and

(2) The maps induce isomorphisms
\[
	K_{m+1}\otimes_{A_{m+1}}^{\mathbb{L}} A_m
	\to
	K_m.
\]
Then $K=\mathbb{R}\varprojlim K_m$ is a perfect object of $D(A)$
and $K\otimes_A^{\mathbb{L}}A_m\to K_m$ is an isomorphism for all $m$.
\end{lem}

\begin{proof}
Take $P_1$ a bounded complex of finite projective $R$-modules such that
$K_1$ is quasi-isomoprhic to $P_1$.
We know $I_m / I_{m+1}$ is nilpotent for all $m$.
By \cite{Stacks} More On Algebra, Tag 09AR, we find inductively for all $m$
a bounded complex of finite projective $R$-modules $P_m$ such that
there is an isomorphism of complexes $P_{m+1}\otimes_{A_{m+1}}A_m \simeq P_m$.
In this way $P_{m+1}$ has the same amplitude as $P_m$ and each term of
pro-complex $P_\bullet$ satisfies the Mittag-Leffler condition.
Hence $K:=\mathbb{R}\varprojlim K_m=\varprojlim P_m$ and it satisfies the conditions from the lemma.
\end{proof}

\begin{prop}
\label{prop:derivedlimit}
(cf. \cite{Stacks} Crystalline Cohomology, Tag 07MX)

There exists a perfect object $\mathbb{R}\Gamma_{\crys}(X/A,\mathcal{E})$ in $D(A)$ such that
\[
	\mathbb{R}\Gamma_{\crys}(X/A,\mathcal{E})\otimes_A^{\mathbb{L}}A_m
	\simeq
	\mathbb{R}\Gamma_{\crys}(X/A_m,\mathcal{E}_m).
\]
\end{prop}

\begin{proof}
Base change theorem (\cite{BO} Theorem 7.8) gives
\[
	\mathbb{R}\Gamma_{\crys}(X/A_{m+1},\mathcal{E}_{m+1})\otimes_{A_{m+1}}^{\mathbb{L}}A_m
	\simeq
	\mathbb{R}\Gamma_{\crys}(X/A_m,\mathcal{E}_m)
\]
for every $n$. By this result and the comparison theorem (\cite{BO} Theorem 7.1) we obtain
\[
	\mathbb{R}\Gamma_{\crys}(X/A_1,\mathcal{E}_1)
	\simeq
	\mathbb{R}\Gamma_{\Zar}(X,(\mathcal{E}_1)_X \otimes_{\mathcal{O}_X} \Omega_{X/A_1}^\bullet).
\]
We show first that $\mathbb{R}\Gamma_{\Zar}(X,(\mathcal{E}_1)_X \otimes \Omega_{X/A_1}^\bullet)$ is perfect.
By using the stupid filtration on $(\mathcal{E}_1)_X \otimes \Omega_{X/A_1}^\bullet$,
we are reduced to showing that $\mathbb{R}\Gamma_{\Zar}(X,(\mathcal{E}_1)_X \otimes \Omega_{X/A_1}^q)$ is perfect
by \cite{Stacks} More On Algebra, Tag 066R.
It follows from the fact that $(\mathcal{E}_1)_X \otimes \Omega_{X/A_1}^q$ is
a locally free sheaf of finite type and $X$ is proper over a Noetherian ring $A_1$.
Thus we have a perfect object $D:=\mathbb{R}\varprojlim_m \mathbb{R} \Gamma_{\crys}(X/A_m,\mathcal{E}_m)$
and it has the property $D\otimes_A^{\mathbb{L}}A_m \simeq \mathbb{R}\Gamma_{\crys}(X/A_m,\mathcal{E}_m)$.
\end{proof}

We have the same proposition for a proper log smooth integral scheme $X$　(in the sense of \cite{KatoLog} Definition 4.3) over $A_1$, with
$A$ (hence $A_1$) endowed with a log structure.
(Use \cite{KatoLog} Theorem (6.10)).

\begin{defn}
Assume that $R$ is a Noetherian $\mathbb{Z}_{(p)}$-algebra in which $p$ is nilpotent.

If $X$ is a proper smooth scheme over $R$ and $\mathcal{E}=\{\mathcal{E}_m \}_m$ is a compatible system of locally free crystals of $X$,
we define crystalline cohomology of $X$ with coefficients $\mathcal{E}$ over $W(R)$ by $H^*_{\crys}(X/W(R),\mathcal{E}):=\mathbb{R}^*\Gamma_{\crys}(X/W(R),\mathcal{E})$.
We define crystalline cohomology of $X$ by $H^*_{\crys}(X/W(R)):=H^*_{\crys}(X/W(R),\mathcal{O}_{X/W(R)})$.

If $(R,P)$ is a pre-log ring and $X$ is a log scheme proper log smooth integral over $(R,P)$,
we define the log crystalline cohomology $H^*_{\logcrys}(X/W(R,P))$ in the similar fashion.
\end{defn}

\begin{thm} \label{thm:infinitecomparison}
Let $R$ be a Noetherian $\mathbb{Z}_{(p)}$-algebra in which $p$ is nilpotent, and $X$ be a proper smooth scheme over $R$.
Then we have a canonical isomorphism
\[
	H_{\crys}^* (X/W(R))
	\to
	\mathbb{H}_{\Zar}^* (X,W\Omega_{X/R}^\bullet).
\]
\end{thm}
\begin{proof}
Let $u_m : (X/W_m(R))_{\crys}\to X_{\Zar}$ be the canonical morphism of topoi.
Using the simplicial method (cf. \cite{LZDRW} \S 3.2) we may assume $X$ is embedded in a smooth affine scheme $Y$ which admits a Frobenius lift $Y_m$.
Let $\overline{Y}_m$ be the pd-envelope of the closed immersion $X \hookrightarrow Y_m$.
From the naturality of the comparison morphism of crystalline cohomology and de Rham cohomology (\cite{BO} Theorem 7.1, Remark 7.5)
we have a commutative diagram
\[
\xymatrix{
	\mathbb{R} u_{m*}\mathcal{O}_{X/W_m(R)} \ar[r] \ar[d]
	&
	\Omega_{\overline{Y}_m/W_m(R)}^\bullet \ar[d]
	\\
	\mathbb{R} u_{m-1*}\mathcal{O}_{X/W_{m-1}(R)} \ar[r]
	&
	\Omega_{\overline{Y}_{m-1}/W_{m-1}(R)}^\bullet .
}
\]
Moreover, the Frobenius lift $Y_m$ makes the following diagram commutative
\[
\xymatrix{
	\mathcal{O}_{Y_m} \ar[r]^-{\Delta_m^*} \ar[d]
	&
	W_m(\mathcal{O}_Y) \ar[d]
	\\
	\mathcal{O}_{Y_{m-1}} \ar[r]^-{\Delta_{m-1}^*}
	&
	W_{m-1}(\mathcal{O}_Y),
}
\]
where $\Delta_n^*\ (n=m-1,m)$ is the map induced by the morphism
$\Delta_n:W_n(Y)\to Y_n$ in \cite{LZDRW} (3.5).
Hence we see comparison isomorphisms
$\mathbb{R} u_{m*}\mathcal{O}_{X/W_m(R)}\to W_m\Omega_{X/R}^\bullet$
are compatible with restriction
and then obtain the canonical isomorphism
$\mathbb{R}\Gamma_{\crys}(X/W_*(R))\to \mathbb{R}\Gamma_{\Zar}(X,W_*\Omega_{X/R}^\bullet)$ in $D(\bN,(W_m(R)))$.
Apply $\mathbb{R}\varprojlim$ to this,
then we get the isomorphism $\mathbb{R}\Gamma_{\crys}(X/W(R))\to \mathbb{R}\Gamma_{\Zar}(X,W\Omega_{X/R}^\bullet)$
by Proposition \ref{prop:derivedlimit} and \cite{LZDRW} Proposition 1.13.
\end{proof}

\section{Log \texorpdfstring{$F$}{F}-\texorpdfstring{$V$}{V}-procomplexes and log de Rham-Witt complex}
\label{sect:log-F-V-procpx-and-dRWcpx}

\subsection{Log pd-derivations}
\begin{defn}
(\cite{Og} Definition 1.1.9)

Let $\theta$ be a morphism of pre-log rings
\[
\xymatrix{
	Q \ar[r]^\beta
	&
	S
	\\
	P \ar[r]^\alpha \ar[u]^{\theta^\flat}
	&
	R \ar[u]_{\theta^\sharp}
}
\]
and let $M$ be a $S$-module.
Then a log-derivation of $(S,Q)/(R,P)$ with values in $M$ is a pair $(D,\delta)$,
where $D : S \to M$ is a derivation of $S/R$ with values in $M$
and $\delta : Q \to M$ is a homomorphism of monoids
such that the following conditions are satisfied:

(1) For every $q\in Q, D(\beta(q)) = \beta(q)\delta(q).$

(2) For every $p\in P, \delta(\theta^\flat(p))=0.$
\end{defn}

\begin{defn}
Let $R$ be a ring and $S$ an $R$-algebra. Let $I \subset S$ be an ideal equipped with a pd-structure $\{\gamma_n \}_{n\in \bN}$. Let $M$ be a $S$-module.

(1) (\cite{LZDRW} Definition 1.1)
A pd-derivation of $S/R$ with values in $M$ is a derivation
$D : S \to M$ of $S/R$ which satisfies
\[
	D (\gamma_n(b)) = \gamma_{n-1}(b)D (b)
\]
for $n \ge 1$ and each $b \in I$.

(2) Let $(R,P) \to (S,Q)$ be a morphism of pre-log rings.
A log derivation $(D : S \to M, \delta : Q \to M)$ is called a log pd-derivation if $D$ is a pd-derivation.
\end{defn}

We denote by $\breve{\mbox{Der}}_{(R,P)}((S,Q),M)$ the set of log pd-derivations.
The functor $M \mapsto \breve{\mbox{Der}}_{(R,P)}((S,Q),M)$ is representable by a universal object
\[
	(d : S \to \breve{\Lambda}^1_{(S,Q)/(R,P)},
	d\log : Q \to \breve{\Lambda}^1_{(S,Q)/(R,P)}),
\]
where the $S$-module $\breve{\Lambda}^1_{(S,Q)/(R,P)}$ is obtained as the quotient module of the log differentials $\Lambda^1_{(S,Q)/(R,P)}$
by the submodule generated by all elements $d (\gamma_n(b))-\gamma_{n-1}(b)db$ for $b\in I,n\ge 1$.

\begin{defn}
(1) Let $R\to S$ be a morphism of rings. A differential graded $S/R$-algebra is a unitary graded $S$-algebra
\[
	E^\bullet=\bigoplus_{i\ge 0}E^i
\]
equipped with an $R$-linear differential $d:E^\bullet \to E^\bullet$ such that the following relations hold:
\begin{align*}
	\omega \eta &= (-1)^{ij} \eta \omega,\ \omega \in E^i, \eta \in E^j, \\
	\omega \cdot \omega &= 0,\ \text{with}\ i \ \text{odd}, \\
	d(\omega \eta) &= d(\omega) \eta + (-1)^i \omega d(\eta),\ \omega \in E^i, \eta \in E^j,\\
	d^2 &= 0.
\end{align*}

(2) Let $(R,P)\to (S,Q)$ be a morphism of pre-log rings.
A log differential graded $(S,Q)/(R,P)$-algebra is a triple $(E^\bullet,d ,\partial)$,
where $(E^\bullet, d)$ is a differential graded $S/R$-algebra and $\partial : Q \to E^1$ is a morphism of monoids,
such that $(d : S\to E^0 \to E^1, \partial)$ is a log derivation and $d\partial = 0$.

A morphism of log differential graded $(S,Q)/(R,P)$-algebras
$f:(E^\bullet ,d, \partial) \to(E^{\prime \bullet} ,d^\prime, \partial^\prime)$ is a morphism of differential graded $S/R$-algebras
$f:(E^\bullet ,d) \to (E^{\prime \bullet} ,d^\prime)$ that satisfies $\partial^\prime = f \circ \partial$.

(3) A log pd-differential graded $(S,Q)/(R,P)$-algebra is a log differential graded $(S,Q)/(R,P)$-algebra $(E^\bullet ,d, \partial)$ such that $d : S \to E^0 \to E^1$ is a pd-derivation.
\end{defn}

 \subsection{Frobenius action on log pd-derivations}
We consider a continuous (i.e. it factors through $W_l(S)$ for some $l > 0$) $W(R)$-linear pd-derivation $D :W(S) \to M$ to
a discrete $W(S)$-module (i.e. it is obtained by restriction of scalars $W(S) \to W_l(S)$ for some $l > 0$).
By \cite{LZDRW} \S 1.1, we have a map $^F D : W(S)  \to M$
given by $\xi = [x]+{}^{V}\rho \mapsto [x^{(p-1)}]D ([x])+D(\rho), x \in S$.
Then $^F D : W(S) \to M_{[F]}$ is a continuous $W(R)$-linear pd-derivation. We extend this to pre-log rings.

Let $(S,Q)$ be a pre-log ring and $(D ,\delta):W(S,Q) \to M$ a $W(R)$-linear log pd-derivation. Then the pair $(^F D , \delta )$ is also a log pd-derivation. In fact, for $q\in Q$, we have
\begin{align*}
	^F D (W(\alpha)(q))
	&= {}^FD ([\alpha(q)])\\
	&= [\alpha(q)^{(p-1)}]D ([\alpha(q)])\\
	&= [\alpha(q)^{(p-1)}]\cdot [\alpha(q)] \delta(q)\\
	&= {}^F(W(\alpha)(q))\delta(q).
\end{align*}

By the universal property of the logarithmic differential sheaf,
we obtain a morphism $F : \breve{\Lambda}^1_{W_{m+1}(S,Q)/W_{m+1}(R,P)}\to (\breve{\Lambda}^1_{W_m(S,Q)/W_m(R,P)})_{[F]}$ from
\[
	(d, d\log ) : (W_m(S),Q)
	\to
	\breve{\Lambda}^1_{W_m(S,Q)/W_m(R,P)}
\]
and it satisfies $F\circ (d,d\log)=({}^F d,d\log)$.
By definition, we obtain
\begin{align*}
	^F(d\xi)&=(^Fd)(\xi), \ \xi \in W_{m+1}(S),\\
	^F(d\log m)&=d\log m, \ m\in Q,\\
	^Fd([x])&=[x]^{p-1}d[x], \ x\in S,\\
	d(^F\xi )&=p{}^Fd\xi, \ \xi \in W_{m+1}(S)\\
	^Fd(^V\xi) &=d\xi,\ \xi \in W_m(S).
\end{align*}

\subsection{Log \texorpdfstring{$F$}{F}-\texorpdfstring{$V$}{V}-procomplexes}
Let $(R,P) \to (S,Q)$ be a morphism of pre-log rings.

\begin{defn}
A log $F$-$V$-procomplex over $(R,P)$-algebra $(S,Q)$ is a projective system
\[
	\{ E_m^\bullet=(E_m^\bullet,D_m, \partial_m), \pi_m : E_{m+1}^\bullet \to E_m^\bullet\}_{m\in\bN}
\]
of a log differential graded $W_m(S,Q)/W_m(R,P)$-algebra $(E_m^\bullet,D_m ,\partial _m)$,
\[
	\ldots
	\to E_{m+1}^\bullet
	\xrightarrow{\pi_m} E_{n}^\bullet
	\to \ldots
	\to E_1^\bullet
	\to E_0^\bullet=0.
\]
Moreover, $\{E_m^\bullet \}$ is equipped with two sets of homomorphisms of graded abelian groups,
\[
F:E_{m+1}^\bullet \to E_m^\bullet,
V:E_m^\bullet\to E_{m+1}^\bullet,
m\ge 0,
\]
and the following properties hold.

(i) $\partial_m$ are compatible with $\pi_m$,
i.e., $\partial_m=\pi_m \circ \partial_{m+1}$ for any $m \ge 0$.

(ii) The morphisms $W_m(S) \to E_m^0$ are compatible with $F$ and $V$ for any $m \ge 0$.

(iii) The restriction maps $\pi_m : E_m^\bullet \to E_{m-1}^\bullet$ are compatible with $F$ and $V$ for any $m \ge 1$.

(iv) Let $E_{m,[F]}^\bullet$ be the graded $W_{m+1}(S)$-algebra obtained via restriction of scalars $F:W_{m+1}(S) \to W_m(S)$. Then $F$ induces a homomorphism of graded $W_{m+1}(S)$-algebras,
\[
	F:E_{m+1}^\bullet
	\to
	E_{m,[F]}^\bullet.
\]

(v) We have
\begin{align*}
	^{FV}\omega &= p\omega, \ \omega \in E_m^\bullet,\ n\ge 0, \\
	^{F}D_{m+1}{}^{V}\omega &= D_m\omega, \\
	^{F}D_{m+1}[x] &= [x^{p-1}]D_m[x], \ x\in S, \\
	^{V}(\omega ^{F}\eta) &= (^{V}\omega)\eta, \ \eta \in E_{m+1}^\bullet,\\
	^{F}\partial_{m+1} (q) &= \partial_m (q), \ q\in Q.
\end{align*}
\end{defn}

A morphism of log $F$-$V$-procomplexes $f:\{E_m^\bullet = (E_m^\bullet,D_m,\partial_m),\pi\} \to \{E_m^{\prime \bullet} = (E_m^{\prime \bullet},D_m^\prime,\partial_m^\prime),\pi^\prime\}$ is a morphism of
pro-log differential graded $W_*(S,Q)/W_*(R,P)$-algebras $f=\{f_m:E_m^\bullet \to E_m^{\prime \bullet} \}$ that is compatible with $F$ and $V$.

\subsection{Construction of log de Rham-Witt complex}
\label{subsect:odrwc-construction}

Let $R$ be a $\mathbb{Z}_{(p)}$-algebra.
For a pre-log ring $(S,Q)$ over $(R,P)$, we construct the log de Rham-Witt complex
\[
	\{ W_m\Lambda^{\bullet} \}_{m\in \bN}
	= \{W_m\Lambda^{\bullet}_{(S,Q)/(R,P)} \}_{m\in \bN}
\]
as the universal log $F$-$V$-procomplex by induction on $m$.
We set
\[
W_1\Lambda^{\bullet}
:=\Lambda^{\bullet}_{(S,Q)/(R,P)}
=\breve{\Lambda}^{\bullet}_{W_1(S,Q)/W_1(R,P)}.
\]
To define $W_{m+1}\Lambda^\bullet$ we assume that we have

\begin{itemize}
\item $\{W_n\Lambda^{\bullet}\}_{n\le m}$, a system of log pd-differential graded $W_n(S,Q)/W_n(R,P)$-algebras $W_n\Lambda^{\bullet}$,
\item $\breve{\Lambda}^{\bullet}_{W_n(S,Q)/W_n(R,P)} \to W_n\Lambda^{\bullet}$,
for $n\le m$, surjective morphisms of log differential graded algebras
which are compatible with the restriction maps and with $F$,
\item $V:W_n\Lambda^{\bullet} \to W_{n+1}\Lambda^{\bullet}$, additive maps for $1\le n < m$,
\item $W_n\Lambda^0=W_n(S)$ for $1\le n < m$,
\item $^{FV}\omega = p\omega,\ ^Fd^V\omega = \omega,\ ^V(\omega ^F\eta )= {}^V\omega \cdot \eta$ \ for $\omega \in W_n\Lambda^\bullet,\  \eta \in W_{n+1}\Lambda^\bullet$, $1\le n < m$.
\end{itemize}
We define an ideal $I\subset \breve{\Lambda}^\bullet_{W_{m+1}(S,Q)/W_{m+1}(R,P)}$ as follows. Consider all relations of the form
\[
	\sum_{l=1}^M \xi^{(l)} \cdot d\log q_1^{(l)}\cdots d\log q_{r_l}^{(l)}
	\cdot d\eta_{r_l+1}^{(l)}\cdots d\eta_{i}^{(l)}=0
\]
in $W_m\Lambda^\bullet$, where $\xi^{(l)},\eta_k^{(l)}\in W_m(S),q_k^{(l)}\in Q$.
Then $I\subset \breve{\Lambda}^\bullet_{W_{m+1}(S,Q)/W_{m+1}(R,P)}$ is defined to be the ideal generated by the elements
\begin{align*}
	\sum_{l=1}^M {}^V\xi^{(l)}
	\cdot d\log q_1^{(l)}\cdots d\log q_{r_l}^{(l)}
	\cdot d^V\eta_{r_l+1}^{(l)}\cdots d^V\eta_{i}^{(l)},\\
	\sum_{l=1}^M d^V\xi^{(l)}
	\cdot d\log q_1^{(l)}\cdots d\log q_{r_l}^{(l)}
	\cdot d^V\eta_{r_l+1}^{(l)}\cdots d^V\eta_{i}^{(l)}.
\end{align*}
Then $I$ is stable by $d$. Moreover,
\[
	F:\breve{\Lambda}_{W_{m+1}(S,Q)/W_{m+1}(R,P)}^{\bullet}
	\xrightarrow{F}
	\breve{\Lambda}_{W_m(S,Q)/W_m(R,P)}^{\bullet}
	\to
	W_m\Lambda^\bullet
\]
annihilates $I$ since we have
\begin{align*}
	^{FV}\xi &=p\xi \in \breve{\Lambda}_{W_m(S,Q)/W_m(R,P)}^0=W_m(S),\
	\xi \in W_m(S)\\
	^Fd^V\eta &= d\eta \in \breve{\Lambda}_{W_m(S,Q)/W_m(R,P)}^1,\
	\eta \in W_m(S)\\
	^F d\log q &= d\log q \in \breve{\Lambda}_{W_m(S,Q)/W_m(R,P)}^1, \
	q\in Q.
\end{align*}

Therefore $F$ induces
\[
	F:\bar{\Lambda}_{m+1}^\bullet :=
	\breve{\Lambda}_{W_{m+1}(S,Q)/W_{m+1}(R,P)}^{\bullet}/I
	\to
	W_m\Lambda^\bullet.
\]

On the other hand, we have a well-defined map
\begin{align*}
	V:W_m\Lambda^\bullet &\to \bar{\Lambda}_{m+1}^\bullet,\\
	\xi \cdot d\log q_1 \cdots d\log q_r
	\cdot d\eta_{r+1} \cdots d\eta_{i}
	&\mapsto {}^V\xi \cdot
	d\log q_1 \cdots d\log q_r
	\cdot d^V\eta_{r+1} \cdots d^V\eta_{i}.
\end{align*}

We have $^Fd^V\omega = d\omega$.
Let $J$ be the ideal of $\bar{\Lambda}_{m+1}^\bullet$ generated by the elements
\begin{align*}
	^V(\omega \cdot {}^F\eta)&- {}^V\omega \cdot \eta,\\
	d(^V(\omega \cdot {}^F\eta)&- {}^V\omega \cdot \eta),
\end{align*}
where $\omega \in W_m\Lambda^\bullet$ and $\eta \in \bar{\Lambda}_{m+1}^\bullet$.
We see that $F$ annihilate $J$.
We set $W_{m+1}\Lambda^\bullet:=\bar{\Lambda}_{m+1}^\bullet / J$.
Then we have maps
\begin{align*}
	F&:W_{m+1}\Lambda^\bullet
	\to
	W_m\Lambda^\bullet,
	\\
	V&:W_m\Lambda^\bullet
	\to
	W_{m+1}\Lambda^\bullet.
\end{align*}

We can see that all requirements of the definition of log $F$-$V$-procomplexes are satisfied.
We set $W\Lambda^\bullet := \varprojlim_m W_m\Lambda^\bullet$.
By the construction, the log de Rham-Witt complex we made
is a natural extension of the de Rham-Witt complex constructed in \cite{LZDRW}:
i.e., $W_m\Lambda_{(S,\{*\})/(R,\{*\})}^\bullet \simeq W_m\Omega_{S/R}^\bullet$.

The following proposition is clear from the definition.

\begin{prop}
\label{prop:universal-property}
(cf. \cite{LZDRW} Proposition 1.6)
Let $\{E_m^\bullet,D_m,\partial_m \}_{m \in \bN}$ be a log $F$-$V$-procomplex
over $(S,Q)/(R,P)$.
Then there is a unique morphism of log $F$-$V$-procomplexes
\[
	\{ W_m\Lambda_{(S,Q)/(R,P)}^\bullet \}
	\to
	\{ E_m^\bullet \}
\]
over $(S,Q)/(R,P)$.
\end{prop}

\subsection{Standard Filtration}
The differential graded ideals
\[
	\Fil^sE_m^i = V^sE_{m-s}^i+dV^sE_{m-s}^{i-1}\subset E_m^i
\]
gives a filtration of a log $F$-$V$-procomplex $\{E_m^\bullet \}$
and it is called the standard filtration.
Since restriction maps and $F, V$ are compatible, we find
\begin{align*}
\left.
\begin{array}{l}
	\pi(\Fil^sE_m^i) \subset \Fil^sE_{m-1}^i,\\
	F(\Fil^sE_m^i) \subset \Fil^{s-1}E_{m-1}^i,\\
	V(\Fil^sE_m^i) \subset \Fil^{s+1}E_{m+1}^i,\\
	d(\Fil^sE_m^i) \subset \Fil^sE_m^{i+1}.
\end{array}
\right\} \cdots (*)
\end{align*}

\begin{prop}
\label{prop:fil-exact-sequence}
(cf. \cite{HMKtheory} 3.2.4)
Let $(R,P) \to (S,Q)$ be a morphism of pre-log rings and $m, s$ positive integers satisfying $m\ge s$.
We set $W_m\Lambda^\bullet := W_m\Lambda_{(S,Q)/(R,P)}^\bullet$.
Then we have the following exact sequence:
\[
	0
	\to
	\Fil^sW_m\Lambda^\bullet
	\to
	W_m\Lambda^\bullet
	\xrightarrow{\pi^{m-s}}W_s\Lambda^\bullet
	\to
	0.
\]
\end{prop}

\begin{proof}
For any log $F$-$V$-procomplex $\{ E_m^\bullet \}$,
the composition of the two morphisms $\Fil^sE_m^\bullet \hookrightarrow E_m^\bullet \xrightarrow{\pi^{m-s}} E_s^\bullet$ is zero
since $\pi$ commutes with $F$ and $V$, and $\pi^{m}(E_m^\bullet) \subset E_0^\bullet = 0$.
Therefore $\pi^{m-s}$ induces a morphism
\[
	\pi^{m-s}:
	E_m^\bullet / \Fil^s(E_m^\bullet)
	\to
	E_s^\bullet.
\]

Fix $r := m-s$ and set $E_n^{\prime \bullet} := E_{n+r}^\bullet / \Fil^n E_{n+r}^\bullet$.
Then $\{ E_n^{\prime \bullet} \}$ is a log $F$-$V$-procomplex over $(S,Q)/(R,P)$ by $(*)$.
We show that $\{ W_n\Lambda^{\prime\bullet} \}$ is the universal log $F$-$V$-procomplex.
Since the projection map $\pi$ of $W_n\Lambda^{\bullet}$ is surjective,
we have the canonical surjective morphism
$\{ W_n\Lambda^{\prime\bullet} \} \to \{ W_n\Lambda^{\bullet} \}$ of log $F$-$V$-procomplexes.
The diagram:
\[
\xymatrix{
	\Lambda_{W_{n+r}(S,Q)/W_{n+r}(R,P)}^\bullet \ar@{->>}[d] \ar@{->>}[r]
	&
	W_{n+r}\Lambda^\bullet \ar@{->>}[d]
	\\
	\Lambda_{W_{n}(S,Q)/W_{n}(R,P)}^\bullet \ar[r]
	&
	W_{n}\Lambda^{\prime \bullet}.
}
\]
shows the morphism $\Lambda_{W_{n}(S,Q)/W_{n}(R,P)}^\bullet
\to W_{n}\Lambda^{\prime \bullet}$ is surjective.

Let $\{ E_n^\bullet \}$ be any log $F$-$V$-procomplex over $(S,Q)/(R,P)$.
By the universal property of $\{ W_{n}\Lambda^{\bullet} \}$,
there is a unique morphism $\{ W_{n}\Lambda^{\bullet}\} \to \{ E_n^\bullet \}$ of log $F$-$V$-procomplexes.
We compose it with canonical surjection $\{ W_n\Lambda^{\prime\bullet}\} \to \{ W_n\Lambda^{\bullet} \}$.
Then we get a morphism  $\{ W_{n}\Lambda^{\prime \bullet}\} \to \{ E_n^\bullet \}$.
This is the unique morphism of log $F$-$V$-procomplexes from $\{ W_{n}\Lambda^{\prime \bullet} \}$ to $\{ E_n^\bullet \}$
because  $\Lambda_{W_{n}(S,Q)/W_{n}(R,P)}^\bullet \to W_{n}\Lambda^{\prime \bullet}$ is surjective.
Hence $\{ W_{n}\Lambda^{\prime \bullet} \}$ has the universal property.
\end{proof}

\subsection{Base change for \'{e}tale morphisms}

We establish the \'{e}tale base change property of log de Rham-Witt complexes.
The following propositions can be shown by the same method used in \cite{LZDRW} Proposition 1.7 and 1.9.

\begin{prop}
\label{prop:etale-base-change-property}
Let $R$ be a ring such that $R$ is $F$-finite (in the sense of \cite{LZDRW} Proposition A.2) or $p$ is nilpotent in $R$.
Let $(R,P) \to (S,Q)$ be a morphism of pre-log rings and
$S \to S'$ be an \'etale ring map.
Then the natural morphism
\[
	W_m\Lambda_{(S',Q)/(R,P)}^\bullet
	\to
	W_m(S')\otimes_{W_m(S)}W_m\Lambda_{(S,Q)/(R,P)}^\bullet
\]
is an isomorphism.
\end{prop}

\begin{prop}
Let $(R,P)$ be a pre-log ring such that $R$ is $F$-finite or $p$ is nilpotent in $R$.
Assume we are given an unramified ring homomorphism $R \to R'$
and a morphism $(R',P) \to (S,Q)$ of pre-log rings.
Then we have a natural isomorphism of log $F$-$V$-procomplexes relative to $(S,Q)/(R,P)$:
\[
	\{ W_m\Lambda_{(S,Q)/(R,P)}^\bullet \}
	\to
	\{W_m\Lambda_{(S,Q)/(R',P)}^\bullet \}.
\]
\end{prop}

We define the log de Rham-Witt complex on log schemes.
The following lemma immediately follows from \cite{Og}, Chapter II, Proposition 2.2.1.
\begin{lem}
\label{lem:chart}
Let $\beta: Q\to \mathcal{M}$ be a chart for a sheaf of fine monoids $\mathcal{M}$ on a scheme $X$. Suppose that $\beta$ factors
\[
	Q
	\xrightarrow{\theta}
	Q'
	\xrightarrow{\beta'}
	\mathcal{M},
\]
where $Q'$ is a constant sheaf of a finitely generated monoid. Then, \'etale locally on $X$, $\beta'$ can be factored
\[
	Q'
	\xrightarrow{\theta'}
	Q''\xrightarrow{\beta"}
	\mathcal{M},
\]
where $\beta''$ is a chart for $\mathcal{M}$ and $Q''$ is finitely generated.
\end{lem}

\begin{prop-def}
\label{prop-def:chart-welldefined}
Let $f:(X, \mathcal{M}) \to (Y,\mathcal{N})$ be a morphism of fine log schemes over $\mathbb{Z}_{(p)}$.
We assume that $p$ is nilpotent in $Y$.
We identify the \'etale topology of $W_m(X)$
and that of $X$ (See \cite{Stacks} \'Etale Cohomology, Tag 03SI).
Then there is a unique quasi-coherent sheaf
$W_m\Lambda_{(X, \mathcal{M})/(Y, \mathcal{N})}^\bullet$ on $X_{\et}$
which has the following property:
If there is a commutative diagram
\[
\xymatrix{
	U=\Spec\ S' \ar[r] \ar[d]_{\gamma'}
	&
	V=\Spec\ R' \ar[d]^\gamma
	\\
	X \ar[r]^{f}
	&
	Y,
}
\]
where $\gamma$ and $\gamma'$ are \'etale morphisms and
there is a chart $(Q \to \mathcal{M}|_{U}, P \to \mathcal{N}|_{V}, P \to Q)$ of the morphism $(U,\mathcal{M}|_U) \to (V,\mathcal{N}|_V)$,
then we have a canonical isomorphism
\[
	\Gamma (U,W_m\Lambda_{(X, \mathcal{M})/(Y, \mathcal{N})}^\bullet)
	= W_m\Lambda_{(S',Q)/(R',P)}^\bullet.
\]
\end{prop-def}

\begin{proof}
When $X=\Spec S$ and $Y=\Spec R$ are affine and $f$ has a chart
$(Q\to \mathcal{M}, P\to \mathcal{N}, P\to Q)$, the presheaf
\[
	X_{\text{\'{e}t}} \ni (U'=\Spec S' \to X)
	\mapsto
	W_m\Lambda_{(S',Q)/(R,P)}^\bullet
\]
defines a quasi-coherent sheaf on $X_{\text{\'{e}t}}$
because of the base change property of \'{e}tale morphisms
(Proposition \ref{prop:etale-base-change-property}).
We temporarily denote by $\mathcal{F}_{(P,Q)}$ this sheaf .
We have to show if there exists another chart
$(Q'\to \mathcal{M}, P'\to \mathcal{N}, P'\to Q')$ of $f$,
we have an isomorphism $\mathcal{F}_{(P,Q)}\simeq \mathcal{F}_{(P',Q')}$.
We denote by
\[
	W_m(f):
	(W_m(X)=\Spec W_m(S),W_m(\mathcal{M}))
	\to
	(W_m(Y)=\Spec W_m(R), W_m(\mathcal{N}))
\]
the morphism induced by $f$.
Since $(Q\to \mathcal{M}, P\to \mathcal{N}, P\to Q)$
(resp. $(Q'\to \mathcal{M}, P'\to \mathcal{N}, P'\to Q')$) is a chart of $f$,
we have a canonical chart
\begin{align*}
	&(Q\to W_m(\mathcal{M}), P\to W_m(\mathcal{N}), P\to Q)\\
	(\text{resp.} &(Q'\to W_m(\mathcal{M}), P'\to W_m(\mathcal{N}), P'\to Q'))
\end{align*}
of $W_m(f)$.
Then we have an isomorphism
\[
	\Lambda^{\bullet}_{W_m(S,Q)/W_m(R,P)}
	\simeq
	\Lambda^{\bullet}_{W_m(X,\mathcal{M})/W_m(Y,\mathcal{N})}(W_m(X))
	\simeq
	\Lambda^{\bullet}_{W_m(S,Q')/W_m(R,P')}
\]
by \cite{Og} Corollary 1.1.11.
Hence it induces an isomorphism
\[
	\breve{\Lambda}^{\bullet}_{W_m(S,Q)/W_m(R,P)}
	\simeq
	\breve{\Lambda}^{\bullet}_{W_m(S,Q')/W_m(R,P')}.
\]
Let
\[
	I_{(P,Q)}\subset \breve{\Lambda}^\bullet_{W_{m+1}(S,Q)/W_{m+1}(R,P)}\
	(\text{resp. } I_{(P',Q')}\subset \breve{\Lambda}^\bullet_{W_{m+1}(S,Q')/W_{m+1}(R,P')})
\]
and
\begin{align*}
	J_{(P,Q)} &\subset \breve{\Lambda}_{W_{m+1}(S,Q)/W_{m+1}(R,P)}^{\bullet}/I_{(P,Q)}\\
	(\text{resp. } J_{(P',Q')} &\subset \breve{\Lambda}_{W_{m+1}(S,Q')/W_{m+1}(R,P')}^{\bullet}/I_{(P',Q')})
\end{align*}
be ideals defined in the construction of $W_m\Lambda_{(S,Q)/(R,P)}^\bullet$
(resp. $W_m\Lambda_{(S,Q')/(R,P')}^\bullet$). See \S \ref{subsect:odrwc-construction}.

First we assume that there is a morphism of charts
\[
	(Q\to \mathcal{M}, P\to \mathcal{N}, P\to Q)
	\to
	(Q'\to \mathcal{M}, P'\to \mathcal{N}, P'\to Q').
\]
This morphism induces a canonical map
$\mathcal{F}_{(P,Q)}\to \mathcal{F}_{(P',Q')}$.
We show that it is an isomorphism.
Let $\alpha : Q_X\to \mathcal{O}_X$ and $\alpha' : Q'_X\to \mathcal{O}_X$ be the structure morphisms.
Let $\beta : Q\to Q'$ be a morphism of monoids induced by the morphism of charts.
For any geometric point $\bar{x} \rightarrow X$, we have isomorphisms
\[
	Q/\alpha^{-1}(\mathcal{O}_{X,\bar{x}}^*)
	\xrightarrow[\sim]{\bar{\beta}_{\bar{x}}}
	Q'/\alpha'^{-1}(\mathcal{O}_{X,\bar{x}}^*)
	\xrightarrow{\sim}
	\mathcal{M}_x/\mathcal{O}_{X,\bar{x}}^*.
\]
Since $Q'$ is finitely generated,
by replacing $X$ with some \'etale neighbourhood of $\bar{x}$,
we can assume that for any $q' \in Q'$
there exists $q\in Q, s , s' \in \alpha'^{-1}(\mathcal{O}_{X}(X)^*)$
such that $q'\cdot s' = \beta(q)\cdot s$.
We see
\begin{align*}
	d\log q' &=d\log(q' \cdot s')-d\log s'\\
	&=d\log(\beta(q)\cdot s)-d\log s'\\
	&=d\log \beta(q) + d\log s - d\log s'\\
	&=d\log \beta(q) + \alpha'(s)^{-1}ds - \alpha'(s')^{-1}ds'.
\end{align*}
Hence we see that $I_{(P,Q)}\to I_{(P',Q')}$ is an isomorphism. This isomorphism induces
\[
	\breve{\Lambda}_{W_{m+1}(S,Q)/W_{m+1}(R,P)}^{\bullet}/I_{(P,Q)}
	\xrightarrow{\sim}
	\breve{\Lambda}_{W_{m+1}(S,Q')/W_{m+1}(R,P')}^{\bullet}/I_{(P',Q')}.
\]
By the construction, we have $J_{(P,Q)}\xrightarrow{\sim} J_{(P',Q')}$ via this morphism. So we see
\[
W_m\Lambda_{(S,Q)/(R,P)}^\bullet \xrightarrow{\sim} W_m\Lambda_{(S,Q')/(R,P')}^\bullet,
\]
and this shows $\mathcal{F}_{(P,Q)}\xrightarrow{\sim} \mathcal{F}_{(P',Q')}$.

We consider the general case.
Let $\bar{x} \rightarrow X$ be any geometric point of $X$.
By Lemma \ref{lem:chart}, there exists a commutative diagram
\[
\xymatrix{
	U=\Spec\ S' \ar[r] \ar[d]_{\gamma'}
	&
	V=\Spec\ R' \ar[d]^\gamma
	\\
	X \ar[r]^{f}
	&
	Y,
}
\]
where $U$ is an \'etale neighbourhood of $\bar{x}$,
and the morphisms $\gamma, \gamma'$ are \'etale,
such that we admit a chart $(Q''\to \mathcal{M}|_U, P''\to \mathcal{N}|_V, P''\to Q'')$ of $(U,\mathcal{M}|_U)\to (V,\mathcal{N}|_V)$
and morphisms of coherent charts
\begin{align*}
	(Q\to \mathcal{M}|_U, P\to \mathcal{N}|_V, P\to Q)
	&\to (Q''\to \mathcal{M}|_U, P''\to \mathcal{N}|_V, P''\to Q''),
	\\
	(Q'\to \mathcal{M}|_U, P'\to \mathcal{N}|_V, P'\to Q')
	&\to (Q''\to \mathcal{M}|_U, P''\to \mathcal{N}|_V, P''\to Q'').
\end{align*}
Then we see that
\begin{align*}
	&\mathcal{F}_{(P,Q)}(U)=W_m\Lambda_{(S,Q)/(R,P)}^\bullet\\
	\xrightarrow{\sim}
	&\mathcal{F}_{(P'',Q'')}(U)=W_m\Lambda_{(S,Q'')/(R,P'')}^\bullet\\
	\xleftarrow{\sim}
	&\mathcal{F}_{(P',Q')}(U)=W_m\Lambda_{(S,Q')/(R,P')}^\bullet
\end{align*}
by the proof of the previous case.
The collection of these maps glue to an isomorphism
$\mathcal{F}_{(P,Q)}\simeq \mathcal{F}_{(P',Q')}$.
\end{proof}

\subsection{Exact sequences}
The log de Rham-Witt sheaves satisfy the same exact sequences as the usual K\"{a}hler differentials.
The following results are generalizations of a part of \cite{LZGMconn}.
\begin{prop}
\label{prop:exact-sequence-Kahler-diffs}
(1) Let $X \to Y \to S$ be morphisms of fine log schemes.
Then there is the following exact sequences:
\[
	W_m\Lambda_{Y/S}^1\otimes_{W_m({\mathcal{O}_Y})}W_m\Lambda_{X/S}^{\bullet-1}
	\to
	W_m\Lambda_{X/S}^\bullet
	\to
	W_m\Lambda_{X/Y}^\bullet
	\to 0.
\]
(2) Let $X\xrightarrow{i} Y \to S$ be morphisms of fine log schemes,
where $i:X \to Y$ is an exact closed immersion defined by
a quasi-coherent ideal $\mathfrak{a}\subset \mathcal{O}_Y$.
Then there is the following exact sequences:
\[
	W_m(\mathfrak{a})/W_m(\mathfrak{a})^2\otimes_{W_m(\mathcal{O}_Y)}W_m\Lambda_{Y/S}^{\bullet-1}
	\to
	W_m(\mathcal{O}_X)\otimes_{W_m(\mathcal{O}_Y)}W_m\Lambda_{Y/S}^\bullet
	\to
	W_m\Lambda_{X/S}^\bullet
	\to
	0.
\]
\end{prop}

\begin{proof}
(1) Since the problem is local, we can assume that morphisms of log schemes are associated to morphisms of pre-log rings $(R,P) \to (S,Q) \to (S',Q')$.

Let $I_m^\bullet \subset W_m\Lambda_{(S',Q')/(R,P)}^\bullet$ be the ideal generated by the elements of the form
$ds,d\log m$ where $s\in W_m(S), m\in W_m(Q)$.
Then we see that $I_m^\bullet$ is invariant under $F, V$ and $d$.

The natural surjective morphism
\[
W_m\Lambda_{(S',Q')/(R,P)}^\bullet \to W_m\Lambda_{(S',Q')/(S,Q)}^\bullet
\]
factors $W_m\Lambda_{(S',Q')/(R,P)}^\bullet /I_m^\bullet$.

Since $I_m^\bullet$ is stable by $F, V$ and $d$,
we see $\{ W_m\Lambda_{(S',Q')/(R,P)}^\bullet / I_m^\bullet \}$
is a log $F$-$V$-procomplex over $(S',Q')/(S,Q)$.
We obtain an isomorphism
\[
W_m\Lambda_{(S',Q')/(R,P)}^\bullet /I_m^\bullet \simeq W_m\Lambda_{(S',Q')/(S,Q)}^\bullet
\]
and a short exact sequence
\[
	0
	\to
	I_m^\bullet
	\to
	W_m\Lambda_{(S',Q')/(R,P)}^\bullet
	\to
	W_m\Lambda_{(S',Q')/(S,Q)}^\bullet
	\to
	0.
\]
Thus we have the following exact sequence
\begin{align*}
	W_m\Lambda_{(S',Q')/(R,P)}^1\otimes_{W_m(S)}W_m\Lambda_{(S',Q')/(R,P)}^{\bullet-1}
	&\to
	W_m\Lambda_{(S',Q')/(R,P)}^\bullet\\
	&\to
	W_m\Lambda_{(S',Q')/(S,Q)}^\bullet
	\to
	0.
\end{align*}

(2) This is reduced to the case of a morphism of log schemes associated to a morphism of
pre-log rings $(R,P) \to (S,Q)\to (S':=S/\mathfrak{a},Q)$.
Since the canonical morphism $W_m(S,Q) \to W_m(S',Q)$ is a strict closed immersion defined by the ideal $W_m(\mathfrak{a})$,
we have the following exact sequence (\cite{Og}, Prop 2.3.2):
\[
	W_m(\mathfrak{a})/W_m(\mathfrak{a})^2
	\xrightarrow{d}
	W_m(S')\otimes_{W_m(S)}\Lambda_{W_m(S,Q)/W_m(R,P)}^1
	\to
	\Lambda_{W_m(S',Q)/W_m(R,P)}^1
	\to
	0.
\]
Then we have the following complexes:
\begin{align*}
	W_m(\mathfrak{a})/W_m(\mathfrak{a})^2\otimes_{W_m(S)}W_m\Lambda_{(S,Q)/(R,P)}^{\bullet -1}
	&\xrightarrow{d}
	W_m(S')\otimes_{W_m(S)}W_m\Lambda_{(S,Q)/(R,P)}^\bullet
	\\
	&\to
	W_m\Lambda_{(S',Q)/(R,P)}^\bullet
	\to
	0.
\end{align*}

It remains to prove the exactness at $W_m(S')\otimes_{W_m(S)}W_m\Lambda_{(S,Q)/(R,P)}^\bullet$.
Since
\[
	W_*(\mathfrak{a})W_*\Lambda_{(S,Q)/(R,P)}^\bullet+dW_*(\mathfrak{a})W_*\Lambda_{(S,Q)/(R,P)}^{\bullet-1}
\]
is stable by operators $F,V$ and $d,$
\begin{align*}
	&\{ W_m \omega_{(S',Q)/(R,P)}^\bullet \}\\
	:=
	&\{ W_m\Lambda_{(S,Q)/(R,P)}^\bullet/(W_m(\mathfrak{a})W_m\Lambda_{(S,Q)/(R,P)}^\bullet+dW_m(\mathfrak{a})W_m\Lambda_{(S,Q)/(R,P)}^{\bullet-1}) \}
\end{align*}
is a log $F$-$V$-procomplex over $(S',Q)/(R,P)$.

It is easy to verify that $\{ W_m \omega_{(S',Q)/(R,P)}^\bullet \}$
satisfies the universal property, so we have an isomorphism
$	\{ W_m \omega_{(S',Q)/(R,P)}^\bullet \}
	\xrightarrow{\sim}
	\{ W_m\Lambda_{(S',Q)/(R,P)}^\bullet \}$.
\end{proof}

\subsection{Log phantom components}
\label{subsect:log-phantom-components}
Let $R$ be a $\mathbb{Z}_{(p)}$-algebra, $(R,P) \to (S,Q)$ a morphism of pre-log rings and $M$ an $S$-module.
We denote by $M_{\mathbf{w}_m}$ the $W(S)$-module $M$ obtained by the restriction of scalars $\mathbf{w}_m:W(S) \to S$ via the Witt polynomial.

We establish the log version of phantom components defined in \cite{LZDRW} \S 2.4.
Define a complex $E_m^\bullet$ of $W_m(S)$-modules by
\[
	E_m^\bullet =
	\bigoplus_{i=0}^{m-1}\Lambda_{(S,Q)/(R,P),\mathbf{w}_i}^\bullet.
\]
We define $F:E_m^\bullet  \to E_{m-1}^\bullet$
and $V:E_m^\bullet \to E_{m+1}^\bullet$ by following formulas:
For $[\rho_0,\ldots,\rho_{m-1}]\in E_m^\bullet$, $\rho_i\in \Lambda_{(S,Q)/(R,P),\mathbf{w}_i}^\bullet$,
\begin{align*}
	F [\rho_0,\ldots,\rho_{m-1}]& := [\rho_1,\ldots,\rho_{m-1}],\\
	V [\rho_0,\ldots,\rho_{m-1}]& := [0, p\rho_0,\ldots,p\rho_{m-1}].
\end{align*}

For $1 \le i \le m$, let
$\omega_m : W_m\Lambda_{(S,Q)/(R,P)}^\bullet \to
\Lambda_{(S,Q)/(R,P),\mathbf{w}_m}^\bullet$ be the composition of the Frobenius
$F^i : W_m\Lambda_{(S,Q)/(R,P)}^\bullet \to W_{m-i}\Lambda_{(S,Q)/(R,P)}^\bullet$
and the restriction map $\pi^{m-i} : W_{m-i}\Lambda_{(S,Q)/(R,P)}^\bullet \to \Lambda_{(S,Q)/(R,P),\mathbf{w}_m}^\bullet$.
The sum of the maps $(\omega_0,\ldots,\omega_{m-1})$ define a homomorphism of projective systems of algebras
\[
	\underline{\omega}^m
	: W_m\Lambda_{(S,Q)/(R,P)}^\bullet
	\to
	E_m^\bullet.
\]
We can prove the following proposition by using the same argument of \cite{LZDRW} Proposition 2.15.

\begin{prop}
\label{prop:log-phamtom-component-factors}
The morphism $\underline{\omega}^m$ commutes with $F$ and $V$. We have
\[
	d\underline{\omega}^m =[1,p,p^2,\ldots \ ]\underline{\omega}^m d
\]
where $[1,p,p^2,\ldots \ ]\in \prod S =E_m^0$.
\end{prop}

\section{Log basic Witt differentials in special cases}
\label{sect:basicwittdiffs}

\subsection{SNCD case}
\label{subsect:log-basic-witt-diffs-sncd}
In this subsection, we consider the log version of the $p$-basic elements and the basic Witt differentials in the SNCD case.
In fact, we treat a slightly more generalized case which we need later.

Let $R$ be a $\mathbb{Z}_{(p)}$-algebra and $R[T]:=R[T_1,\ldots , T_n]$
and $e$ an integer such that $0\le e\le n$
and $f$ an nonnegative integer.
We consider the log structure associated with the map
$\bN^e \oplus \bN^f \to R[T],\ \bN^e \ni e_i\mapsto T_i \ (1\le i \le e),
 \bN^f \ni c_i\mapsto 0 \ (1\le i \le f)$
 where $e_i$ (resp. $c_i$) are basis of $\bN^e$ (resp. $\bN^f$).

We define the log $p$-basic differentials of $\Lambda_{(R[T],\bN^e \oplus \bN^f)/R}^\bullet := \Lambda_{(R[T],\bN^e \oplus \bN^f)/(R,\{*\})}^\bullet$ as follows.
Let $p^{-\infty}$ be a symbol
and we set $p\cdot p^{-\infty}:=p^{-\infty}$
and $p^{-1}\cdot p^{-\infty}:=p^{-\infty}$
and $\text{ord}_p(p^{-\infty}):=-\infty$.

A function $k:[1,n] \to \mathbb{Z}_{\ge 0} \sqcup \{p^{-\infty}\}$
is called a weight if for every $e<i\le n, k_i=k(i)\in \mathbb{Z}_{\ge 0}$.
Let $\Supp k:=\{i\in[1,n]\mid k_i\neq 0\}$.

We associate a weight without poles $k^+$ to a weight $k$ by
\begin{align*}
(k^+)_i:=
\left\{
\begin{array}{ll}
	0 & (k_i=p^{-\infty}),
	\\
	k_i & (k_i\neq p^{-\infty}).
\end{array}
\right.
\end{align*}

For each weight $k$,
we fix a total order on $\Supp\ k=\{i_1,\ldots,i_r\}$ in such a way that
\[
\text{ord}_p k_{i_1}\le \text{ord}_p k_{i_2}\le \cdots \le \text{ord}_p k_{i_r},\
\text{ord}_p k_{i_j}=\text{ord}_p k_{i_{j+1}}\Rightarrow i_j\le i_{j+1}.
\]

If $I$ is an interval of $\Supp k$,
the restriction of $k$ to $I$ will be given by $k_I$.

We say $(I_{-\infty},I_0,I_1,\ldots ,I_l)$ is a partition of $\Supp k$
if $I_j$ are intervals of $\Supp k$
and $I_{-\infty} = \{i\in[1,n]\mid k_i=p^{-\infty}\},
\Supp k = I_{-\infty} \sqcup I_0 \sqcup I_1 \sqcup \cdots \sqcup I_l$,
the elements of $I_j$ are smaller than that of $I_{j+1}$
(with respect to the fixed order)
and $I_1,\ldots ,I_l$ are not empty.
$I_{-\infty}$ and $I_0$ can be empty.
We associate the element
\[
\epsilon (k,\mathcal{P},J)
:= \left(\prod _{i\in J} d\log c_i\right)
\cdot \left(\prod _{i\in I_{-\infty}} d\log T_i\right)
\cdot e(k^+,\mathcal{P}'=(I_0, I_1, \cdots, I_l))
\]
of $\Lambda_{(R[T],\bN^e \oplus \bN^f)/R}^{|J|+|I_{-\infty}|+l}$
to the triple $(k,\mathcal{P}=(I_{-\infty},I_0,\ldots ,I_l),J)$.
Here $d\log T_i:=d\log e_i$, $k$ is a weight,
$\mathcal{P}$ is a partition of $\Supp k$ and
$J$ is a subset of $[1,f]$.
\[
e(k^+,\mathcal{P}'=(I_0, I_1, \cdots, I_l))=T^{k_{I_0}}(p^{-\text{ord}_p{k_{I_1}}}dT^{k_{I_1}})\cdots (p^{-\text{ord}_p{k_{I_l}}}dT^{k_{I_l}})
\]
is the $p$-basic element defined in \cite{LZDRW} \S 2.1, where
$T^{k_{I_j}} = \prod_{i \in I_j}T_i^{k_i}$ and
$\text{ord}_p k_{I_j} = \min_{i \in I_j} \text{ord}_p k_i$.
We call the elements of this form log $p$-basic elements.

\begin{lem}
\label{lem:log-p-basic-diffs-sncd}
The log $p$-basic elements form a base of the log de Rham complex $\Lambda_{(R[T],\bN^e \oplus \bN^f)/R}^\bullet$ as an $R$-module.
\end{lem}

\begin{proof}
The $R$-module $\Lambda_{(R[T],\bN^e \oplus \bN^f)/R}^l$ has the following basis:
\begin{align}
	d\log c_{h_1} \cdots d\log c_{h_s}
	\cdot d\log T_{i_1}\cdots d\log T_{i_m}
	\cdot \left(\prod_{j \in [1,n]}T_j^{k_j}\right)
	\cdot d\log T_{j_1}\cdots d\log T_{j_{l-m-s}},
\end{align}
where
$1\le h_1 < \cdots < h_s\le f,
1\le i_1 < \cdots < i_m\le e,
k:[1,n]\to \mathbb{Z}_{\ge 0},
k|_{\{i_1,\ldots,i_m \}}=0,
j_1< \cdots < j_{l-m-s} \in \Supp k$.

Let $\Lambda^l (I,k,J)\subset \Lambda_{(R[T],\bN^e \oplus \bN^f)/R}^l$ be the free $R$-submodule spanned by all elements of the form $(1)$
for a fixed $I=\{i_1,\ldots,i_m \}\subset [1,e]$, $J =\{ h_1,\ldots,h_s \} \subset [1,f]$
and a weight $k:[1,n]\to \mathbb{Z}_{\ge 0}$
such that $k|_{\{i_1,\ldots,i_m \}}=0$.
We have a decomposition as free $R$-modules:
\[
\Lambda_{(R[T],\bN^e \oplus \bN^f)/R}^l
=
\bigoplus_{I,k,J} \Lambda^l (I,k,J).
\]
The rank of $\Lambda^l (I,k,J)$ is $\binom ms$
where $m=|\Supp k |$ and $s=l-| I | - | J |$.
The number of log $p$-basic elements of the form
\[
	\left(\prod _{i\in J} d\log c_i\right)
	\cdot \left(\prod _{i\in I} d\log T_i\right)
	\cdot e(k,\mathcal{P})
\]
for fixed $I$ and $J$ and $k$ is also $\binom ms$.
Hence it is enough to show that the log $p$-basic elements of this form generate $\Lambda^l (I,k,J)$ as an $R$-module.
It follows from the proof of \cite{LZDRW} Proposition 2.1.
\end{proof}

Next we determine the log version of the basic Witt differentials for the pre-log ring
$(R[T],\bN^e \oplus \bN^f)=(R[T_1,\ldots,T_m],\bN^e \oplus \bN^f)$ over $R=(R,\{*\})$.
We denote by $X_i \in W(R)$ the Teichm\"{u}ller lift $[T_i]$ of $T_i$ .
We consider the log de Rham-Witt complex
$W\Lambda_{(R[T],\bN^e \oplus \bN^f)/R}^\bullet$.

We call a function $k:[1,n] \to \mathbb{Z}_{\ge 0}[1/p] \sqcup \{p^{-\infty}\}$ a weight
if for $e <i\le n,\ k_i:=k(i)\in \mathbb{Z}_{\ge 0}[1/p]$.
Set $t(k_{i_l}):=-\text{ord}_p k_{i_l}$ and $u(k_{i_l})=\max (0,t(k_{i_l}))$.
For each weight $k$, we fix a total order on $\Supp\ k=\{i_1,\ldots,i_r\}$ in such a way that
\[
	\text{ord}_p k_{i_1}\le \text{ord}_p k_{i_2}\le \cdots \le \text{ord}_p k_{i_r},\
	\text{ord}_p k_{i_j}=\text{ord}_p k_{i_{j+1}}\Rightarrow i_j\le i_{j+1}.
\]

If $I=\{i_t,\cdots,i_{t+m}\}$ is an interval of $\Supp k$,
the restriction of $k$ to $I$ will be given by $k_I$.
We set $t(k_I)=t(k_{i_t}), u(k_I)=u(k_{i_t})$.
If $k$ is fixed in our discussion,
we write $t(I)$ and $u(I)$ instead of $t(k_I)$ and $u(k_I)$.

We say $(I_{-\infty},I_0,I_1,\ldots ,I_l)$ is a partition of $\Supp\ k$
if $I_j$ are intervals of $\Supp k$,
$I_{-\infty} = \{i\in[1,n]\mid k_i=p^{-\infty}\},
\Supp k = I_{-\infty} \sqcup I_0 \sqcup I_1 \sqcup \cdots \sqcup I_l$,
the elements of $I_j$ are smaller than that of $I_{j+1}$
(with respect to the fixed order)
and $I_1,\ldots ,I_l$ are not empty.
$I_{-\infty}$ and $I_0$ can be empty.

Let $(\xi,k,\mathcal{P},J)$ be a quadruple such that $k$ is a weight,
$\mathcal{P}=(I_{-\infty},I_0,\ldots,I_l)$ is a partition of $k$,
$\xi \in {}^{V^{u(k^+)}}W(R)$ and $J \subset [1,f]$.
We define a log basic Witt differential $\epsilon=\epsilon(\xi,k,\mathcal{P},J)\in W\Lambda_{(R[T],\bN^e \oplus \bN^f)/R}^{|J| + |I_{-\infty}|+l}$ by
\[
	\epsilon =
	\left(\prod _{i\in J} d\log c_i\right)
	\cdot \left( \prod_{i\in I_{-\infty}}d\log X_i \right)
	\cdot e(\xi,k^+,(I_0,\ldots,I_l)),
\]
where $e(\xi,k^+,(I_0,\ldots,I_l))$ is the basic Witt differential defined in \cite{LZDRW} \S 2.2.
We call the log basic Witt differential $\epsilon(\xi,k,\mathcal{P},J)$ is integral
if $e(\xi,k^+,(I_0,\ldots,I_l))$ is integral,
i.e., $(k^+)_i\in \mathbb{Z}_{\ge 0}$ for all $i$.
The log basic Witt differential $\epsilon(\xi,k,\mathcal{P},J)$ is called fractional if it is not integral.

We denote by $\epsilon_m (\xi,k,\mathcal{P},J)$ the image of $\epsilon (\xi,k,\mathcal{P},J)$ in $W_m\Lambda_{(R[T],\bN^e \oplus \bN^f)/R}^\bullet$.
The element $\epsilon_m (\xi,k,\mathcal{P},J)$ depends only on the residue class $\bar{\xi}$ of $\xi$ in $W_m(R)$.
We see $\bar{\xi}\in {}^{V^u}W_{m-u}(R)$ because $\xi \in {}^{V^u}W(R)$ for $u=u(k^+)$.
We have $\epsilon_m (\xi,k,\mathcal{P},J)=0$ if $p^{m-1}\cdot k^+$ is not integral.

The relations
\begin{align*}
	{}^F d\log c_i=\log c_i,\
	&{}^F d\log X_i=\log X_i,\\
	{}^V d\log c_i=V(1)\log c_i,\
	&{}^V d\log X_i=V(1)\log X_i,\\
	d(d\log c_i)=0,\
	&d(d\log X_i)=0
\end{align*}
and \cite{LZDRW} Proposition 2.5 and 2.6 give the following formulas:

(1)
\begin{align*}
&^F\epsilon(\xi,k,(I_{-\infty},I_0,\ldots,I_l),J)\\
=&\left\{
\begin{array}{ll}
	\epsilon(^F\xi,pk,(I_{-\infty},I_0,\ldots,I_l),J)
	& (I_0 \neq \emptyset, k^+\ \text{is integral}),
	\\
	\epsilon(^{V^{-1}}\xi,pk,(I_{-\infty},I_0,\ldots,I_l),J)
	& (I_0 = \emptyset,k^+\ \text{not integral}).
\end{array}
\right.
\end{align*}

(2)
\begin{align*}
&^V\epsilon(\xi,k,(I_{-\infty},I_0,\ldots,I_l),J)\\
=&\left\{
\begin{array}{ll}
	\displaystyle \epsilon(^V\xi,\frac{1}{p}k,(I_{-\infty},I_0,\ldots,I_l),J)
	&
	(I_0 \neq \emptyset \ \text{or}\ k^+\ \text{is integral and divisible by}\ p),\\
	\displaystyle \epsilon(p^V\xi,\frac{1}{p}k,(I_{-\infty},I_0,\ldots,I_l),J)
	&
	(I_0 = \emptyset,(1/p)k^+\ \text{is not integral}).
\end{array}
\right.
\end{align*}

(3) If $I=\Supp k^+$ and $t=t(k_I)$,
\begin{align*}
&d\epsilon(\xi,k,(I_{-\infty},I_0,\ldots,I_l),J)\\
=&\left\{
\begin{array}{ll}
	0
	&
	(I = \emptyset \ \text{or}\ I_0=\emptyset),
	\\
	\epsilon(\xi,k,(I_{-\infty},\emptyset,I_0,\ldots,I_l),J)
	&
	(I_0 \neq \emptyset, k^+\ \text{not integral}),
	\\
	p^{-t}\epsilon(\xi,k,(I_{-\infty},\emptyset,I_0,\ldots,I_l),J)
	&
	(I_0 \neq \emptyset,k^+\ \text{integral}).
\end{array}
\right.
\end{align*}

Let
\[
	\widetilde{\omega}_m:W\Lambda_{(R[T_1,\ldots,T_n],\bN^e \oplus \bN^f)/R}^\bullet
	\to
	\Lambda_{(R[T_1,\ldots,T_n],\bN^e \oplus \bN^f)/R,\mathbf{w}_m}^\bullet.
\]
be the composition of
\[
\omega_m:
	W_{m+1}\Lambda_{(R[T_1,\ldots,T_n],\bN^e \oplus \bN^f)/R}^\bullet
	\to
	\Lambda_{(R[T_1,\ldots,T_n],\bN^e \oplus \bN^f)/R,\mathbf{w}_m}^\bullet
\]
which we defined in \S \ref{subsect:log-phantom-components}
followed by the natural projection map
\[
	W\Lambda_{(R[T_1,\ldots,T_n],\bN^e \oplus \bN^f)/R}^\bullet
	\to
	W_{m+1}\Lambda_{(R[T_1,\ldots,T_n],\bN^e \oplus \bN^f)/R}^\bullet.
\]

\begin{prop}
\label{prop:calc-phantom-component-sncd}
Let $\epsilon = \epsilon(\xi,k,(I_{-\infty},I_0,\ldots,I_l),J)
\in W\Lambda_{(R[T_1,\ldots,T_n],\bN^e \oplus \bN^f)/R}^{|J| + | I_{-\infty} |+l}$
be a log basic Witt differential where $\xi = {}^{V^u}\eta, u=u(k^+)$. Then

\begin{align*}
&\widetilde{\omega}_m(\epsilon)\\
=&
\left\{
\begin{array}{l}
	0 \hspace{18em} (\text{if } p^m\cdot k^+\ \text{not integral}),
	\\
	\mathbf{w}_m(\xi)
	\cdot \left( \prod_{i\in J}d\log c_i \right)
	\cdot \left( \prod_{i\in I_{-\infty}}d\log T_i \right)\cdot\\
	\hspace{8em}
	T^{p^m k_{I_0}} (p^{-\text{ord}p^m k_{I_1}}dT^{p^m k_{I_1}})
	\cdots (p^{-\text{ord}p^m k_{I_l}}dT^{p^m k_{I_l}})\\
	\hspace*{12em}(\text{if } p^m \cdot k^+\ \text{integral},
	I_0 \neq \emptyset \  \text{or}\ k^+ \ \text{integral}),\\
	\mathbf{w}_{m-u}(\eta)\cdot \left( \prod_{i\in J}d\log c_i \right)
	\cdot \left( \prod_{i\in I_{-\infty}}d\log T_i \right) \cdot \\
	\hspace{8em} (p^{-\text{ord}p^m k_{I_1}} dT^{p^m k_{I_1}})
	\cdots (p^{-\text{ord}p^m k_{I_l}}dT^{p^m k_{I_l}})\\
	\hspace*{18em}(\text{if}\ p^m\cdot k^+\ \text{integral}, I_0 = \emptyset).\\
\end{array}
\right.
\end{align*}
 \end{prop}

\begin{proof}
It follows from the construction of $\mathbf{w}_m$, \cite{LZDRW} Proposition 2.16 and calculations of log parts.
\end{proof}

\begin{prop}
\label{prop:expression-basic-witt-sncd}
Any element of $W\Lambda_{(R[T_1,\ldots,T_n],\bN^e \oplus \bN^f)/R}^\bullet$ has a unique expression as a convergent sum of log basic Witt differentials:
\begin{align}
\label{log-de-rham-witt-diff-expression}
	\sum_{k,\mathcal{P},J}\epsilon(\xi_{k,\mathcal{P},J},k,\mathcal{P},J),
\end{align}
where $k$ runs over all possible weights,
$\mathcal{P}$ over all partitions
and $J$ over all subsets of $[1,f]$.
A convergent sum means that for any given number $m$,
we have $\xi_{k,\mathcal{P},J}\in {}^{V^m}W(R)$ for all but finitely many weights $k$.
\end{prop}

\begin{proof}
For $\xi \in W(R[T])$, we can see that $\xi$ can be written uniquely as a convergent sum
$ \xi = \sum_{k,m\ge 0} {}^{V^m}([a_{k,m}]X^k)$,
where $X_r=[T_r]$, $a_{k,m}\in R$ and $k$ runs all possible integral weights.

For a given nonnegative integer $m$ and a weight $k$,
$\rho \le m$ denotes the maximum nonnegative integer such that $p^{-\rho}k$  is integral. Then we have
\[
	{}^{V^m}([a_{k,m}]X^k) = {}^{V^{m-\rho}}({}^{V^\rho}[a_{k,m}]X^{p^{-\rho}k}).
\]
Hence $\xi$ is written as the convergent sum
$\xi= \sum_{k:\text{weight}} {}^{V^{u(k)}}(\eta_{k} X^{p^{u(k)}k})$.

Since we have a canonical surjective map
\begin{align*}
	\Omega_{W(R[T])/W(R)}^1
	\oplus \bigoplus_{i=1}^e W(R[T])d\log T_i
	\oplus \bigoplus_{i=1}^f W(R[T])& d\log c_i
	\to
	\\
	&\Lambda_{(W(R[T]),\bN^e \oplus \bN^f)/W(R)}^1,
\end{align*}
any element in $W\Lambda_{(R[T_1,\ldots,T_n],\bN^e \oplus \bN^f)/R}^\bullet$ is written as a convergent sum of elements of the form
\begin{align*}
	d\log c_{i_1} &\cdots d\log c_{i_s}
	\cdot d\log X_{j_1}\cdots d\log X_{j_l}\cdot\\
	&{}^{V^{u_0}}
	(\eta_0 X^{p^{u_0}k^{(0)}})
	d^{V^{u_1}}(\eta_1 X^{p^{u_1}k^{(1)}})
	\cdots d^{V^{u_m}}(\eta_m X^{p^{u_m}k^{(m)}}),\ \cdots (*)
\end{align*}
where
$1\le i_1<\cdots <i_s \le f,
1\le j_1 < \cdots < j_l \le e$,
$k^{(0)},\ldots ,k^{(m)}$ are weights
and $u_i$ is the least nonnegative integer such that $p^{u_i}\cdot k^{(i)}$ is integral.

We prove that all the elements of the form $(*)$ can be written as a sum of log basic Witt differentials, by dividing them into four cases.

Case 0. $\{j_1,\ldots ,j_l \}\cap (\bigcup_{i=0}^m \Supp\ k^{(i)}) = \emptyset$.

If $\{j_1,\ldots ,j_l \}\cap (\bigcup_{i=0}^m \Supp\ k^{(i)}) = \emptyset$,
$(*)$ can be written as a sum of log basic Witt differentials
by \cite{LZDRW} Theorem 2.8 and our definition of log basic Witt differentials.

Case 1. $k^{(i)}$ are all integral, i.e., $u_0=u_1=\cdots=u_m=0$.
\begin{align*}
	(*)=d\log c_{i_1} \cdots d\log c_{i_s}
	\cdot d\log X_{j_1} & \cdots d\log X_{j_l}\cdot\\
	& (\eta_0 X^{k^{(0)}})d(\eta_1 X^{k^{(1)}})\cdots d(\eta_m X^{k^{(m)}}).
\end{align*}
It can be reduced to the case 0 by following calculations.

We write $e_i$ for $(0,\ldots, 1 ,\ldots,0)$,
whose $i$th entry is $1$ and the others are $0$.
If $k$ is an integral weight without poles and $t \in \Supp k$, we have
\begin{align*}
	d\log X_t\cdot X^k &= X^{k-k_t e_t} \cdot X_t^{k_t-1}dX_t,\\
	d\log X_t \cdot dX^k
	&= d\log X_t \cdot (X_t^{k_t}dX^{k-k_t e_t}+X^{k-k_t e_t}dX_t^{k_t})\\
	&= d\log X_t \cdot X_t^{k_t}dX^{k-k_t e_t}\\
	&= X_t^{k_t-1}dX_t \cdot dX^{k-k_t e_t}.
\end{align*}

Case 2. $u_0 \ge u_j$ for $j=1,\ldots ,m$.

We can rewrite $(*)$ as follows:
\begin{align*}
V^{u_0}(d\log c_{i_1}& \cdots d\log c_{i_s}
\cdot d\log X_{j_1}\cdots d\log X_{j_l} \cdot \\
&\eta_0 X^{p^{u_0}k^{(0)}}\cdot
{}^{F^{u_0-u_1}}d(\eta_1 X^{p^{u_1}k^{(1)}})
\cdots {}^{F^{u_0-u_m}}d(\eta_m X^{p^{u_m}k^{(m)}}).
\end{align*}
Since $V$ maps log basic Witt differentials to log basic Witt differentials,
it follows from case 1.

Case 3. $u_1 \ge u_j$ for $j=0,\ldots ,m$.

We apply Leibniz rule:
\begin{align*}
	^{V^{u_0}}(\eta_0 X^{p^{u_0}k^{(0)}})d^{V^{u_1}}(\eta_1 X^{p^{u_1}k^{(1)}})
	=d(^{V^{u_0}}(\eta_0 X^{p^{u_0}k^{(0)}})^{V^{u_1}}(\eta_1 X^{p^{u_1}k^{(1)}}))\\
	-^{V^{u_1}}(\eta_1 X^{p^{u_1}k^{(1)}})d^{V^{u_0}}(\eta_0 X^{p^{u_0}k^{(0)}}).
\end{align*}

By the Leibniz rule and the fact that $d$ maps log basic Witt differentials to log basic Witt differentials,
we can reduce it to the former three cases.

Next we prove the independence of the log basic Witt differentials.
Suppose the element $\omega = \sum_{k,\mathcal{P},J}\epsilon (\xi_{k,\mathcal{P},J},k,\mathcal{P},J)$ of the form as (\ref{log-de-rham-witt-diff-expression}) is equal to zero.
We show $\xi_{k,\mathcal{P},J}=0$ for all $k,\mathcal{P}, J$.
It suffices to show that the image of $\xi_{k,\mathcal{P},J}$ in $W_m(R)$ is zero for all $m$.
We fix a positive integer $m$.
Let $\bar{\xi}_{k,\mathcal{P},J}$ be the image of $\xi_{k,\mathcal{P},J}$ in $W_m(R)$.
First we suppose $R$ is $p$-torsion free.
Consider the morphism
\[
	\widetilde{\omega}_i:
	W\Lambda_{(R[T_1,\ldots,T_n],\bN^e \oplus \bN^f)/R}^\bullet
	\to
	\Lambda_{(R[T_1,\ldots,T_n],\bN^e \oplus \bN^f)/R,\mathbf{w}_i}^\bullet.
\]
for $0 \le i \le m-1$.
Proposition \ref{prop:calc-phantom-component-sncd} shows that
$\mathbf{w}_i(\xi_{k,\mathcal{P},J})=0$
for $0 \le i \le m-1$ because log $p$-basic elements are linearly independent by Lemma \ref{lem:log-p-basic-diffs-sncd}.
Since we assume that $R$ has no $p$-torsion,
$\bar{\xi}_{k,\mathcal{P},J}=0$ for all $k,
\mathcal{P}, J$.
Hence the proof of independence is completed if $R$ is $p$-torsion free.

We consider the general case.
Take a surjective ring homomorphism $\phi:\widetilde{R}\to R$
where $\widetilde{R}$ is a ring without $p$-torsion.
Set $\mathfrak{a}:=\ker \phi$.
Let $(\widetilde{R}[T],\bN^e \oplus \bN^f)$ be a pre-log ring
whose pre-log structure is given by
$\bN^e \oplus \bN^f \to \widetilde{R}[T],\
\bN^e \ni e_i\mapsto T_i \ (1\le i \le e), \bN^f \ni c_i\mapsto 0 \ (1\le i \le f)$.
We denote by $W\Lambda_{(\mathfrak{a}\widetilde{R}[T],\bN^e \oplus \bN^f)/\widetilde{R}}^\bullet$ the subgroup of $W\Lambda_{(\widetilde{R}[T],\bN^e \oplus \bN^f)/\widetilde{R}}^\bullet$
which consists of convergent sums of log basic Witt differential of $\epsilon(\xi_{k,\mathcal{P},J},k,\mathcal{P},J)$
with $\xi_{k,\mathcal{P},J}\in W(\mathfrak{a})$.
We see $W\Lambda_{(\mathfrak{a}\widetilde{R}[T],\bN^e \oplus \bN^f)/\widetilde{R}}^\bullet$ is a ideal of $W\Lambda_{(\widetilde{R}[T],\bN^e \oplus \bN^f)/\widetilde{R}}^\bullet$
by the first part of the proof and Proposition 2.11 of \cite{LZDRW}.
Let $W_m\Lambda_{(\mathfrak{a}\widetilde{R}[T],\bN^e \oplus \bN^f)/\widetilde{R}}^\bullet$
be the image of $W\Lambda_{(\mathfrak{a}\widetilde{R}[T],\bN^e \oplus \bN^f)/\widetilde{R}}^\bullet$
in $W_m\Lambda_{(\widetilde{R}[T],\bN^e \oplus \bN^f)/\widetilde{R}}^\bullet$.
Define a procomplex $\{ E_m^\bullet \}$ by
\[
	E_m^\bullet:=W_m\Lambda_{(\widetilde{R}[T],\bN^e \oplus \bN^f)/\widetilde{R}}^\bullet
	/W_m\Lambda_{(\mathfrak{a}\widetilde{R}[T],\bN^e \oplus \bN^f)/\widetilde{R}}^\bullet.
\]
Set $E^\bullet:=\varprojlim_m E_m^\bullet$.
Then we have $E_m^0=W_m(R)$ and
\[
	E^\bullet
	\simeq
	W\Lambda_{(\widetilde{R}[T],\bN^e \oplus \bN^f)/\widetilde{R}}^\bullet
	/W\Lambda_{(\mathfrak{a}\widetilde{R}[T],\bN^e \oplus \bN^f)/\widetilde{R}}^\bullet.
\]

Since $W\Lambda_{(\mathfrak{a}\widetilde{R}[T], \bN^e \oplus \bN^f)/\widetilde{R}}^\bullet$ is invariant under $F,V$ and $d$, we see $\{ E_m^\bullet \}$ is an log $F$-$V$-procomplex over $(R[T],\bN^e \oplus \bN^f)/R$. Hence we obtain a morphism
\[
	\{ W_m\Lambda_{(R[T],\bN^e \oplus \bN^f)/R}^\bullet \}
	\to
	\{ E_m^\bullet \}
\]
of log $F$-$V$-procomplexes. Then there is the following commutative diagram
\[
\xymatrix{
	W\Lambda_{(\widetilde{R}[T],\bN^e \oplus \bN^f)/\widetilde{R}}^\bullet \ar@{->>}[d] \ar@{->>}[r]
	&
	W\Lambda_{(R[T],\bN^e \oplus \bN^f)/R}^\bullet \ar[ld]
	\\
	E^\bullet . &
}
\]

By the $p$-torsion free case, any element $\omega$ of $W\Lambda_{(\widetilde{R}[T],\bN^e \oplus \bN^f)/\widetilde{R}}^\bullet$ is uniquely written as a convergent sum as (\ref{log-de-rham-witt-diff-expression}).
The commutativity of the diagram indicated above and the fact that the composite morphism
\[
	W\Lambda_{(\mathfrak{a}\widetilde{R}[T],\bN^e \oplus \bN^f)/\widetilde{R}}^\bullet
	\to
	W\Lambda_{(\widetilde{R}[T],\bN^e \oplus \bN^f)/\widetilde{R}}^\bullet
	\to
	W\Lambda_{(R[T],\bN^e \oplus \bN^f)/R}^\bullet
\]
is zero implies Proposition \ref{prop:expression-basic-witt-sncd} holds for any $R$.
\end{proof}

\begin{cor}
Any element $\omega$ of $W_m\Lambda_{(R[T_1,\ldots,T_n],\bN^e \oplus \bN^f)/R}^\bullet$ can be written as a finite sum
\[
	\omega =
	\sum_{k,\mathcal{P},J}\epsilon_m (\xi_{k,\mathcal{P},J},k,\mathcal{P},J),\
	\xi_{k,\mathcal{P},J}\in {}^{V^{u(k^+)}}W_{m-u(k^+)}(R).
\]
Here $k$ runs over all weights such that $p^{m-1}\cdot k^{+}$ is integral,
$\mathcal{P}$ runs over all partitions and $J$ over all subsets of $[1,f]$.
The coefficients $\xi_{k,\mathcal{P},J}$ are uniquely determined by $\omega$.
\end{cor}

\subsection{Semistable case}
\label{subsect:log-basic-witt-diffs-semistable}
We consider the log $p$-basic elements and the basic Witt differentials in specific cases, which contains the semistable case.

For positive integers $d\le e \le n$ and a nonnegative integer $f$,
we consider the pre-log ring
\begin{align*}
	& (A=R[T_1,\ldots,T_n]/(T_1\cdots T_d),P=\bN^{e}\oplus \bN^f),\\
	& \bN^e \ni e_i\mapsto T_i \in A \ (1\le i \le e),
	\bN^f \ni c_i\mapsto 0 \in A \ (1\le i \le f),
\end{align*}
where $e_i$ (resp. $c_i$) are basis of $\bN^e$ (resp. $\bN^f$), for later discussions in this paper.
The module of (relative) log differential forms $\Lambda_{(A,P)/R}^1$ is isomorphic to a free $A$-module
$\bigoplus_{i=1}^e A d\log T_i
\oplus \bigoplus_{i=e+1}^n A dT_i
\oplus \bigoplus_{i=1}^f A d\log c_i$.
Hence $\Lambda_{(A,P)/R}^\bullet$ has the following elements as a basis of $R$-module:
\[
	T_1^{k_1}\cdots T_n^{k_n}
	\cdot \prod_{i\in G} d\log T_i
	\cdot \prod_{i\in H} d\log T_i
	\cdot \prod_{j\in I} d\log T_i
	\cdot \prod_{i\in J} d\log c_i ,
\]
where $G\subset [1,d], H\subset [d+1,e],
I\subset [e+1,n] \cap \Supp k, J\subset [1,f]$
and $\min_{1\le i \le d}k_i = 0$.
We conclude that the log $p$-basic differentials $\epsilon (k,\mathcal{P},J)$
satisfying $[1,d]\not \subset \Supp k^+$ forms the basis as an $R$-module
by a similar argument to that in Lemma \ref{lem:log-p-basic-diffs-sncd}.

Next we study the basic Witt differentials of $W\Lambda_{(A,P)/R}^\bullet$.

\begin{prop}
Any element in $W\Lambda_{(A,P)/R}^\bullet$ has a unique expression as a convergent sum
\[
	\sum_{k,\mathcal{P},J} \epsilon(\xi_{k,\mathcal{P},J},k,\mathcal{P},J)
\]
of log basic Witt differentials.
Here $k$ runs over all possible weights such that $[1,d]\not\subset \Supp^+ k$,
$\mathcal{P}$ over all partitions of $\Supp k$
and $J$ over all subsets of $[1,f]$.
A convergent sum means that for any given number $m$,
we have $\xi_{k,\mathcal{P},J}\in {}^{V^m}W(R)$ for all but finitely many weights $k$.
\end{prop}

\begin{proof}
As proof of Proposition \ref{prop:expression-basic-witt-sncd},
any element $\xi$ of $W(A)$ can be written as the following convergent sum:
\[
	\xi= \sum_{k:\text{weight},[1,d]\not\subset \Supp k} {}^{V^{u(k)}}(\eta_{k} X^{p^{u(k)}k}).
\]
Hence an element of $W\Lambda_{(A,P)/R}^\bullet$ can be written as a convergent sum of the following form
\begin{align*}
	d\log c_{i_1} &\cdots d\log c_{i_s}
	\cdot d\log X_{j_1}\cdots d\log X_{j_l} \cdot \\
	&{}^{V^{u_0}}(\eta_0 X^{p^{u_0}k^{(0)}}) d{}^{V^{u_1}}(\eta_1 X^{p^{u_1}k^{(1)}})\cdots d{}^{V^{u_m}}(\eta_m X^{p^{u_m}k^{(m)}})
\end{align*}
with $1\le i_1<\cdots <i_s \le f, 1\le j_1<\cdots <j_l \le e$,
each $k^{(i)}$ is a weight satisfying $[1,d]\not\subset \Supp k^{(i)}$ for all $i$
and $u_i$ is the least nonnegative integer such that $p^{u_i}k^{(i)}$ is integral.

We show this is equal to zero if $[1,d]\subset \bigcup_{i=1}^m \Supp k^{(i)}$.
We can assume that all $k^{(i)}$ are integral by the proof of Proposition \ref{prop:expression-basic-witt-sncd}.
If $k$ is an integral weight, $dX^k$ is divisible by $X^{k|_{[1,d]}}$.
Hence if $[1,d]\subset \bigcup_{i=1}^m \Supp k^{(i)}$, the element indicated above is zero.

We can prove that any element of $W\Lambda_{(A,P)/R}^\bullet$ can be written as the form indicated in the proposition in the same as Proposition \ref{prop:expression-basic-witt-sncd}
because the actions of $F,V,d$ on log basic Witt differentials do not change the condition $[1,d]\not\subset \Supp^+ k$.

We can also show that this expression is unique by a similar argument to the proof of Proposition \ref{prop:expression-basic-witt-sncd}.
\end{proof}

\begin{cor}
\label{prop:expression-basic-witt-semistable}
Any element $\omega$ of $W_m\Lambda_{(A,P)/R}^\bullet$ can be written as a finite sum
\[
	\omega = \sum_{k,\mathcal{P},J}\epsilon_m(\xi,k,\mathcal{P},J),\
	\xi_{k,\mathcal{P},J}\in {}^{V^{u(k^+)}}W_{m-u(k^+)}(R).
\]
Here $\epsilon_m$ is the image of $\epsilon$ on $W_m\Lambda_{(A,P)/R}^\bullet$,
$k$ runs over all weights such that $[1,d]\not\subset \Supp^+ k$,
$p^{m-1}\cdot k^{+}$ is integral,
$\mathcal{P}$ over all partitions of $\Supp k$ and $J$ over all subsets of $[1,f]$.
The coefficients $\xi_{k,\mathcal{P},J}$ are uniquely determined by $\omega$.
\end{cor}

Set $d\log X := d\log X_1 + \cdots + d\log X_e + d\log c_1 + \cdots + d\log c_f$.

We define an element $\epsilon'(\xi,k,(I_{-\infty}, I_0, \ldots,I_l),J)$
for a log basic Witt differential
$\epsilon(\xi,k,(I_{-\infty}, I_0, \ldots,I_l),J)$ by
\[
\epsilon'=
\left\{
\begin{array}{l}
	\epsilon(\xi,k,(I_{-\infty},I_0,\ldots,I_l),J)
	 \hspace{14em}
	(k_e\neq p^{-\infty}),
	\\
	\left( \prod_{j\in J} d\log c_j \right)
	\cdot \left( \prod_{i\in I_{-\infty},i\neq e} d\log X_i \right)
	\cdot d\log X
	\cdot e(\xi,k^+,(I_0,\ldots,I_l))
	\\
	\hspace{25em}
	(k_e = p^{-\infty}),
\end{array}
\right.
\]
where $e(\xi,k^+,(I_0,\ldots,I_l))$ is the classical basic Witt differential defined in \cite{LZDRW}.
If $k_e= p^{-\infty}$, we see
\begin{align*}
	&\epsilon'(\xi,k,(I_{-\infty},I_0,\ldots,I_l),J)\\
	=& \epsilon(\xi,k,(I_{-\infty},I_0,\ldots,I_l),J)\\
	&+ \left( \prod_{i\in J} d\log c_i \right)
	\left(\prod_{i\in I_{-\infty},i\neq e} d\log X_i \right)
	\left( \sum_{i\in[1,e]\setminus I_{-\infty}} d\log X_i \right)
	e(\xi,k^+,(I_0,\ldots,I_l))\\
	&+ \left( \prod_{i\in J} d\log c_i \right)
	\left(\prod_{i\in I_{-\infty},i\neq e} d\log X_i \right)
	\left( \sum_{i\in[1,f]\setminus J} d\log c_i \right)
	e(\xi,k^+,(I_0,\ldots,I_l))\\
	=& \epsilon(\xi,k,(I_{-\infty},I_0,\ldots,I_l),J)\\
	&+(\text{linear combination of\ } \epsilon(\xi,k,(I_{-\infty},I_0,\ldots,I_l),J)\
	\text{such that}\ k_e\neq p^{-\infty},\\
	&I_{-\infty}\ \text{and}\ J \ \text{vary and}\ \epsilon \ \text{are different from the above}).
\end{align*}
From this we obtain

\begin{prop}
\label{prop:decomposition-of-drw-cpx-semistable}
$W\Lambda_{(A,P)/R}^\bullet$ has a decomposition as $W(R)$-modules:
\[
	W\Lambda_{(A,P)/R}^\bullet = WC_{(A,P)/R}^\bullet \oplus WC_{(A,P)/R}^{'\bullet},
\]
where $WC_{(A,P)/R}^\bullet$ (resp.  $WC_{(A,P)/R}^{'\bullet}$) consists of the elements
which can be written as a convergent sum of the elements of the form $\epsilon'$
such that $k_e\neq p^{-\infty}$ (resp. $k_e =  p^{-\infty}$).
\end{prop}

Note that the decomposition we stated above is \textit{not} a decomposition as complexes.

\section{Log Witt lift and Log Frobenius lift}
\label{sect:log-witt-lift-and-log-frobenius-lift}

Let $R$ be a $\mathbb{Z}_{(p)}$-algebra in which $p$ is nilpotent and $(R,P) \to (S,Q)$ a log smooth morphism of pre-log rings.
We define the log version of Witt lifts and Frobenius lifts of \cite{LZDRW} \S 3.1.

\begin{defn}
A log Witt lift of $(S,Q)$ over $(R,P)$ is a system $((S_n,Q_n),\delta_n:(S_n,Q_n) \to W_n(S,Q))_{n\ge 1}$ satisfying the following conditions.

(1) For each $n\ge 1$, $(S_n,Q_n)$ is log smooth over $W_n(R,P)$, and
\[
	W_n(R,P)\otimes_{W_{n+1}(R,P)}(S_{n+1},Q_{n+1})
	\simeq
	(S_n,Q_n),\
	(S_1,Q_1)=(S,Q).
\]

(2) Let $\mathbf{w}_0:W_n(S,Q)\to (S,Q)$ be the morphism
induced by the Witt polynomial $\mathbf{w}_0: W_n(S)\to S$ and $\text{id}_Q$.
For $n> 1$, $\mathbf{w}_0\delta_n$ is the natural map $(S_n,Q_n)\to (S,Q)$
and the following diagram commutes:
\[
\xymatrix{
	(S_{n+1},Q_{n+1}) \ar[r]^{\delta_{n+1}} \ar[d]
	&
	W_{n+1}(S,Q) \ar[d]
	\\
	(S_n,Q_n) \ar[r]^{\delta_n}
	&
	W_n(S,Q).
}
\]
\end{defn}

\begin{defn}
A log Frobenius lift of $(S,Q)$ over $(R,P)$ is a system
\[
	((S_n,Q_n),\phi_n :(S_n,Q_n)
	\to
	(S_{n-1},Q_{n-1}),\
	\delta_n:(S_n,Q_n)
	\to
	W_n(S,Q))_{n\ge 1},
\]
satisfying the following conditions:

(1) $((S_n,Q_n),\delta_n)$ is a log Witt lift of $(S,Q)$ over $(R,P)$.

(2) For $n\ge 1$, $\phi_n$ is compatible with the Frobenius on the log Witt ring $F:W_n(R,P) \to W_{n-1}(R,P)$,
the absolute Frobenius $\text{Frob} : S/pS \to S/pS$ and $\times p : Q \to Q$.

(3) The following diagram commutes:
\[
\xymatrix{
	(S_{n+1},Q_{n+1}) \ar[r]^{\delta_{n+1}} \ar[d]_{\phi_{n+1}}
	&
	W_{n+1}(S,Q) \ar[d]_F
	\\
	(S_n,Q_n) \ar[r]^{\delta_n}
	&
	W_n(S,Q).
}
\]
\end{defn}

We also define log Witt lifts and log Frobenius lifts for a morphism
$f:(X,\mathcal{M}) \to (Y,\mathcal{N})$ of fine log schemes.

\begin{defn}
A log Witt lift of  $(X,\mathcal{M})$ over $(Y,\mathcal{N})$ is a system
$((X_n,\mathcal{M}_n),\Delta_n:W_n(X,\mathcal{N})\to (X_n,\mathcal{N}_n))_{n\ge 1}$ satisfying the following conditions.

(1) For each $n\ge 1$, $(X_n,\mathcal{M}_n)$ is log smooth  over $W_n(Y,\mathcal{N})$, and
\[
	W_n(Y,\mathcal{N})\times_{W_{n+1}(Y,\mathcal{N})}(X_{n+1},\mathcal{M}_{n+1})
	\simeq
	(X_n,\mathcal{M}_n),\
	(X_1,\mathcal{M}_1)=(X,\mathcal{M}).
\]

(2) Let $w_0:(X,\mathcal{M})\to W_n(X,\mathcal{M})$ be the morphism
induced by the Witt polynomial $w_0: X\to W_n(X)$ and $\text{id}_{\mathcal{M}}$.
For $n> 1$, $\Delta_n w_0$ is the natural map
$(X,\mathcal{M})\to (X_n,\mathcal{M}_n)$ and the following diagram commutes:
\[
\xymatrix{
	W_n(X,\mathcal{M}) \ar[r]^{\Delta_n} \ar[d]
	&
	(X_n,\mathcal{M}_n) \ar[d]
	\\
	W_{n+1}(X,\mathcal{M}) \ar[r]^{\Delta_{n+1}}
	&
	(X_{n+1},\mathcal{M}_{n+1}).
}
\]
\end{defn}

\begin{defn}
A log Frobenius lift of $(X,\mathcal{M})$ over $(Y,\mathcal{N})$ is a system
\[
	((X_n,\mathcal{M}_n),\Phi_n :(X_{n-1},\mathcal{M}_{n-1})
	\to
	(X_n,\mathcal{M}_n), \Delta_n:W_n(X,\mathcal{M})
	\to
	(X_n,\mathcal{M}_n))_{n\ge 1},
\]
satisfying the following conditions:

(1) $((X_n,\mathcal{M}_n),\Delta_n)$ is a log Witt lift of
$(X,\mathcal{M})$ over $(Y,\mathcal{N})$.

(2) For $n\ge 1$, $\Phi_n$ is compatible with
the Frobenius on the log Witt scheme $F:W_{n-1}(Y,\mathcal{N})\to W_n(Y,\mathcal{N})$, the absolute Frobenius
$\text{Frob} : X \otimes \mathbb{F}_p \to X \otimes \mathbb{F}_p$
and $\times p : \mathcal{M} \to \mathcal{M}$.

(3) The following diagram commutes:
\[
\xymatrix{
	W_n(X,\mathcal{M}) \ar[r]^{\Delta_n} \ar[d]_F
	&
	(X_n,\mathcal{M}_n) \ar[d]^{\Phi_{n+1}}
	\\
	W_{n+1}(X,\mathcal{M}) \ar[r]^{\Delta_{n+1}}
	&
	(X_{n+1},\mathcal{M}_{n+1}).
}
\]
\end{defn}

\begin{lem}
\label{lem:existance-of-witt-lift}
(1) Let $(R,P)\to (S,Q)$ be a log smooth morphism of pre-log rings.
Then $(S,Q)$ has a log Frobenius lift over $(R,P)$.

(2) Let $(X,\mathcal{M})\to (Y,\mathcal{N})$ be a log smooth morphism of fine log schemes.
Then \'etale locally on $X$, $(X,\mathcal{M})$ has a log Frobenius lift over $(Y,\mathcal{N})$.
\end{lem}

\begin{proof}
By the toroidal characterization of the log smoothness of log schemes (Theorem \ref{thm:toroidal-characterization-log-smoothness}), (2) follows from (1). We show (1).

The morphism $(R,P)\to (S,Q)$ has a decomposition
$(R,P)
\to
(R\otimes_{\mathbb{Z}[P]}\mathbb{Z}[Q],Q)
\to
(S,Q)$.
Since $(S,Q)$ is log smooth over $(R,P)$,
the ring map $R \otimes_{\mathbb{Z}[P]} \mathbb{Z}[Q] \to S$ is \'etale.

First we construct a log Frobenius map on $(T:=R\otimes_{\mathbb{Z}[P]}\mathbb{Z}[Q],Q)$ over $(R,P)$.
Let $\alpha : P\to R$ be the structure morphism of the pre-log ring $(R,P)$.
Set $T_n:=W_n(R)\otimes_{\mathbb{Z}[P]}\mathbb{Z}[Q]$
where the structural morphism $\mathbb{Z}[P] \to W_n(R)$ is induced by
$a \in P \to [\alpha(a)]$.
Then $(T_n,Q_n:=Q)$ will be a pre-log ring in the obvious way.
In particular, $(T_n,Q_n)$ is log smooth over $W_n(R,P)$.
We extend $F:W_n(R) \to W_{n-1}(R)$ to a morphism
\[
	\phi_n:
	(T_n, Q_n)
	\to
	(T_{n-1},Q_{n-1}),\
	a\otimes b
	\mapsto
	{}^Fa\otimes b^p,\
	a\in W_n(R),
	b\in Q.
\]
and also define $\delta_n:(T_n,Q_n) \to W_n(T,Q)$ induced by
$T_n \to W_n(T) ; a\in Q \mapsto [1 \otimes a], \text{id}_Q : Q_n = Q \to Q$.
Then $((T_n,Q_n), \phi_n, \delta_n)$ is a log Frobenius lift of $(T,Q)$.

To obtain a log Frobenius lift on $(S,Q)$,
it is suffice to show that if $(S,Q)\to (S',Q)$ is a morphism of pre-log rings
such that the underlying ring map $S\to S'$ is an \'{e}tale morphism
and the underlying monoid map $Q\to Q$ is the identity map
and there a log Frobenius lift $((S_n,Q_n),\phi_n,\delta_n)$ of $(S,Q)$,
there is a unique log Frobenius lift of the form $((S'_n,Q_n),\psi_n,\epsilon_n)$ of $(S',Q)$
and $(S,Q)\to (S',Q)$ lifts to a homomorphism
$((S_n,Q_n),\phi_n,\delta_n) \to ((S'_n,Q_n),\psi_n,\epsilon_n)$.
We can prove this in the same manner as the proof of \cite{LZDRW} Proposition 3.2.
\end{proof}

\section{Comparison morphism}
\label{sect:comparison-morphism}

We construct the comparison morphism between the log crystalline cohomology
and the hypercohomology of the log de Rham-Witt complex.

\subsection{Extension of derivations}
\label{subsect:extension-of-derivation}
In this subsection, we consider the log version of the discussion in \cite{ICrys} 0, \S 3.1.
First we recall the definition of the trivial extension of a quasi-coherent sheaf (\cite{Og} Example 2.1.6).

\begin{defn}
Let $f:X \to Y$ be a morphism of fine log schemes
and $E$ a quasi-coherent sheaf of $\mathcal{O}_X$-modules.

The trivial $Y$-extension of $X$ by $E$ is
the log scheme $T$ defined by $\mathcal{O}_T := \mathcal{O}_X \oplus E$
with $(a,b) (a',b') := (aa',ab'+a'b),$ with $\mathcal{M}_T := \mathcal{M}_X \oplus E$,
and $\alpha_T(m,e) := (\alpha_X(m), \alpha_X(m)e)$
if $m \in \mathcal{M}_X$ and $e \in E$.
The canonical projection $\mathcal{O}_T \to \mathcal{O}_X$
(resp. the canonical map $\mathcal{O}_Y\to \mathcal{O}_X \to \mathcal{O}_T$)
defines a morphism of log schemes $X\to T$ (resp. $T\to Y$).
We also have an evident retraction $T\to X$ over $Y$.
\end{defn}

Let $(Y,\mathcal{N},\mathcal{I},\gamma)$ be a fine log pd-scheme.
Let $i:(X,\mathcal{M})\to (X',\mathcal{M}')$ be a closed immersion of log schemes.
We assume $\gamma$ extends to $X$ and $i$ has a factorization $(X,\mathcal{M})\xrightarrow{j} (Z,\mathcal{L})\xrightarrow{g} (X',\mathcal{M}')$
with $j$ an exact closed immersion and $g$ log \'etale.
We admit this kind of factorizations \'etale locally on $X$ (\cite{KatoLog} (4.10) (1)).

Set $\mathcal{J}:=\ker (\mathcal{O}_Z\to j_*\mathcal{O}_X)$.
Then the log pd-envelope $(D,\mathcal{M}_D, \overline{\mathcal{J}},[\ ])$ of $i$
is the usual pd-envelope $(D, \overline{\mathcal{J}},[\ ])$ of $X$ in $Z$
with log structure $\mathcal{M}_D$ given by the inverse image of $\mathcal{L}$.
Since $g$ is log \'etale, the canonical morphism
$g^*\Lambda_{(X',\mathcal{M}')/(Y,\mathcal{N})}^1
\to
\Lambda_{(Z,\mathcal{L})/(Y,\mathcal{N})}^1$
is an isomorphism (\cite{KatoLog} Proposition (3.12)).

\begin{prop}
The log derivation
$(d,d\log ):(\mathcal{O}_{X'},\mathcal{M}')
\to
\Lambda_{(X',\mathcal{M}')/(Y,\mathcal{N})}^1$
extends uniquely to
\[
	(d',d'\log ):
	(\mathcal{O}_{D},\mathcal{M}_D)
	\to
	\mathcal{O}_{D}\otimes_{\mathcal{O}_{X'}} \Lambda_{(X',\mathcal{M}')/(Y,\mathcal{N})}^1
	\simeq
	\mathcal{O}_{D}\otimes_{\mathcal{O}_{Z}} \Lambda_{(Z,\mathcal{L})/(Y,\mathcal{N})}^1
\]
such that $d'x^{[n]}=x^{[n-1]}\otimes dx$ for all $x\in \mathcal{J}, n\ge 1$
and $d'\log m=1\otimes d\log m$ for all $m\in \mathcal{L}$.
\end{prop}

\begin{proof}
Let $E:=\mathcal{O}_{D}\otimes_{\mathcal{O}_{X'}} \Lambda_{(X',\mathcal{M}')/(Y,\mathcal{N})}^1$
and $T=\Spec (\mathcal{O}_D\oplus E, \mathcal{M}_D\oplus E)$
be the trivial $Z$-extension of $D$ by $E$.
We define a pd-structure on a ideal $E\subset \mathcal{O}_T$ by $u^{[n]}=0$ for $n\ge 2$.

Since $\mathcal{O}_T$ is an augmented $\mathcal{O}_D$-algebra and $E$ is an augmented ideal,
there exists a unique pd-structure $\delta$ on
$\overline{\mathcal{J}} \cdot \mathcal{O}_T + E \subset \mathcal{O}_T$
which is compatible with the pd-structures on $\overline{\mathcal{J}}$
and $E$ by (\cite{B} I 1.6.5).
$\delta$ satisfies $\delta_n(x+u)=x^{[n]}+x^{[n-1]}u$ for $x\in \overline{\mathcal{J}}, u \in E$.
By the construction, $\delta$ is compatible with $\gamma$.
Let $\alpha : \mathcal{O}_Z \to \mathcal{O}_T = \mathcal{O}_D \oplus E $
(resp. $\beta : \mathcal{L} \to \mathcal{M}_T$)
be a morphism defined by $\alpha(z)=(z, 1\otimes dz)$
(resp. $\beta(e)=(e,1\otimes d\log e)$.
They define a morphism
$\eta_0 = (\alpha,\beta):(\mathcal{O}_Z,\mathcal{L})\to (\mathcal{O}_T,\mathcal{M}_T)$.
By the universal property of the log pd-envelope,
$\eta_0$ induces an $\mathcal{O}_Y$-pd-morphism
$\eta:(\mathcal{O}_D,\mathcal{M}_D)\to (\mathcal{O}_T,\mathcal{M}_T)$.
We see that this morphism is a section of the canonical projection map
$(\mathcal{O}_T,\mathcal{M}_T)\to (\mathcal{O}_D,\mathcal{M}_D)$.
The morphisms $d':\mathcal{O}_D \xrightarrow{\eta} \mathcal{O}_T \xrightarrow{\text{pr}} E$
and $d'\log: \mathcal{M}_D \xrightarrow{\eta} \mathcal{M}_T \xrightarrow{\text{pr}} E$ define a log derivation
\[
	(d',d'\log ):
	(\mathcal{O}_{D},\mathcal{M}_D)
	\to
	\mathcal{O}_{D}\otimes_{\mathcal{O}_{X'}} \Lambda_{(X',\mathcal{M}')/(Y,\mathcal{N})}^1
	\simeq
	\mathcal{O}_{D}\otimes_{\mathcal{O}_{Z}} \Lambda_{(Z,\mathcal{L})/(Y,\mathcal{N})}^1
\]
 such that $d'x^{[n]}=x^{[n-1]}\otimes dx$ for all $x\in \mathcal{J}$ and $n\ge 1$, and $d'\log m=1\otimes d\log m$ for all $m\in \mathcal{L}$. Uniqueness is easy.
\end{proof}

The log derivation extends to a graded algebra $\mathcal{O}_{D}\otimes_{\mathcal{O}_{X'}} \Lambda_{(X',\mathcal{M}')/(Y,\mathcal{N})}^\bullet$.
We denote by $\breve{\Lambda}_{(D,\mathcal{M}_D)/(Y,\mathcal{N})}^\bullet$
the log pd de Rham complex of $(D,\mathcal{M}_D)$ over $(Y,\mathcal{N})$
with respect to the pd-structure $[\ ]$ on $(D, \overline{\mathcal{J}})$.
The universal property of the log pd de Rham complex induces a map
$\breve{\Lambda}_{(D,\mathcal{M}_D)/(Y,\mathcal{N})}^\bullet \to \mathcal{O}_{D}\otimes_{\mathcal{O}_{X'}} \Lambda_{(X',\mathcal{M}')/(Y,\mathcal{N})}^\bullet$
of $\mathcal{O}_Y$-algebras.
This map is isomorphism by the same proof to \cite{ICrys} Proposition 0.3.1.6.

\subsection{Comparison morphism}
\label{subsect:comparison-morphism}
Let $R$ be a $\mathbb{Z}_{(p)}$-algebra, in which $p$ is nilpotent.
Let $(X,\mathcal{M}) \to \Spec(R,P)$ be a morphism of fine log schemes
and we assume that the pd-structure of $W(R)$ extends to $X$.
We have the natural morphism
\[
	u_m:
	((X,\mathcal{M})/W_m(R,P))_{\crys}^{\text{log}}
	\to
	X_{\et}
\]
from the log crystalline topos to the \'{e}tale topos.
We write the structure sheaf of the log crystalline site $\mathcal{O}_{(X,\mathcal{M})/W_m(R,P)}$ as $\mathcal{O}_m$.

Define a morphism
\[
	\mathbb{R}{u_m}_*\mathcal{O}_m
	\to
	W_m\Lambda_{(X,\mathcal{M})/(R,P)}^\bullet
\]
in the derived category $\text{D}^{+}(X,W_m(R))$ of sheaves of $W_m(R)$-modules on $X_{\text{\'{e}t}}$ as follows:

First, we consider the case that $(X,\mathcal{M})$ has an embedding into a log smooth scheme $(Y,\mathcal{N})$ over $(R,P)$
such that $(Y,\mathcal{N})$ has a log Witt lift $((Y_m,\mathcal{N}_m),\Delta_m)$.
We already know such embedding exists \'{e}tale locally on $X$ by \cite{HK} (2.9.2) and Lemma \ref{lem:existance-of-witt-lift}.
There exists the following commutative diagram:
\[
\xymatrix{
	(X,\mathcal{M}) \ar[r] \ar[d]_{w_0}
	&
	(Y,\mathcal{N}) \ar[d]_{w_0} \ar [r]
	&
	(Y_m,\mathcal{N}_m)
	\\
	W_m(X,\mathcal{M}) \ar[r]
	&
	W_m(Y,\mathcal{N}) \ar[ur]_{\Delta_m}.
}
\]
The left vertical arrow $w_0 : (X,\mathcal{M}) \to W_m(X,\mathcal{M})$ defines
a log pd-thickening relative to the canonical pd-structure on $^V W_m(R)$.

Then the morphism $W_m(X,\mathcal{M})\to (Y_m,\mathcal{N}_m)$ factors through a morphism
\[
	\mu_m :
	W_m(X,\mathcal{M})
	\to
	(\overline{Y}_m,\overline{\mathcal{N}}_m),
\]
where $(\overline{Y}_m,\overline{\mathcal{N}}_m)$ is the log pd-envelope of the closed immersion $(X,\mathcal{M})\to (Y_m,\mathcal{N}_m)$
with respect to the canonical pd-structure on $^V W_m(R)$.
Then we have an isomorphism in $\mathrm{D}^{+}(X,W_m(R))$
\[
	\mathbb{R}u_{m*} \mathcal{O}_m
	\to
	\mathcal{O}_{\overline{Y}_m} \otimes_{\mathcal{O}_{Y_m}} \Lambda_{(Y_m,\mathcal{N}_m)/W_m(R,P)}^\bullet.
\]
Since $X \to \overline{Y}_m$ is a nilimmersion,
we can consider the right hand side as a sheaf on $X_{\et}$.

By the discussion in \S \ref{subsect:extension-of-derivation}, we have an isomorphism
\[
	\breve{\Lambda}_{(\overline{Y}_m,\overline{\mathcal{N}}_m)/W_m(R,P)}^\bullet
	\simeq
	\mathcal{O}_{\overline{Y}_m} \otimes_{\mathcal{O}_{Y_m}}\Lambda_{(Y_m,\mathcal{N}_m)/W_m(R,P)}^\bullet.
\]

We define the comparison morphism as follows:
\[
\xymatrix{
	\mathcal{O}_{\overline{Y}_m} \otimes_{\mathcal{O}_{Y_m}}\Lambda_{(Y_m,\mathcal{N}_m)/W_m(R,P)}^\bullet \ar@{-->}[r] \ar[d]_{\sim}
	&
	W_m\Lambda_{(X,\mathcal{M})/(R,P)}^\bullet
	\\
	\breve{\Lambda}_{(\overline{Y}_m,\overline{\mathcal{N}}_m)/W_m(R,P)}^\bullet \ar[r]^{\mu_m}
	&
	\breve{\Lambda}_{W_m(X,\mathcal{M})/W_m(R,P)}^\bullet. \ar[u]
}
\]
One can show this comparison morphism is independent of embeddings
and Witt lifts using the fibered product argument in \cite{ICrys} II.1.1.

Next, we treat general cases.
Recall the definition of embedding system (\cite{HK} p.237) :
\begin{defn}
Let $f:(X,\mathcal{M}) \to (S,\mathcal{L})$ be a morphism of fine log schemes such that the underlying morphism $X\to S$ is locally of finite type,
an embedding system for $f$ is a pair of simplicial objects $(X^\bullet, \mathcal{M}^\bullet)$ and $(Z^\bullet, \mathcal{N}^\bullet)$ in the category of fine log schemes endowed with morphism
\[
	(X^\bullet, \mathcal{M}^\bullet)
	\to
	(X,\mathcal{M}),
	(X^\bullet, \mathcal{M}^\bullet)
	\to
	(Z^\bullet, \mathcal{N}^\bullet),
	(Z^\bullet, \mathcal{N}^\bullet)
	\to
	(S,\mathcal{L})
\]
satisfying the following conditions (i)-(iv).

(i) The diagram
\[
	\xymatrix{
	(X^\bullet, \mathcal{M}^\bullet) \ar[r] \ar[d]
	&
	(Z^\bullet, \mathcal{N}^\bullet) \ar[d]
	\\
	(X,\mathcal{M}) \ar[r]
	&
	(S,\mathcal{L}) }
\]
is commutative.

(ii) The morphism $X^\bullet\to X$ is a hypercovering for the \'{e}tale topology and $\mathcal{M}^i$ is the inverse image of $\mathcal{M}$ on $X^i$ for each $i\ge 0$.

(iii) Each $(Z^i, \mathcal{N}^i)\to (S,\mathcal{L})$ is log smooth.

(iv)  Each $(X^i, \mathcal{M}^i)\to (Z^i,\mathcal{N}^i)$ is closed immersion.
\end{defn}

Let $\{ X(i) \}_{i\in I}$ be an \'etale covering of $X$ such that each $X(i)$ can be embedded to a log smooth scheme $Y(i)$ which has a log Witt lift $\{ Y_m(i)\}_m$. Set
\begin{align*}
	X(i_1,\ldots,i_r)
	&:=X(i_1)\times_X \cdots \times_X X(i_r),\\
	Y_m(i_1,\ldots,i_r)
	&:= Y_m(i_1)\times_{W_m(R,P)} \cdots \times_{W_m(R,P)} Y_m(i_r).
\end{align*}
Then $X(i_1,\ldots,i_r) \to Y_m(i_1,\ldots,i_r)$ is closed immersion since $X$ is separated.
For $r\in \bN$, let
\begin{align*}
	X^r:=\coprod_{i_1,\ldots,i_r\in I} X(i_1,\ldots,i_r),\
	Y_m^r:=\coprod_{i_1,\ldots,i_r\in I} Y_m(i_1,\ldots,i_r).
\end{align*}
We get an embedding system $X^\bullet \to Y_m^\bullet$.
We denote by $\overline{Y}_m^\bullet$ the log pd-envelope with respect to this closed immersion.
Let $\theta : (X^\bullet)_{\text{\'{e}t}}^\sim \to X_{\text{\'{e}t}}^\sim$ be the natural augmentation morphism.

By the liftable case, we have a morphism
\[
	\mathcal{O}_{\overline{Y}_m^\bullet}\otimes_{\mathcal{O}_{Y_m^\bullet}}\Lambda_{(Y_m^\bullet,\mathcal{N}_m^\bullet)/W_m(R,P)}^\bullet
	\to
	W_m\Lambda_{(X^\bullet,\mathcal{M}^\bullet)/(R,P)}^\bullet.
\]

Applying $\mathbb{R}{u_m}_*$ to both sides, we get the comparison morphism
\[
	\mathbb{R}{u_m}_*\mathcal{O}_m
	\to
	W_m\Lambda_{(X,\mathcal{M})/(R,P)}^\bullet.
\]
This is because the canonical morphism
\[
	\mathbb{R}u_{m*}\mathcal{O}_m
	\to
	\mathbb{R}\theta_* (\mathcal{O}_{\overline{Y}_m^\bullet}\otimes_{\mathcal{O}_{Y_m^\bullet}}\Lambda_{(Y_m^\bullet,\mathcal{N}_m^\bullet)/W_m(R,P)}^\bullet)
\]
is quasi-isomorphism by \cite{HK} Proposition 2.20 and we have a natural isomorphism
\[
	\mathbb{R}\theta_* W_m\Lambda_{(X^\bullet,\mathcal{M}^\bullet)/(R,P)}^\bullet
	\simeq
	W_m\Lambda_{(X,\mathcal{M})/(R,P)}^\bullet
\]
from the \'etale base change property of log de Rham-Witt complexes.

We prove that the comparison morphism is compatible with the Frobenius structure (cf. \cite{LZDRW} Proposition 3.6).
Frobenius morphisms and multiplications by $p$
\[
\xymatrix{
	W_m(\mathcal{O}_X) \ar[r]^{F}
	&
	W_{m-1}(\mathcal{O}_X)
	&
	\mathcal{M} \ar[r]^{\times p}
	&
	\mathcal{M}
	\\
	W_m(R) \ar[r]^{F} \ar[u]
	&
	W_{m-1}(R) \ar[u],
	&
	P \ar[r]^{\times p} \ar[u]
	&
	P \ar[u]
}
\]
defines a map of log de Rham complexes
\[
	\Lambda_{(W_m(X),W_m(\mathcal{M}))/W_m(R,P)}^\bullet
	\to
	\Lambda_{(W_{m-1}(X),W_{m-1}(\mathcal{M}))/W_{m-1}(R,P)}^\bullet
\]
and it factors $\mathbb{F}:W_m\Lambda_{(X,\mathcal{M})/(R,P)}^\bullet \to W_{m-1}\Lambda_{(X,\mathcal{M})/(R,P)}^\bullet$.
We have $\mathbb{F} = p^j F$ on $W_m\Lambda_{(X,\mathcal{M})/(R,P)}^j$
because $d^F \xi=p {}^F d\xi$ for $\xi \in W_m(\mathcal{O}_X)$
and $d\log m^p = pd\log m$ for $m\in W_m(\mathcal{M})$.

Let $(X_0, \mathcal{M}_0) := (X,\mathcal{M})\times \mathbb{F}_p$
and $R_0 := R\otimes \mathbb{F}_p$.
Consider the commutative diagram:
\[
\xymatrix{
	X_0 \ar[r]^{\text{Frob}} \ar[d]
	&
	X_0 \ar[d]
	\\
	\Spec W_{m-1}(R) \ar[r]^{F}
	&
	\Spec W_m(R).
}
\]
It induces a map
\[
	\mathbb{R}\overline{u}_{m*}\mathcal{O}_{(X_0, \mathcal{M}_0)/W_m(R_0,P)}
	\to
	\mathbb{R}\overline{u}_{m-1*}\mathcal{O}_{(X_0, \mathcal{M}_0)/W_{m-1}(R_0,P)},
\]
where $\overline{u}_{m} :
((X_0, \mathcal{M}_0)/W_m(R_0,P))_{\crys}^{\log}
\to
(X_0)_{\et}=X_{\et}$ is the canonical morphism of topoi. We have a canonical isomorphism
\[
	\mathbb{R}\overline{u}_{m*}\mathcal{O}_{(X_0, \mathcal{M}_0)/W_m(R_0,P)}
	\to
	\mathbb{R}u_{m*}\mathcal{O}_{(X,\mathcal{M})/W_m(R,P)}.
\]
So we obtain $\mathbb{F}:
\mathbb{R}u_{m*}\mathcal{O}_m
\to
\mathbb{R}u_{m-1*}\mathcal{O}_{m-1}$.

\begin{prop}
We have a commutative diagram:
\[
\xymatrix{
	\mathbb{R}u_{m*}\mathcal{O}_m \ar[r] \ar[d]_{\mathbb{F}}
	&
	W_m\Lambda_{(X,\mathcal{M})/(R,P)}^\bullet \ar[d]^{\mathbb{F}}
	\\
	\mathbb{R}u_{m-1*}\mathcal{O}_{m-1} \ar[r]
	&
	W_{m-1}\Lambda_{(X,\mathcal{M})/(R,P)}^\bullet.
}
\]
\end{prop}

\begin{proof}
By the simplicial method as above, we can assume that $(X,\mathcal{M})$ is embedded in a log smooth scheme $(Y,\mathcal{N})$
which admits a log Frobenius lift $\{(Y_m,\mathcal{N}_m)\}_m$. Let $\Phi_m:(Y_{m-1},\mathcal{N}_{m-1})\to(Y_m,\mathcal{N}_m)$ be the given lift of the absolute Frobenius.
The map $\mathbb{F}:\mathbb{R}u_{m*}\mathcal{O}_m \to\mathbb{R}u_{m-1*}\mathcal{O}_{m-1}$ is represented by the map
\[
	\mathcal{O}_{\overline{Y}_m}\otimes_{\mathcal{O}_{Y_m}}\Lambda_{(Y_m,\mathcal{N}_m)/W_m(R,P)}^\bullet
	\to
	\mathcal{O}_{\overline{Y}_{m-1}}\otimes_{\mathcal{O}_{Y_{m-1}}}\Lambda_{(Y_{m-1},\mathcal{N}_{m-1})/W_{m-1}(R,P)}^\bullet
\]
which is induced by $\Phi_m$.

By the properties of log Frobenius lifts, we have a commutative diagram
\[
\xymatrix{
	\mathcal{O}_{\overline{Y}_m}\otimes_{\mathcal{O}_{Y_m}}\Lambda_{(Y_m,\mathcal{N}_m)/W_m(R,P)}^\bullet \ar[r] \ar[d]^{\mathbb{F}}
	&
	W_m\Lambda_{(X,\mathcal{M})/(R,P)}^\bullet \ar[d]^{\mathbb{F}}
	\\
	\mathcal{O}_{\overline{Y}_{m-1}}\otimes_{\mathcal{O}_{Y_{m-1}}}\Lambda_{(Y_{m-1},\mathcal{N}_{m-1})/W_{m-1}(R,P)}^\bullet \ar[r]
	&
	W_{m-1}\Lambda_{(X,\mathcal{M})/(R,P)}^\bullet
}
\]
and this is identified to the diagram in the proposition.
\end{proof}

\section{Comparison theorem}
\label{sect:comparison}
Let $R$ be a $\mathbb{Z}_{(p)}$-algebra such that $p$ is nilpotent in $R$.

\subsection{NCD case}
Let $Y$ be a log scheme over $S=\Spec R$.
We assume that the structure morphism $Y\to S$ has \'etale locally on $Y$ a decomposition
\[
	Y
	\xrightarrow{u}
	\Spec(A=R[T_1,\ldots,T_n],P=\bN^{e}\oplus \bN^f)
	\to
	\Spec R
\]
with $u$ exact and \'{e}tale (in the usual sense), $1 \le e \le n, f \ge 0$,
and $(A,P)$ is the pre-log ring we discussed in \S \ref{subsect:log-basic-witt-diffs-sncd}.
If $X$ is a smooth scheme and $D$ is a normal crossing divisor on $X$,
the log scheme $(X,D)$ satisfies this condition.

First we consider the log scheme $X:=\Spec(A=R[T_1,\ldots,T_n], P=\bN^{e}\oplus \bN^f)$.
Let $(A_m=W_m(R)[T_1,\ldots,T_n], P)$ be the pre-log ring of the type we discussed in \S \ref{subsect:log-basic-witt-diffs-sncd}.
Then $X_m:=\Spec(A_m, P)$ is a lift of $X$ over $W_m(R)$.
Let $\phi_m:A_{m+1}\to A_m$ be the morphism defined by $F:W_{m+1}(R)\to W_m(R)$ and $T_i\to T_i^p$ for $1 \le i \le n$.
The morphism $\phi_m$ and the multiplication by $p$ morphism $\times p:P \to P$ define a morphism of log schemes $\Phi_m:X_m \to X_{m+1}$.
Let $\delta_m:A_m \to W_m(A)$ be the morphism induced by the canonical morphism $W_m(R)\to W_m(A)$ and $T_i\mapsto [T_i]$ for $1 \le i \le n$.
We denote by $\Delta_m : W_m(X)\to X_m$ the morphism corresponding to $\delta_m$ and the identity morphism on $P$.
The pair $(X_m, \Delta_m)$ defines the comparison morphism
\[
	\Lambda_{X_m/W_m(R)}^\bullet
	\to
	W_m\Lambda_{X/R}^\bullet.
\]
If $f=0$, the family $(X_m, \Phi_m, \Delta_m)_m$ is a log Frobenius lift of $X$
and the comparison morphism coincides with the morphism induced by
the log Witt lift $(X_m, \Delta_m)_m$ in \S \ref{sect:comparison-morphism}.

\begin{thm}
\label{thm:ncd-comarison-general}
The comparison morphism
\[
	\Lambda_{X_m/W_m(R)}^\bullet
	\to
	W_m\Lambda_{X/R}^\bullet
\]
is a quasi-isomorphism. This morphism is functorial.
\end{thm}
\begin{proof}

We prove that the canonical comparison morphism of complexes
\[
	\Lambda_{(A_m,P)/W_m(R)}^\bullet
	\to
	W_m\Lambda_{(A,P)/R}^\bullet
\]
is a quasi-isomorphism.
We have a decomposition of complexes
\[
	W_m\Lambda_{(A,P)/R}^\bullet
	=W_m\Lambda_{(A,P)/R}^{\text{int},\bullet}\oplus W_m\Lambda_{(A,P)/R}^{\text{frac},\bullet},
\]
where $W_m\Lambda_{(A,P)/R}^{\text{int},\bullet}$
(resp. $W_m\Lambda_{(A,P)/R}^{\text{frac},\bullet})$
is the integral part (resp. the fractional part).
By the formula of derivation on log basic Witt differentials given in \S \ref{sect:basicwittdiffs},
we see that $W_m\Lambda_{(A,P)/R}^{\text{frac},\bullet}$ is acyclic.
The comparison morphism maps $\Lambda_{(A_m,P)/W_m(R)}^\bullet$ isomorphically to the complex $W_m\Lambda_{(A,P)/R}^{\text{int},\bullet}$
because the comparison map sends the log $p$-basic differential
\[
	\left(\prod _{i\in J} d\log c_i\right)
	\cdot
	\left(\prod _{i\in I_{-\infty}} d\log T_i\right)
	\cdot T^{k_{I_0}}(p^{-\text{ord}_p{k_{I_1}}}dT^{k_{I_1}})\cdots (p^{-\text{ord}_p{k_{I_l}}}dT^{k_{I_l}})
\]
to the following log basic Witt differential:
\[
	\left(\prod _{i\in J} d\log c_i\right)
	\cdot \left(\prod _{i\in I_{-\infty}} d\log X_i\right)
	\cdot X^{k_{I_0}}({}^{F^{-t(I_1)}}dX^{p^{t(I_1)}k_{I_1}}) \cdots ({}^{F^{-t(I_l)}}dX^{p^{t(I_l)}k_{I_l}}).
\]
\end{proof}

\begin{thm}
\label{thm:comparison-theorem-NCD}
Let $Y$ be a smooth scheme over $R$ and $D$ be a normal crossing divisor of $Y$. Then the canonical homomorphism
\[
	\mathbb{R}u_{(Y,D)/W_m(R)*}\mathcal{O}_{(Y,D)/W_m(R)}
	\to
	W_m\Lambda_{(Y,D)/R}^\bullet
\]
is an isomorphism in $D^+(Y,W_m(R))$.
Moreover, if $R$ is Noetherian and $Y$ is proper over $R$, we have a canonical isomorphism
\[
	H_{\logcrys}^* ((Y,D)/W(R))
	\to
	\mathbb{H}_{\et}^*(Y,W\Lambda_{(Y,D)/R}^\bullet).
\]
\end{thm}

\begin{proof}
Using the similar method of \cite{LZDRW} Theorem 3.5,
we may assume that $(Y,D)=\Spec(A=R[T_1,\ldots,T_n],P=\bN^{e})$,
where the log structure is given by $e_i\mapsto T_i$.
There is the canonical log Frobenius lift $(\Spec(W_m(R)[T_1,\ldots,T_n],P), \Phi_m, \Delta_m)_m$ of $(Y,D)$.
Since the pre-log ring $(A,P)$ is log smooth over $R$,
the comparison morphism of \S \ref{subsect:log-basic-witt-diffs-sncd} becomes the map
$\Lambda_{(A_m,P)/W_m(R)}^\bullet
	\to
	W_m\Lambda_{(A,P)/R}^\bullet$.
Hence the first claim follows from Theorem \ref{thm:ncd-comarison-general}.
The proof of the second claim is similar to that of Theorem \ref{thm:infinitecomparison}.
\end{proof}

\subsection{Semistable case}
\label{subsect:comparison-semistable}
We prove the comparison theorem for semistable log schemes.
Let $(Y,\mathcal{M})$ be a log scheme over $S=\Spec (R,\bN)$ of the following type:

\'Etale locally on $Y$, the structure morphism $Y\to S$ has a decomposition
\[
	Y
	\xrightarrow{u}
	\Spec(A=R[T_1,\ldots,T_n]/(T_1\cdots T_d),P=\bN^{e}\oplus \bN^f) \xrightarrow{\delta}
	S
\]
with $u$ exact and \'{e}tale (in the usual sense),
$1\le d \le e \le n, f \ge 0$, $(A,P)$ is the pre-log ring we discussed in \S \ref{subsect:log-basic-witt-diffs-semistable},
and $\delta$ is induced by the diagonal map $\bN \to \bN^e\oplus \bN^f$.
Obviously, semistable log schemes over $S$ (Definition \ref{defn:semistable} (2)) satisfy this condition. Set
\[
	W_m\widetilde{\Lambda}^\bullet
	= W_m\Lambda_{Y/(R,\{*\})}^\bullet,
	W_m\Lambda^\bullet
	=W_m\Lambda_{Y/(R,\bN)}^\bullet.
\]
Let $t_m \in W_m(\mathcal{M})$ be the image of the base of $\bN$ under the morphism
$\bN\to \mathcal{M} \to W_m(\mathcal{M})$ and $\theta_m=d\log t_m\in W_m\widetilde{\Lambda}^1$.

\begin{lem}
\label{lem:exactness}
We have the following exact sequence:
\begin{align*}
	0
	\to
	W_m(\mathcal{O}_Y) = W_m\widetilde{\Lambda}^0
	\xrightarrow{\wedge \theta_m}
	W_m\widetilde{\Lambda}^1
	\xrightarrow{\wedge \theta_m}
	W_m\widetilde{\Lambda}^2
	\xrightarrow{\wedge \theta_m}
	\cdots.
\end{align*}
\end{lem}

\begin{proof}
It is easy to see that this is a chain complex.
Since the question is local and using the \'{e}tale base change property,
we can assume that
\[
	Y=\Spec (A=R[T_1,\ldots,T_n]/(T_1\cdots T_d),P=\bN^{e}\oplus \bN^f).
\]
By Proposition \ref{prop:decomposition-of-drw-cpx-semistable} there exists a decomposition
\[
	W_m\Lambda_{(A,P)/R}^\bullet
	= WC_{(A,P)/R}^\bullet \oplus WC_{(A,P)/R}^{'\bullet}.
\]
Define $c: W_m\Lambda_{(A,P)/R}^\bullet \to W_m\Lambda_{(A,P)/R}^{\bullet-1}$ by $(a,b\wedge \theta_m)\mapsto b$.
It is easy to see $(\wedge \theta_m)\circ c+c\circ (\wedge \theta_m)=\text{id}$.
\end{proof}

\begin{lem}
\label{lem:comparelog}
The canonical morphism $W_m\widetilde{\Lambda}^\bullet \to W_m\Lambda^\bullet$ induces an isomorphism
\[
	W_m\widetilde{\Lambda}^\bullet/(W_m\widetilde{\Lambda}^{\bullet-1}\wedge \theta_m)
	\simeq
	W_m\Lambda^\bullet.
\]
we have an exact sequence:
\[
	0
	\to
	W_m\Lambda^{\bullet-1}
	\xrightarrow{\wedge \theta_m}
	W_m\widetilde{\Lambda}^\bullet
	\to
	W_m\Lambda^\bullet
	\to
	0.
\]
\end{lem}
\begin{proof}
It can be easily seen that the surjective morphism
$W_m\widetilde{\Lambda}^\bullet \to W_m\Lambda^\bullet$ factors
$W_m\widetilde{\Lambda}^\bullet/(W_m\widetilde{\Lambda}^{\bullet-1}\wedge \theta_m)$
and $\{ W_m\widetilde{\Lambda}^\bullet/(W_m\widetilde{\Lambda}^{\bullet-1}\wedge \theta_m) \}$ is
a log $F$-$V$-procomplex over $(Y,\mathcal{M}_Y)/(R,\bN)$.
This implies the canonical surjective map
\[
	\Lambda_{W_m(Y,\mathcal{M}_Y)/W_m(R,\{*\})}^\bullet
	\to
	W_m\widetilde{\Lambda}^\bullet/(W_m\widetilde{\Lambda}^{\bullet-1}\wedge \theta_m)
\]
factors $\Lambda_{W_m(Y,\mathcal{M}_Y)/W_m(R,\bN)}^\bullet$.

Let $\{ E_m^\bullet \}$ be any log $F$-$V$-procomplex over $(Y,\mathcal{M})/(R,\bN)$.
There is a morphism $\{ W_m\widetilde{\Lambda}^\bullet/(W_m\widetilde{\Lambda}^{\bullet-1}\wedge \theta_m)\} \to \{ E_m^\bullet \}$ of log $F$-$V$-procomplexes obtained by the composition
\[
	\{ W_m\widetilde{\Lambda}^\bullet/(W_m\widetilde{\Lambda}^{\bullet-1}\wedge \theta_m)\}
	\to
	\{ W_m\Lambda^\bullet \}
	\to
	\{ E_m^\bullet \},
\]
where the second arrow is induced by the universal property of $\{ W_m\Lambda^\bullet \}$.
Moreover, it is unique morphism that fits into the following diagram
\[
\xymatrix{
	\{ \Lambda_{W_m(Y,\mathcal{M})/W_m(R,\{*\})}^\bullet \} \ar[r] \ar[rd]
	&
	\{ W_m\widetilde{\Lambda}^\bullet/(W_m\widetilde{\Lambda}^{\bullet-1}\wedge \theta_m) \}  \ar[d]
	\\
	&
	\{ E_m^\bullet \}
}
\]
because the top arrow is surjective.
Hence we proved $\{ W_m\widetilde{\Lambda}^\bullet/(W_m\widetilde{\Lambda}^{\bullet-1}\wedge \theta_m) \}$ has the universal property
and $\{ W_m\widetilde{\Lambda}^\bullet/(W_m\widetilde{\Lambda}^{\bullet-1}\wedge \theta_m)\} \to \{ W_m\Lambda^\bullet \}$ is an isomorphism.
The second claim follows from the isomorphism and Lemma \ref{lem:exactness}.
\end{proof}

Let $X=X_{d,e,n,f}:=\Spec (R[T_1,\ldots,T_n]/(T_1\cdots T_d),\bN^{e}\oplus \bN^f)$ be
the log scheme corresponding to the pre-log ring of \S \ref{subsect:log-basic-witt-diffs-semistable}
over $S=\Spec (R,\bN)$ for $1\le d \le e \le n, f\ge 0$.
Consider the closed subschemes $Z_1:=V(T_1\cdots T_{d-1})$ and $Z_2:=V(T_d)$
and $Z=Z_1\cap Z_2$ of $X$ endowed with the inverse image log structure of $X$.
We find $Z_1\simeq X_{d-1,e,n,f},Z_2\simeq X_{1,e,n,f}, Z\simeq X_{d-1,e-1,n-1,f+1}$.
If $\mathcal{I}_i:=\ker (\mathcal{O}_X\to \mathcal{O}_{Z_i})\ (i=1,2)$,
we see $\mathcal{I}_1+\mathcal{I}_2=\ker (\mathcal{O}_X\to \mathcal{O}_{Z})$ and $\mathcal{I}_1 \cap \mathcal{I}_2=0$.
Then we get the following exact sequence of $\mathcal{O}_X$-modules:
\[
	0
	\to
	\mathcal{O}_X\to \mathcal{O}_X/\mathcal{I}_1 \oplus \mathcal{O}_X/\mathcal{I}_2
	\to
	\mathcal{O}_X/(\mathcal{I}_1 + \mathcal{I}_2)
	\to
	0.
\]
The log differential sheaf $\Lambda_{X/S}^1$ is a free $\mathcal{O}_X$-module
because $\Lambda_{X/S}^1$ is the quotient of
\[
	\bigoplus_{i=1}^e \mathcal{O}_X d\log T_i
	\oplus \bigoplus_{i=e+1}^n \mathcal{O}_X dT_i
	\oplus \bigoplus_{i=1}^f \mathcal{O}_X d\log c_i
\]
divided by the submodule generated by $d\log T_1 +\cdots +d\log T_e + d\log c_1 +\cdots +d\log c_f$.
Hence we obtain the exact sequence
\begin{align*}
	0
	\to
	\Lambda_{X/S}^\bullet
	\to
	(\mathcal{O}_X/\mathcal{I}_1\otimes_{\mathcal{O}_X}\Lambda_{X/S}^\bullet) \oplus &(\mathcal{O}_X/\mathcal{I}_2\otimes_{\mathcal{O}_X}\Lambda_{X/S}^\bullet)
	\\
	&\to \mathcal{O}_X/(\mathcal{I}_1 + \mathcal{I}_2)\otimes_{\mathcal{O}_X}\Lambda_{X/S}^\bullet
	\to
	0.
\end{align*}

Since closed immersions $Z_i\hookrightarrow X\ (i=1,2)$
and $Z\hookrightarrow  X$ are exact closed immersions defined by
a coherent sheaf of ideals of $\mathcal{M}_X$,
canonical morphisms $\mathcal{O}_X/\mathcal{I}_i\otimes_{\mathcal{O}_X}\Lambda_{X/S}^\bullet \to \Lambda_{Z_i/S}^\bullet \ (i=1,2)$
and $\mathcal{O}_X/(\mathcal{I}_1+\mathcal{I}_2)\otimes_{\mathcal{O}_X}\Lambda_{X/S}^\bullet \to \Lambda_{Z/S}^\bullet$ are isomorphisms (Chapter IV, Corollary 2.3.3 of \cite{Og}).
So we obtain the exact sequence
\[
	0
	\to
	\Lambda_{X/S}^\bullet
	\to
	\Lambda_{Z_1/S}^\bullet \oplus \Lambda_{Z_2/S}^\bullet
	\to
	\Lambda_{Z/S}^\bullet
	\to
	0.
\]

We prove the existence of Mayer-Vietoris exact sequences for the de Rham-Witt complex in semistable cases.
We write $W_m\Lambda_{X/(R,\{*\})}^\bullet$ as $W_m\widetilde{\Lambda}_{X/S}^\bullet$, and so on.
Define $\widetilde{\mathcal{K}}_i:=W_m(\mathcal{I}_i)W_m\widetilde{\Lambda}_{X/S}^\bullet + dW_m(\mathcal{I}_i)W_m\widetilde{\Lambda}_{X/S}^{\bullet-1}\subset W_m\widetilde{\Lambda}_{X/S}^\bullet \ (i=1,2).$
From the fact $W_m(\mathcal{I}_1)+W_m(\mathcal{I}_2)=W_m(\mathcal{I}_1+\mathcal{I}_2)$, $\widetilde{\mathcal{K}}_1+\widetilde{\mathcal{K}}_2$ is equal to
\[
	W_m(\mathcal{I}_1+\mathcal{I}_2)W_m\widetilde{\Lambda}_{X/S}^\bullet + d W_m(\mathcal{I}_1+\mathcal{I}_2)W_m\widetilde{\Lambda}_{X/S}^{\bullet-1}.
\]
Then from Proposition \ref{prop:exact-sequence-Kahler-diffs} (2)
we get $W_m\widetilde{\Lambda}_{Z_i/S}^\bullet \simeq W_m\widetilde{\Lambda}_{X/S}^\bullet/\widetilde{\mathcal{K}}_i \ (i=1,2)$
and $W_m\widetilde{\Lambda}_{Z/S}^\bullet \simeq W_m\widetilde{\Lambda}_{X/S}^\bullet/(\widetilde{\mathcal{K}}_1+\widetilde{\mathcal{K}}_2)$.

\begin{lem}
\label{lem:exacttriviallog}
The following sequence is exact:
\[
	0
	\to
	W_m\widetilde{\Lambda}_{X/S}^\bullet
	\to
	W_m\widetilde{\Lambda}_{Z_1/S}^\bullet \oplus W_m\widetilde{\Lambda}_{Z_2/S}^\bullet
	\to
	W_m\widetilde{\Lambda}_{Z/S}^\bullet
	\to
	0.
\]
\end{lem}
\begin{proof}
Since the sequence is identified to
\[
	0
	\to
	W_m\widetilde{\Lambda}_{X/S}^\bullet
	\to
	W_m\widetilde{\Lambda}_{X/S}^\bullet/\widetilde{\mathcal{K}}_1 \oplus W_m\widetilde{\Lambda}_{X/S}^\bullet/\widetilde{\mathcal{K}}_2
	\to
	W_m\widetilde{\Lambda}_{X/S}^\bullet/(\widetilde{\mathcal{K}}_1 + \widetilde{\mathcal{K}}_2)
	\to
	0,
\]
it suffices to show that the morphism $W_m\widetilde{\Lambda}_{X/S}^\bullet \to W_m\widetilde{\Lambda}_{Z_1/S}^\bullet\oplus W_m\widetilde{\Lambda}_{Z_2/S}^\bullet$ is injective.
Let $\omega \in W_m\widetilde{\Lambda}_{X/S}^\bullet$ be an element of the kernel of this morphism.
By Corollary \ref{prop:expression-basic-witt-semistable},
we see $\omega$ is uniquely written as a finite sum of log basic Witt differentials $\sum_{k,\mathcal{P},J}\epsilon_m(\xi_{k,\mathcal{P},J},k,\mathcal{P},J)$,
$\xi_{k,\mathcal{P},J}\in {}^{V^{u(k^+)}}W_{m-u(k^+)}(R)$
where $k$ runs through all weights such that $[1,d]\not\subset \Supp^+ k,$ $p^{m-1}\cdot k^+$ is integral
and $J$ runs through all subsets of $[1,f]$.
The image of $\omega$ in $W_m\widetilde{\Lambda}_{Z_1/S}^\bullet$ is the sum of log basic Witt differentials $\epsilon_m(\xi_{k,\mathcal{P},J},k,\mathcal{P},J)$
of $W_m\widetilde{\Lambda}_{Z_1/S}^\bullet$
such that $[1,d-1]\not\subset \Supp^+ k$.
Similarly, the image in $W_m\widetilde{\Lambda}_{Z_2/S}^\bullet$
is the sum of log basic Witt differentials $\epsilon_m(\xi_{k,\mathcal{P},J},k,\mathcal{P},J)$
of $W_m\widetilde{\Lambda}_{Z_2/S}^\bullet$ such that $d\notin \Supp^+ k$.
If we apply Corollary \ref{prop:expression-basic-witt-semistable} again to $Z_1$ (resp. $Z_2$),  we get $\xi_{k,\mathcal{P},J}=0$ for $k$ that satisfies $[1,d-1]\not\subset \Supp^+ k$ (resp. $d\notin \Supp^+ k$). We conclude $\omega=0$.
\end{proof}

\begin{prop}
The following sequence is exact:
\[
	0
	\to
	W_m\Lambda_{X/S}^\bullet
	\to
	W_m\Lambda_{Z_1/S}^\bullet\oplus W_m\Lambda_{Z_2/S}^\bullet
	\to
	W_m\Lambda_{Z/S}^\bullet
	\to
	0.
\]
\end{prop}
\begin{proof}
We prove by induction on the degree. Since we have an exact sequence
\[
	0
	\to
	W_m(\mathcal{O}_X)
	\to
	W_m(\mathcal{O}_X/\mathcal{I}_1) \oplus W_m(\mathcal{O}_X/\mathcal{I}_2)
	\to
	W_m(\mathcal{O}_X/(\mathcal{I}_1 + \mathcal{I}_2))
	\to
	0,
\]
the sequence is exact on degree zero.
From Lemma \ref{lem:comparelog} and Lemma \ref{lem:exacttriviallog}, we get the following exact commutative diagram
\[
\xymatrixcolsep{1pc}
\xymatrix{
	&
	&
	0 \ar[d]
	&
	&
	\\
	0 \ar[r]
	&
	W_m\Lambda_X^{i-1} \ar[r]^{\wedge \theta_m} \ar[d]
	&
	W_m\widetilde{\Lambda}_X^i \ar[r] \ar[d]
	&
	W_m\Lambda_X^i \ar[r] \ar[d]
	&
	0
	\\
	0 \ar[r]
	&
	W_m\Lambda_{Z_1}^{i-1}\oplus W_m\Lambda_{Z_2}^{i-1} \ar @/^/ @{{ }{ }}[r]^{\wedge \theta_m \oplus \wedge \theta_m} \ar[r] \ar[d]
	&
	W_m\widetilde{\Lambda}_{Z_1}^i \oplus W_m\widetilde{\Lambda}_{Z_2}^i \ar[r] \ar[d]
	&
	W_m\Lambda_{Z_1}^i \oplus W_m\Lambda_{Z_2}^i \ar[r] \ar[d]
	&
	0
	\\
	0 \ar[r]
	&
	W_m\Lambda_Z^{i-1} \ar[r]^{\wedge \theta_m}
	&
	W_m\widetilde{\Lambda}_Z^i \ar[r] \ar[d]
	&
	W_m\Lambda_Z^i \ar[r]
	&
	0.
	\\
	&
	&
	0
	&
	& \\
}
\]
Using this diagram and nine-lemma, the proposition follows by induction.
\end{proof}

For $X=\Spec(A=R[T_1,\ldots,T_n]/(T_1\cdots T_d),P=\bN^{e}\oplus \bN^f)$, we set $X_m$ by
\[
	X_m:=\Spec (A_m:=W_m(R)[T_1,\ldots,T_n]/(T_1\cdots T_d), P).
\]
There are morphisms of log schemes $\Phi_m:X_m \to X_{m+1}$
and $\Delta_m : W_m(X)\to X_m$ as the case of $\Spec(R[T_1,\ldots,T_n], \bN^{e}\oplus \bN^f)$.
The pair $(X_m, \Delta_m)$ defines the comparison morphism
\[
	\Lambda_{X_m/S_m}^\bullet
	\to
	W_m\Lambda_{X/R}^\bullet.
\]
If $d=e$ and $f=0$, the family $(X_m, \Phi_m, \Delta_m)_m$ is a log Frobenius lift of $X$
and the comparison morphism coincides with the morphism induced by
the log Witt lift $(X_m, \Delta_m)_m$ in \S \ref{sect:comparison-morphism}.

\begin{lem}
\label{lem:semistabledequalto1}
Assume $d=1$.
Let $X_m=\Spec (A_m=W_m(R)[T_1,\ldots,T_n]/(T_1), P=\bN^{e}\oplus \bN^f)$ be the canonical lift of $X$ over $S_m=\Spec (W_m(R),\bN)$.
Then the comparison morphism
\[
	\Lambda_{X_m/S_m}^\bullet
	\to
	W_m\Lambda_{X/S}^\bullet
\]
is a quasi-isomorphism.
\end{lem}

\begin{proof}
First, we consider two log structures on a ring $A=R[T_1,\ldots,T_n]/(T_1)$ over $(R,\bN)$.
The one is defined by
\begin{align*}
	\bN^{e} \oplus \bN^f \to A :
	\bN^{e} \ni e_1
	\mapsto 0,\
	e_i \mapsto T_i \ (i \neq 1),\
	\bN^{f} \ni c_i \mapsto 0
\end{align*}
and the diagonal morphism $\bN \to \bN^{e} \oplus \bN^f$.

The other one is given by a diagram
\[
\xymatrix{
	(1,0,0)
	&
	Q:=\bN\oplus \bN^{e-1}\oplus \bN^f \ar[r]
	&
	A
	\\
	1\ar@{|->}[u]
	&
	\bN \ar[r] \ar[u]
	&
	R, \ar[u]
}
\]
where the upper horizontal morphism is induced by
\[
	(0,e_i,0)
	\mapsto
	T_i,\
	(1,0,0)
	\mapsto
	1,\
	(0,0,c_i)
	\mapsto
	0.
\]

The morphism of monoids $Q \to \bN^{e} \oplus \bN^f$ defined by
\begin{align*}
	(1,0,0)
	&
	\mapsto
	(1,(1,\ldots,1),(1,\ldots,1)),
	\\
	(0,e_i,0)
	&
	\mapsto
	(0,e_i,0),\
	1\le i\le e-1,
	\\
	(0,0,c_i)
	&
	\mapsto
	(0,0,c_i),\
	1\le i \le f
\end{align*}
gives a map $(A,Q)\to (A,\bN^{e}\oplus \bN^f)$ of pre-log rings over $(R,\bN)$.

There is also another log structure on $A$ over $(R,\{*\})$,
which is given by $\bN^{e-1}\oplus \bN^f\to A$
that sends $e_i$ to $T_i$ for $1 \le i\le e-1.$ We get a diagram
\[
\xymatrix{
	(A,\bN^{e-1}\oplus \bN^f) \ar[r]
	&
	(A,Q) \ar[r]
	&
	(A,\bN^{e}\oplus \bN^f)
	\\
	(R,\{*\})\ar[u]\ar[r]
	&
	(R,\bN) \ar@{=}[r] \ar[u]
	&
	(R,\bN). \ar[u]
}
\]
We also have a similar diagram for $A_m=W_m(R)[T_1,\ldots,T_n]/(T_1)$.

They induce a diagram:
\[
\xymatrix{
	\Lambda_{(A_m,\bN^{e-1}\oplus \bN^f)/(W_m(R),\{*\})}^\bullet \ar[r]^-{\alpha_1} \ar[d]
	&
	\Lambda_{(A_m,Q)/(W_m(R),\bN)}^\bullet \ar[r]^-{\alpha_2} \ar[d]
	&
	\Lambda_{(A_m,\bN^{e}\oplus \bN^f)/(W_m(R),\bN)}^\bullet \ar[d]
	\\
	W_m\Lambda_{(A,\bN^{e-1}\oplus \bN^f)/(R,\{*\})}^\bullet \ar[r]^-{\beta_1}
	&
	W_m\Lambda_{(A,Q)/(R,\bN)}^\bullet \ar[r]^-{\beta_2}
	&
	W_m\Lambda_{(A,\bN^{e}\oplus \bN^f)/(R,\bN)}^\bullet.
}
\]
It is easy to see that $\alpha_1$ is an isomorphism. $\alpha_2$ is also an isomorphism
because the canonical morphism $Q^{\gp}\to \mathbb{Z}^e \oplus \mathbb{Z}^f$ induced by $Q\to \bN^{e}\oplus \bN^f$ is an isomorphism.
We also have isomorphisms
\[
	\Lambda_{(W_m(A),\bN^{e-1}\oplus \bN^f)/(W_m(R),\{*\})}^\bullet
	\simeq
	\Lambda_{(W_m(A),Q)/(W_m(R),\bN)}^\bullet
	\simeq
	\Lambda_{(W_m(A),\bN^{e}\oplus \bN^f)/(W_m(R),\bN)}^\bullet
\]
by the same reason.

By the construction of the log de Rham-Witt complexes,
$\beta_1$ and $\beta_2$ are also isomorphisms.
Hence we only have to show
\[
	\Lambda_{(A_m,\bN^{e-1}\oplus \bN^f)/(W_m(R),\{*\})}^\bullet
	\to
	W_m\Lambda_{(A,\bN^{e-1}\oplus \bN^f)/(R,\{*\})}^\bullet
\]
is a quasi-isomorphism, but this is Theorem \ref{thm:comparison-theorem-NCD}.
\end{proof}

\begin{thm}
\label{thm:semistable-comparison-affine}
Let $X_m=\Spec (A_m=W_m(R)[T_1,\ldots,T_n]/(T_1\cdots T_d), P=\bN^{e}\oplus \bN^f)$ be the canonical lift of $X$ over $S_m=\Spec (W_m(R), \bN)$.
Then the comparison morphism
\[
	\Lambda_{X_m/S_m}^\bullet
	\to
	W_m\Lambda_{X/S}^\bullet
\]
is a quasi-isomorphism.
\end{thm}

\begin{proof}
The comparison morphisms are compatible with the Mayer-Vietoris sequence, i.e., the following diagram commutes:
\[
\xymatrix{
	0 \ar[r]
	&
	\Lambda_{X_m/S_m}^\bullet  \ar[r] \ar[d]
	&
	\Lambda_{(Z_1)_m/S_m}^\bullet \oplus \Lambda_{(Z_2)_m/S_m}^\bullet \ar[r] \ar[d]
	&
	\Lambda_{Z_m/S_m}^\bullet \ar[r] \ar[d]
	&
	0
	\\
	0 \ar[r]
	&
	W_m\Lambda_{X/S}^\bullet \ar[r]
	&
	W_m\Lambda_{Z_1/S}^\bullet \oplus W_m\Lambda_{Z_2/S}^\bullet \ar[r]
	&
	W_m\Lambda_{Z/S}^\bullet \ar[r]
	&
	0.
}
\]
Consider the long exact sequences of hypercohomology and using descending induction,
it suffices to show when $d=1$ and it follows from Lemma \ref{lem:semistabledequalto1}.
\end{proof}

\begin{thm}
\label{thm:semistable}
Let $(Y,\mathcal{M})$ be a semistable log scheme over the pre-log ring $(R,\bN)$.
Then the canonical homomorphism
\[
	\mathbb{R}u_{(Y,\mathcal{M})/W_m(R,\bN)*}\mathcal{O}_{(Y,\mathcal{M})/W_m(R,\bN)}
	\to
	W_m\Lambda_{(Y,\mathcal{M})/(R,\bN)}^\bullet
\]
is an isomorphism in $D^+(Y, W_m(R))$.
Moreover, if $R$ is Noetherian and $Y$ is proper over $R$, we have a canonical isomorphism
\[
	H_{\logcrys}^* ((Y,\mathcal{M})/W(R,\mathbb{N}))
	\to
	\mathbb{H}_{\et}^*(Y,W\Lambda_{(Y,\mathcal{M})/(R,\mathbb{N})}^\bullet).
\]
\end{thm}

\begin{proof}
This follows from Theorem \ref{thm:semistable-comparison-affine}
by the similar proof to that of Theorem \ref{thm:comparison-theorem-NCD}.
\end{proof}

\section{Weight spectral sequence and its degeneration for semistable schemes}
\label{sect:ssq-semistable}
In this section, we define the $p$-adic Steenbrink complex for proper strictly semistable log schemes.
First we recall some facts about topology of log structures.

\subsection{Topology of log structure}
\label{subsect:topology-of-log-structures}

We recall some facts about the topology of log structures
(\cite{Sh} \S 1.1, \cite{OlLogStack} Appendix).

\begin{defn}
(\cite{Sh} Definition 1.1.1)
A fine log scheme $(X,\mathcal{M})$ is said to be of Zariski type
if there exists an open covering $X=\bigcup_i X_i$ with respect to the Zariski topology
such that each $(X_i,\mathcal{M}|_{X_i})$ admits a chart.
\end{defn}

\begin{rem}
If $X$ is a smooth scheme with simple normal crossing divisor $D$,
the log scheme $(X, D)$ is a fine log scheme of Zariski type.
A strictly semistable log scheme is also a fine log scheme of Zariski type.
\end{rem}

Let $X$ be a scheme and $\tau : X_{\et}\to X_{\Zar}$ be the canonical morphism of topoi.
For a log structure $(\mathcal{M},\alpha)$ on $X$,
the log structure $(\tau_*\mathcal{M},\tau_*\alpha)$ with respect to the Zariski topology on $X$ is defined by
\[
	\tau_*\alpha:
	\tau_*\mathcal{M}
	\xrightarrow{\tau_* \alpha}
	\tau_*\mathcal{O}_{X_{\et}}
	=\mathcal{O}_{X_{\Zar}}.
\]
Conversely, for a log structure $(\mathcal{M}',\alpha')$ with respect to the Zariski topology, we define the log structure $(\tau^*\mathcal{M}',\tau^*\alpha')$ on $X$ as the associated log structure to the pre-log structure
\[
	\tau^{-1}\mathcal{M}'
	\xrightarrow{\tau^{-1} \alpha'}
	\tau^{-1}\mathcal{O}_{X_{\Zar}}
	\to
	\mathcal{O}_{X_{\et}}.
\]

The pair $(\tau_*,\tau^*)$ induces an equivalence of categories
\[
	\left(
	\begin{array}{c}
	\text{Fine log schemes} \\
	\text{of Zariski type}
	\end{array}
	\right)
	\simeq
	\left(
	\begin{array}{c}
	\text{Fine log schemes with} \\
	\text{respect to the Zariski topology}
	\end{array}
	\right)
\]
(\cite{Sh} Corollary 1.1.11, \cite{OlLogStack} Theorem A.1).

\begin{rem}
Let $f:(X,\mathcal{M})\to (Y,\mathcal{N})$ be a morphism of fine log schemes with respect to the Zariski topology.
Then for $m\ge 1$, there is a quasi-coherent sheaf $W_m \underline{\Lambda}_{(X,\mathcal{M})/(Y,\mathcal{N})}^\bullet$ on $X_{\Zar}$ that satisfies the following condition:
If there is a commutative diagram
\[
\xymatrix{
	U=\Spec\ S' \ar[r] \ar@{^{(}->}[d]
	&
	V=\Spec\ R' \ar@{^{(}->}[d]
	\\
	X \ar[r]^{f}
	& Y,
}
\]
where vertical arrows are open immersions and there is a chart $(Q \to \mathcal{M}|_{U}, P \to \mathcal{N}|_{V}, P \to Q)$ of the morphism $(U,\mathcal{M}|_U)\to (V,\mathcal{N}|_V)$.
Then we have a canonical isomorphism
 \[
	\Gamma (U,W_m \underline{\Lambda}_{(X,\mathcal{M})/(Y,\mathcal{N})}^\bullet)
	\simeq
	W_m\Lambda_{(S',Q)/(R',P)}^\bullet.
\]
This follows from the similar argument to that of Proposition-Definition \ref{prop-def:chart-welldefined}.

We see that $W_m \Lambda_{(X,\mathcal{M})/(Y,\mathcal{N})}^\bullet$ is equal to the complex of \'etale sheaves defined by
\[
	U
	\mapsto
	\Gamma(U,(W_m(\pi))^*(W_m \underline{\Lambda}_{(X,\mathcal{M})/(Y,\mathcal{N})}^\bullet))
\]
for any object $\pi:U\to X$ in $X_{\et}$ by Proposition \ref{prop:etale-base-change-property}.
Hence
\[
	\mathbb{H}_{\Zar}^q (X,W_m \underline{\Lambda}_{(X,\mathcal{M})/(Y,\mathcal{N})}^\bullet)
	\simeq
	\mathbb{H}_{\et}^q (X,W_m \Lambda_{(X,\mathcal{M})/(Y,\mathcal{N})}^\bullet)
\]
by \cite{Fu} Proposition 5.7.5.
\end{rem}

In \S \ref{sect:ssq-semistable}, we consider strictly semistable log schemes
as fine log schemes with respect to the Zariski topology.
By abuse of notation, we write $W_m \Lambda_{(X,\mathcal{M})/(Y,\mathcal{N})}^\bullet$ for the log de Rham-Witt complex with respect to the Zariski topology.
We use the setting of \S \ref{subsect:comparison-semistable}.

\subsection{Poincar\'e residue map}
\label{section:Poincare}
Let $Y$ be a proper strictly semistable log scheme over $S=\Spec (R,\bN)$
such that $R$ is Noetherian
and $Y_1, \ldots ,Y_d$ its irreducible components.
For a subset $J = \{\alpha_1,\ldots, \alpha_j\}$ of $[1,d]$,
we set $Y_J:=Y_{\alpha_1}\cap \cdots \cap Y_{\alpha_j}$.
We give a filtration $P_j$ of $W_m\widetilde{\Lambda}^\bullet := W_m\widetilde{\Lambda}^\bullet_{Y/S}$ by
\[
	P_j W_m\widetilde{\Lambda}^i:=
	\text{image}(W_m\widetilde{\Lambda}^j \otimes_{W_m(\mathcal{O}_{\mathring{Y}})} W_m\Omega_{\mathring{Y}/R}^{i-j}
	\to
	W_m\widetilde{\Lambda}^i),
\]
where $W_m\Omega_{\mathring{Y}/R}^\bullet$ denotes the (classical) de Rham-Witt complex defined in \cite{LZDRW}.
We first define a map
\[
	\widetilde{\rho}_J :
	W_m\Omega_{\mathring{Y}/R}^{\bullet-j}
	\to
	Gr_jW_m\widetilde{\Lambda}^\bullet
	:=P_jW_m\widetilde{\Lambda}^\bullet/P_{j-1}W_m\widetilde{\Lambda}^\bullet
\]
by $\omega
\mapsto
\omega \wedge d\log[T_{\alpha_1}]\wedge \cdots \wedge d\log[T_{\alpha_j}] $,
where $T_1,\ldots, T_n$ are local coordinates of $Y$ such that each $Y_i$ corresponds to $T_i=0$.
One can show this map is independent of the choice of local coordinates by a similar proof to \cite{De} (3.5) .
Let $I_J$ be the ideal of $\mathcal{O}_{\mathring{Y}}$ corresponding to the closed immersion $i_J : Y_J \hookrightarrow Y$.
We would like to show that $\widetilde{\rho}_J$ factors through $i_{J*}W_m\Omega_{\mathring{Y_J}/R}^{\bullet-j}$.
For this, it suffices to show that
$\widetilde{\rho}_J(a\omega) = \widetilde{\rho}_J(da\wedge\omega) = 0$ for any $a \in W_m(I_J)$
and $\omega \in W_m\Omega_{\mathring{Y}/R}^{\bullet-j}$.
Any $a \in W_m(I_J)$ can be written as a finite sum:
\[
	a=\sum_{i=0}^{m-1} {}^{V^i}[c_1^{(i)}T_{\alpha_1}+\cdots +c_j^{(i)}T_{\alpha_j}],
\]
where $c_l^{(i)}\in \mathcal{O}_{\mathring{Y}}$.
Hence we can assume $a={}^{V^i}[c_1T_{\alpha_1}+\cdots +c_jT_{\alpha_j}]$. By \cite{DLZOWV} Proposition 2.23, we have an expression
\[
	[c_1T_{\alpha_1}+\cdots +c_jT_{\alpha_j}]
	=\sum_{k:\text{weight}, |k|=1}
	\beta_k [T_{\alpha_1}]^{k_1}\cdots [T_{\alpha_j}]^{k_j}, \beta_k \in {}^{V^{u(k)}} W_m(R),
\]
where $k:[1,j]\to \mathbb{Z}_{\ge 0}[1/p]$ runs through weights such that $|k|=k_1+\cdots+k_j=1$.
For a weight $k$, let $u(k)$ denote the least nonnegative integer such that $p^{u(k)}k$ is integral.
If $\beta = {}^{V^{u(k)}}\eta$,
the expression $\beta [T_{\alpha_1}]^{k_1}\cdots [T_{\alpha_j}]^{k_j}$ means
\[
{}^{V^{u(k)}}(\eta [T_{\alpha_1}]^{p^{u(k)}k_1} \cdots [T_{\alpha_j}]^{p^{u(k)}k_j}).
\]
Note that $p^{u(k)}k_1,\ldots,p^{u(k)}k_j \in \mathbb{Z}_{\ge 0}$.

Without loss of generality, we may assume that $a={}^{V^t}(\eta [T_{\alpha_1}]^{l_1}\cdots [T_{\alpha_j}]^{l_j})$ with $\eta \in W_m(R) ,t, l_1,\ldots,l_j$ nonnegative integers such that at least one of $l_1,\ldots,l_j$ is positive.
We have
\begin{align*}
	&({}^{V^t}(\eta [T_{\alpha_1}]^{l_1}\cdots [T_{\alpha_j}]^{l_j})\cdot \omega)\wedge d\log [T_{\alpha_1}]\wedge \cdots \wedge d\log [T_{\alpha_j}]\\
	&={}^{V^t}(\eta [T_{\alpha_1}]^{l_1}\cdots [T_{\alpha_j}]^{l_j}\cdot {}^{F^t}\omega \wedge d\log [T_{\alpha_1}]\wedge \cdots \wedge d\log [T_{\alpha_j}])\\
	&\equiv 0 \ \text{modulo}\ P_{j-1}W_m\widetilde{\Lambda}^\bullet
\end{align*}
since at least one of $l_1,\ldots,l_j$ is positive. Hence we see $\widetilde{\rho}_J(a \omega)=0$. Similarly we see $\widetilde{\rho}_J(da\wedge \omega)=0$ because
\begin{align*}
	&(d{}^{V^t}(\eta [T_{\alpha_1}]^{l_1}\cdots [T_{\alpha_j}]^{l_j})\wedge \omega)\wedge d\log [T_{\alpha_1}]\wedge \cdots \wedge d\log [T_{\alpha_j}]\\
	&=d{}^{V^t}(\eta [T_{\alpha_1}]^{l_1}\cdots [T_{\alpha_j}]^{l_j}\wedge {}^{F^t}\omega \wedge d\log [T_{\alpha_1}]\wedge \cdots \wedge d\log [T_{\alpha_j}])\\
	&-^{V^t}(\eta [T_{\alpha_1}]^{l_1}\cdots [T_{\alpha_j}]^{l_j}\wedge {}^{F^t}d\omega \wedge d\log
	[T_{\alpha_1}] \cdots \wedge d\log [T_{\alpha_j}])\\
	&\equiv 0 \ \text{modulo}\ P_{j-1}W_m\widetilde{\Lambda}^\bullet.
\end{align*}
Hence $\widetilde{\rho}_J$ induces the map
$
	\rho_J:i_{J*}W_m\Omega_{\mathring{Y_J}/R}^{\bullet-j}
	\to
	Gr_j W_m\widetilde{\Lambda}^\bullet.
$

Let $Y^{(j)}$ be $\coprod_{|J|=j} Y_J$ and $i^{(j)}:Y^{(j)} \to Y$ the canonical map.
From the collection of maps $\{\rho_J\}_{|J|=j}$ we obtain a map
$i_*^{(j)}W_m\Omega_{\mathring{Y}^{(j)}/R}^{\bullet-j}\to Gr_j W_m\widetilde{\Lambda}^\bullet$.
We sometimes drop $i_*^{(j)}$ when there is no risk of confusion.
\begin{lem}
\label{lem:resisom}
$i_*^{(j)}W_m\Omega_{\mathring{Y}^{(j)}/R}^{\bullet-j}\to Gr_j W_m\widetilde{\Lambda}^\bullet$ is an isomorphism.
We call the inverse isomorphism of this map Poincar\'e residue isomorphism $\Res:Gr_j W_m\widetilde{\Lambda}^\bullet \simeq i_*^{(j)}W_m\Omega_{\mathring{Y}^{(j)}/R}^{\bullet-j}$.
\end{lem}

\begin{proof}
Without loss of generality, we can assume $S=\Spec R$ and
\[
	Y=\Spec(R[T_1,\ldots, T_n]/(T_1\cdots T_d),\bN^d).
\]
 In this case, we find
\[
	\mathring{Y}_J=\Spec (R[T_1,\ldots,\widehat{T_{\alpha_1}},\ldots,\widehat{T_{\alpha_j}},\ldots,T_n])
\]
is the spectrum of a polynomial ring.
On the other hand, an element of $Gr_jW_m\widetilde{\Lambda}^\bullet$ has a unique expression as a sum of basic Witt differentials with $|I_{-\infty}|=j$.
We already know the basic Witt differentials on the left hand side (\cite{LZDRW} \S 2.2).
Comparing the basic Witt differential on both sides, the claim follows.
\end{proof}

\subsection{Weight spectral sequence}
We are ready to construct the weight filtration of a strictly semistable log scheme.
Put $W_mA^{ij}:=W_m\widetilde{\Lambda}^{i+j+1}/P_jW_m\widetilde{\Lambda}^{i+j+1}$.

\begin{lem}
There exists the following exact sequences:
\[
	0
	\to
	W_m\Lambda^i
	\xrightarrow{\theta_m\wedge}
	W_m A^{i0}
	\xrightarrow{(-1)^i\theta_m\wedge}
	W_m A^{i1}
	\xrightarrow{(-1)^i\theta_m\wedge}
	\cdots.
\]
\end{lem}

\begin{proof}
It suffice to show the exactness of the following sequence
(cf. \cite{Mo} Proposition 3.15):
\begin{align*}
	W_m\widetilde{\Lambda}^{i-1}
	\xrightarrow{\theta_m\wedge}
	W_m\widetilde{\Lambda}^i
	\xrightarrow{(-1)^i\theta_m\wedge}
	& W_m\widetilde{\Lambda}^{i+1}/P_0 W_m\widetilde{\Lambda}^{i+1}\\
	& \xrightarrow{(-1)^i\theta_m\wedge}
	W_m\widetilde{\Lambda}^{i+2}/P_1 W_m\widetilde{\Lambda}^{i+2}
	\xrightarrow{(-1)^i\theta_m\wedge} \cdots.
\end{align*}
We deduce the exactness of this complex by a similar argument to that in Lemma \ref{lem:exactness}.
\end{proof}

We consider $W_mA^{\bullet \bullet}$ as a double complex by
\[
\xymatrixcolsep{4pc}
\xymatrix{
	W_mA^{i,j+1}
	&
	\\
	W_mA^{i,j} \ar[u]^{(-1)^i\theta_m\wedge} \ar[r]^-{(-1)^{j+1}d}
	&
	W_mA^{i+1,j}
}
\]
(cf. \cite{Nak} (2.2.1;$\star$)).
Consider a simple complex
\[
	W_mA^{i\bullet}:=
	(\cdots
	\xrightarrow{(-1)^i\theta_m\wedge}
	W_m \widetilde{\Lambda}^{i+j+1}/P_j W_m\widetilde{\Lambda}^{i+j+1}
	\xrightarrow{(-1)^i\theta_m\wedge}
	\cdots)_{j\ge 0},
\]
and define a weight filtration on this complex by
\begin{align*}
	&P_kW_mA^{i\bullet}:=\\
	&(\cdots
	\xrightarrow{(-1)^i\theta_m\wedge}
	(P_{2j+k+1}+P_j)(W_m\widetilde{\Lambda}^{i+j+1})
	/P_j W_m\widetilde{\Lambda}^{i+j+1}
	\xrightarrow{(-1)^i\theta_m\wedge}
	\cdots)_{j\ge 0}.
\end{align*}
If we ignore the compatibility with the Frobenius, we obtain an isomorphism
\begin{align*}
	Gr_k W_mA^{*\bullet}
	&= \bigoplus_{j\ge \text{max}\{-k,0\}} Gr_{2j+k+1} W_m\widetilde{\Lambda}^{*+j+1} \{-j\}_\bullet \\
	&\xrightarrow[\Res]{\sim}
	\bigoplus_{j\ge \text{max}\{-k,0\}} W_m\Omega_{\mathring{Y}^{(2j+k+1)}/R}^* \{-j-k\}_* \{ -j \}_\bullet,
\end{align*}
where, for $n \in \mathbb{Z}, \{n\}_\bullet$ (resp. $\{n\}_*$) denotes the shift
of the complex with respect to $\bullet$ (resp. $*$) by $n$ with the signature of
differentials unchanged.
Hence we get a spectral sequence:
\begin{align*}
	E_1^{-k,h+k}=
	\bigoplus_{j\ge \text{max}\{-k,0\}} \mathbb{H}_{\Zar}^{h-2j-k} (\mathring{Y}^{(2j+k+1)},&W_m\Omega_{\mathring{Y}^{(2j+k+1)}/R}^*)\\
	&\Rightarrow \mathbb{H}_{\Zar}^h (\mathring{Y}, W_m\Lambda_{Y/(R,\bN)}^*).
\end{align*}

We would like to construct a spectral sequence also for the non-truncated de Rham-Witt cohomology.
The canonical projection map
$\pi:W_{m+1}\widetilde{\Lambda}^\bullet \to W_m\widetilde{\Lambda}^\bullet$ satisfies
$\pi(P_jW_{m+1}\widetilde{\Lambda}^\bullet)\subset P_jW_m\widetilde{\Lambda}^\bullet$.
Then $\pi$ induces the map
\[
	\pi :W_{m+1}A^{ij} \to W_mA^{ij}
\]
and there exist two commutative diagrams
\[
\xymatrix{
	W_{m+1}A^{i,j+1} \ar[r]^{\pi}
	&
	W_mA^{i,j+1}
	&
	W_{m+1}A^{ij} \ar[r]^{\pi}\ar[d]_{(-1)^{j+1}d}
	&
	W_mA^{ij} \ar[d]^{(-1)^{j+1}d}
	\\
	W_{m+1}A^{ij} \ar[u]^{(-1)^i\theta_{m+1}\wedge} \ar[r]^{\pi}
	&
	W_mA^{ij}, \ar[u]_{(-1)^i\theta_m\wedge}
	&
	W_{m+1}A^{i+1,j} \ar[r]^{\pi}
	&
	W_mA^{i+1,j}.
}
\]
Therefore we get a morphism of double complexes
\[
	\pi : W_{m+1}A^{\bullet \bullet}
	\to
	W_m A^{\bullet \bullet}.
\]
For any nonnegative integer $k$, the projection morphism
$\pi : P_k W_{m+1}\widetilde{\Lambda}^\bullet \to P_k W_m\widetilde{\Lambda}^\bullet$
is surjective by definition.
We know $W_m \widetilde{\Lambda}^i$ is a coherent sheaf of $W_m(\mathcal{O}_{\mathring{Y}})$-module
and there is an exact sequence
\[
	0
	\to
	P_{k-1} W_m\widetilde{\Lambda}^i
	\to
	P_k W_m\widetilde{\Lambda}^i
	\xrightarrow{\Res}
	W_m\Omega_{\mathring{Y}^{(k)}/R}^{i-k}
	\to
	0.
\]
From this one sees that $P_k W_m\widetilde{\Lambda}^i$ is a quasi-coherent sheaf of $W_m(\mathcal{O}_{\mathring{Y}})$-modules for each $k$ by induction.
Moreover, there exists the following commutative diagram with exact rows:
\[
\xymatrix{
	0 \ar[r]
	&
	P_{k-1} W_{m+1}\widetilde{\Lambda}^i \ar[d]^{\pi} \ar[r]
	&
	P_k W_{m+1}\widetilde{\Lambda}^i \ar[d]^{\pi} \ar[r]^-{\Res}
	&
	W_{m+1}\Omega_{\mathring{Y}^{(k)}/R}^{i-k} \ar[d]^{\pi} \ar[r]
	&
	0
	\\
	0 \ar[r]
	&
	P_{k-1} W_m\widetilde{\Lambda}^i \ar[r]
	&
	P_k W_m\widetilde{\Lambda}^i \ar[r]^-{\Res}
	&
	W_m\Omega_{\mathring{Y}^{(k)}/R}^{i-k} \ar[r]
	&
	0.
}
\]
This exact sequence and the fact that $\{P_kW_m\widetilde{\Lambda}^\bullet\}_m$ satisfy the Mittag-Leffler condition show that the sequence
\[
	0
	\to
	P_k W\widetilde{\Lambda}^\bullet
	\to
	P_{k+1} W\widetilde{\Lambda}^\bullet
	\xrightarrow{\Res}
	W\Omega_{\mathring{Y}^{(k+1)}/R}^\bullet\{-k-1\}
	\to
	0
\]
is exact. The weight spectral sequence (we ignore Frobenius action)
\[
	E_1^{-k,h+k}=
	\bigoplus_{j\ge \text{max}\{-k,0\}}
	\mathbb{H}_{\Zar}^{h-2j-k} (\mathring{Y}^{(2j+k+1)},W\Omega_{\mathring{Y}^{(2j+k+1)}/R}^*)
	\Rightarrow
	\mathbb{H}_{\Zar}^h (\mathring{Y},W\Lambda_{Y/(R,\bN)}^*)
\]
is deduced from this exact sequence.

\subsection{Frobenius compatibility}
In this subsection we discuss the Frobenius compatibility of the spectral sequence we constructed in the last subsection.
We assume that $p$ is nilpotent in $R$.

\begin{lem}
\label{lem:Gys1}
(cf. \cite{Mo} Proposition 4.12, \cite{Nak} (10.1.16))

Let $j$ be a nonnegative integer.
For $1\le q \le j+1$, $\iota^{(q)}:Y^{(j+1)}\hookrightarrow Y^{(j)}$ denotes different closed immersions,
and $\rho_m^{(q)} : i_*^{(j)} W_m\Omega_{\mathring{Y}^{(j)}/R}^\bullet \to i_*^{(j+1)} W_m\Omega_{\mathring{Y}^{(j+1)}/R}^\bullet$ be a morphism induced by $\iota^{(q)}$.
We set $\rho_m := \sum_{q=1}^{j+1}(-1)^{q+1} \rho_m^{(q)}$.
Then there is the following commutative diagram:
\[
\xymatrixcolsep{4pc}
\xymatrix{
	Gr_j W_m \widetilde{\Lambda}^\bullet \ar[r]^-{\theta_m \wedge} \ar[d]^{\Res}_\simeq
	&
	Gr_{j+1} W_m \widetilde{\Lambda}^{\bullet+1} \ar[d]^{\Res}_\simeq
	\\
	W_m\Omega_{\mathring{Y}^{(j)}/R}^{\bullet-j} \ar[r]^-{(-1)^{\bullet-j} \rho_m}
	&
	W_m\Omega_{\mathring{Y}^{(j+1)}/R}^{\bullet-j}.
}
\]
\end{lem}

\begin{proof}
Since we can check the commutativity locally,
we may work on $Y_J$ for some $J =\{\alpha_1,\ldots,\alpha_{j+1} \}$.
For $1\le q\le j+1$, let $J_q=\{\alpha_1,\ldots,\widehat{\alpha_q},\ldots,\alpha_{j+1} \}$.
The claim follows from the commutativity of the following diagram:
\[
\xymatrixcolsep{4pc}
\xymatrix{
	Gr_j W_m \widetilde{\Lambda}^\bullet \ar[r]^{\theta_m \wedge}
	&
	Gr_{j+1} W_m \widetilde{\Lambda}^{\bullet+1}
	\\
	W_m\Omega_{\mathring{Y}_{J_q}/R}^{\bullet-j} \ar[r]^{(-1)^{\bullet-j+q-1}\rho_m^{(q)}}\ar[u]
	&
	W_m\Omega_{\mathring{Y}_J/R}^{\bullet-j}, \ar[u]
}
\]
which we can check directly from definitions.
\end{proof}

We describe the Frobenius on torsion $p$-adic Steenbrink complexes.
We assume that $p^nR=0$.
We mention that for an integer $k \ge n$ and nonnegative integer $j$,
the multiplication $p^k:W_{m+1}\Lambda^j\to W_{m+1}\Lambda^j$ factors $\underline{p}^k:W_m\Lambda^j\to W_{m+1}\Lambda^j$
since $p^n$ annihilates
\begin{align*}
	\ker (W_{m+1}\Lambda^j\to W_m\Lambda^j) &= \Fil^m W_{m+1}\Lambda^j\\
	&= {}^{V^m}\Lambda_{Y/R}^j + d^{V^m}\Lambda_{Y/R}^{j-1}
\end{align*}
(See Proposition \ref{prop:fil-exact-sequence}).

\begin{thm}
(cf. \cite{Nak} Proposition 9.8)

Let $m$, $k$ be two positive integers and $j$ a nonnegative integer.

(1) $\underline{p}^{n-1+k}F:W_m\Lambda^j\to W_m\Lambda^j$ is a unique morphism
which makes the following diagram commutative:
\[
\xymatrix{
	W_{m+1}\Lambda^j \ar[d]_{p^{n-1+k}F} \ar[r]^{\pi}
	&
	W_m\Lambda^j \ar@{-->}[ld]^{\underline{p}^{n-1+k}F}
	\\
	W_m\Lambda^j
	&
}
\]
Furthermore, $\underline{p}^{n-1+k}F$ is compatible with $d$ and $\pi$,
i.e., The following two diagrams commute:
\[
\xymatrix{
	W_m\Lambda^j \ar[r]^d \ar[d]_{\underline{p}^{n-1+k} F}
	&
	W_m\Lambda^{j+1} \ar[d]^{\underline{p}^{n+k}F}
	&
	W_{m+1}\Lambda^j \ar[r]^\pi \ar[d]_{\underline{p}^{n-1+k}F}
	&
	W_m\Lambda^j \ar[d]^{\underline{p}^{n-1+k} F
	}\\
	W_m\Lambda^j \ar[r]^d
	&
	W_m\Lambda^{j+1},
	&
	W_{m+1}\Lambda^j \ar[r]^\pi
	&
	W_m\Lambda^j.
}
\]
If $\tilde{F}: W_m\Lambda^k\to W_m\Lambda^k$ is the morphism induced by
the absolute Frobenius morphism on $\mathring{Y}$,
the morphism $\underline{p}^{n-1+k}F: W_m\Lambda^k\to W_m\Lambda^k$ is equal to $p^{n-1}\tilde{F}$.

(2) There is a unique morphism $\widetilde{\Phi}^{(j)}_m:W_mA^{0j}\to W_mA^{0j}$
which makes the following diagram commutative:
\[
\xymatrix{
	W_{m+1}A^{0j} \ar[d]_F \ar[r]^{\pi}
	&
	W_mA^{0j}. \ar@{-->}[ld]^{\widetilde{\Phi}^{(j)}_m}
	\\
	W_mA^{0j}
	&
}
\]
Furthermore, $\widetilde{\Phi}^{(j)}_m$ is compatible with $\theta_m \wedge$ and $\pi$:
\[
\xymatrix{
W_mA^{0,j+1} \ar[r]^{\widetilde{\Phi}^{(j+1)}_m} & W_mA^{0,j+1} & W_{m+1}A^{0j} \ar[r]^{\widetilde{\Phi}^{(j)}_{m+1}}\ar[d]_{\pi} & W_{m+1}A^{0j} \ar[d]^{\pi}\\
W_mA^{0j} \ar[u]^{\theta_m\wedge} \ar[r]^{\widetilde{\Phi}^{(j)}_m} & W_mA^{0j}, \ar[u]_{\theta_m\wedge} & W_mA^{0j} \ar[r]^{\widetilde{\Phi}^{(j)}_m} & W_mA^{0j}.
}
\]
(3) $\underline{p}^{n-1+k}F:W_m\widetilde{\Lambda}^j\to W_m\widetilde{\Lambda}^j$ is the unique morphism which makes the following diagram commutative:
\[
\xymatrix{
	W_{m+1}\widetilde{\Lambda}^j \ar[d]_{p^{n-1+k}F} \ar[r]^{\pi}
	&
	W_m\widetilde{\Lambda}^j. \ar@{-->}[ld]^{\underline{p}^{n-1+k}F}
	\\
	W_m\widetilde{\Lambda}^j
	&
}
\]
Furthermore, $\underline{p}^{n-1+k}F$ is compatible with $d, (\theta_m \wedge), \pi$ and $\underline{p}^{n-1+k}F$ on $W_m\Lambda^j$.
In other words, the following diagrams commute:
\[
\xymatrix{
	W_m\widetilde{\Lambda}^j \ar[r]^d \ar[d]_{\underline{p}^{n-1+k} F}
	&
	W_m\widetilde{\Lambda}^{j+1} \ar[d]^{\underline{p}^{n+k}F}
	&
	W_m\widetilde{\Lambda}^{j+1} \ar[r]^{\underline{p}^{n-1+k} F}
	&
	W_m\widetilde{\Lambda}^{j+1}
	\\
	W_m\widetilde{\Lambda}^j \ar[r]^d
	&
	W_m\widetilde{\Lambda}^{j+1},
	&
	W_m\widetilde{\Lambda}^j  \ar[r]^{\underline{p}^{n-1+k} F} \ar[u]^{\theta_m\wedge}
	&
	W_m\widetilde{\Lambda}^j, \ar[u]_{\theta_m\wedge}
	\\%
	W_m\widetilde{\Lambda}^{j} \ar[r]^{\underline{p}^{n-1+k} F} \ar[d]_\pi
	&
	W_m\widetilde{\Lambda}^{j+1} \ar[d]^\pi
	&
	W_m\widetilde{\Lambda}^{j+1} \ar[r]^{\underline{p}^{n-1+k} F}
	&
	W_m\widetilde{\Lambda}^{j+1}
	\\
	W_m\widetilde{\Lambda}^j  \ar[r]^{\underline{p}^{n-1+k} F}
	&
	W_m\widetilde{\Lambda}^j,
	&
	W_m\Lambda^j \ar[r]^{\underline{p}^{n-1+k} F} \ar[u]^{\theta_m\wedge}
	&
	W_m\Lambda^{j}, \ar[u]_{\theta_m\wedge}.
}
\]
(4) The morphism $\underline{p}^{n-1+k} F$ on $W_m\widetilde{\Lambda}^j$ preserves the weight filtration $P$ on $W_m\widetilde{\Lambda}^j$.
For an integer $i\ge 1$, $\underline{p}^{n-1+i}F: W_m\widetilde{\Lambda}^{i+\bullet+1}\to W_m\widetilde{\Lambda}^{i+\bullet+1}$ induces an endomorphism
\[
\underline{p}^{n-1+i}F: W_mA^{i\bullet}\to W_mA^{i\bullet}
\]
of complexes.

(5) Let $i$ be a positive integer. The following diagrams commute:
\[
\xymatrix{
	W_{m+1}A^{ij} \ar[d]_{p^{n-1+i}F} \ar[r]^{\pi}
	&
	W_mA^{ij} \ar[ld]^{\underline{p}^{n-1+i}F}
	&
	W_mA^{ij} \ar[r]^{(-1)^{j+1}d} \ar[d]_{\underline{p}^{n-1+i}F}
	&
	W_mA^{i+1,j} \ar[d]^{\underline{p}^{n+i}F}
	\\
	W_mA^{ij}
	&
	&
	W_mA^{ij} \ar[r]^{(-1)^{j+1}d}
	& W_mA^{i+1,j},
	\\
	W_mA^{i,j+1} \ar[r]^{\underline{p}^{n-1+i}F}
	&
	W_mA^{i,j+1}
	&
	W_{m+1}A^{ij} \ar[r]^{\underline{p}^{n-1+i}F}\ar[d]_{\pi}
	&
	W_{m+1}A^{ij} \ar[d]^{\pi}
	\\
	W_mA^{ij} \ar[u]^{(-1)^i\theta_m\wedge} \ar[r]^{\underline{p}^{n-1+i}F}
	&
	W_mA^{ij}, \ar[u]_{(-1)^i\theta_m\wedge}
	&
	W_mA^{ij} \ar[r]^{\underline{p}^{n-1+i}F}
	&
	W_mA^{ij}.
}
\]
(6) The following diagram is commutative:
\[
\xymatrixcolsep{4pc}
\xymatrix{
	 W_mA^{0j} \ar[r]^{(-1)^{j+1}d} \ar[d]_{p^{n-1}\widetilde{\Phi}_m^{(j)}}
	 &
	 W_mA^{1j} \ar[d]^{\underline{p}^nF}
	 \\
	 W_mA^{0j} \ar[r]^{(-1)^{j+1}d}
	 &
	 W_mA^{1j}.
}
\]
(7) The following diagram is commutative:
\[
\xymatrix{
	W_m\widetilde{\Lambda}^j \ar[r]^{\underline{p}^{n-1+k}F} \ar[d]
	&
	W_m\widetilde{\Lambda}^j \ar[d]
	\\
	W_m\Lambda^j = W_m\widetilde{\Lambda}^j/(\theta_m\wedge W_m\widetilde{\Lambda}^{j-1}) \ar[r]^{\underline{p}^{n-1+k}F}
	&
	W_m\Lambda^j = W_m\widetilde{\Lambda}^j/(\theta_m\wedge W_m\widetilde{\Lambda}^{j-1}).
}
\]
\end{thm}

\begin{proof}
(1) Uniqueness follows from the surjectivity of $\pi$.
Since $p^{n-1+k}=\underline{p}^{n-1+k}\circ \pi$ and $d^F\omega = p^F d\omega$, the diagrams commute.
We obtain the compatibility with projections by the compatibility of $\pi$ and $\underline{p}^{n-1+k}$, $\pi$ and $F$.

(2) The Poincar\'e residue morphism gives an isomorphism
\[
	\Res: W_mA^{0j}
	\to
	W_m(\mathcal{O}_{\mathring{Y}^{(j+1)}})
\]
and there is the Frobenius morphism
\[
	\Phi_m^{(j)} :
	W_m(\mathcal{O}_{\mathring{Y}^{(j+1)}})
	\to
	W_m(\mathcal{O}_{\mathring{Y}^{(j+1)}}).
\]
Define
\[
\widetilde{\Phi}_m^{(j)}:=
	\Res^{-1} \circ \Phi_m^{(j)} \circ \Res : W_mA^{0j}
	\to
	W_mA^{0j}.
\]
The commutativity of the first diagram is deduced from following commutative diagram:
\[
\xymatrix{
	W_{m+1}A^{0j} \ar[dd]_F \ar[r]^{\pi} \ar[rd]_{\Res}^{\sim}
	&
	W_mA^{0j} \ar[r]_{\Res}^\sim
	&
	W_m(\mathcal{O}_{\mathring{Y}^{(j+1)}}) \ar[dd]^{\Phi_m^{(j)}}
	\\
	&
	W_{m+1}(\mathcal{O}_{\mathring{Y}^{(j+1)}}) \ar[ur]_\pi \ar[rd]_F
	&
	\\
	W_mA^{0j} \ar[rr]^\sim_\Res
	&
	&
	W_m(\mathcal{O}_{\mathring{Y}^{(j+1)}}).
}
\]
The commutativity of second diagram follows from Lemma \ref{lem:Gys1} and the commutativity of $\rho_m$ and $\Phi_m^{(j)}$. Third case is trivial.

(3)  The proof of (3) is the same as that of (1).

(4)(5) Trivial.

(6) Since we know $\pi$ is surjective and $\pi$ commutes with $d$, this follows from the following commutative diagram:
\[
\xymatrix{
	W_{m+1}A^{0j} \ar[r]^{d} \ar[d]_{F}
	&
	W_{m+1}A^{1j} \ar[d]^{pF}
	\\
	W_{m}A^{0j} \ar[r]^{d}
	&
	W_{m}A^{1j}.
}
\]

(7) Trivial.
\end{proof}

\begin{thm} (cf. \cite{Nak} Theorem 9.9)
There exists a unique endomorphism
$\widetilde{\Phi}_m^{(n;\bullet *)} : W_m A^{\bullet *} \to W_m A^{\bullet *}$
of double complexes which makes the following diagram commutative:
\[
\xymatrix{
	W_{m+1} A^{\bullet *} \ar[d]_{p^{n-1+\bullet} F} \ar[r]^{\pi}
	&
	W_m A^{\bullet *}. \ar@{-->}[ld]^{\widetilde{\Phi}_m^{(n;\bullet *)}}
	\\
	W_m A^{\bullet *}
	&
}
\]
The endomorphism $\widetilde{\Phi}_m^{(n;\bullet *)}$ defines an endomorphism
$\widetilde{\Psi}_m^{(n)}:W_m A^\bullet \to W_m A^\bullet$
and there is the following commutative diagram:
\[
\xymatrix{
	W_m A^\bullet \ar[r]^{\widetilde{\Psi}_m^{(n)}}
	&
	W_m A^\bullet
	\\
	W_m \Lambda^\bullet \ar[r]^{p^{n-1}\Psi_m} \ar[u]^{\theta_m \wedge}
	& W_m \Lambda^\bullet \ar[u]_{\theta_m \wedge},
}
\]
where $\Psi_m$ is the endomorphism induced by the absolute Frobenius.

The Poincar\'e residue isomorphism $\Res$ induces an isomorphism
\[
\Res: Gr_k W_m A^\bullet \simeq \bigoplus_{j\ge\max \{-k,0\}}(W_m\Omega_{\mathring{Y}^{(2j+k+1)}/R}^\bullet,(-1)^{j+1}d)\{-2j-k\}
\]
which makes the following diagram commutative:
\[
\xymatrix{
	Gr_k W_m A^\bullet \ar[r]_-\sim^-\Res \ar[d]_{\widetilde{\Psi}_m^{(n)}}
	&
	\bigoplus_{j\ge\max \{-k,0\}}(W_m\Omega_{\mathring{Y}^{(2j+k+1)}/R}^\bullet,(-1)^{j+1}d)\{-2j-k\} \ar[d]^{p^{n-1+j+k}\Psi_m}
	\\
	Gr_k W_m A^\bullet \ar[r]_-\sim^-\Res
	&
	\bigoplus_{j\ge\max \{-k,0\}}(W_m\Omega_{\mathring{Y}^{(2j+k+1)}/R}^\bullet,(-1)^{j+1}d)\{-2j-k\}.
}
\]
\end{thm}
\begin{proof}
Define $\widetilde{\Phi}_m^{(n;\bullet *)}$ by
\[
	\widetilde{\Phi}_m^{(n;\bullet *)}
	=\left\{
	\begin{array}{ll}
		p^{n-1}\widetilde{\Phi}_m^{(*)}
		&
		(\bullet = 0),
		\\
		\underline{p}^{n-1+*}F
		&
		(\bullet \neq 0).
	\end{array}
	\right.
\]
The commutativity of the second diagram follows from the following commutative diagram:
\[
\xymatrix{
	Gr_k W_{m+1}A^{i\bullet} \ar[r]_-\sim^-\Res \ar[d]_{p^iF}
	&
	\bigoplus_{j\ge \text{max}\{-k,0\}}
	W_{m+1}\Omega_{\mathring{Y}^{(2j+k+1)}/R}^{i-j-k} \{-j\} \ar[d]^{p^{j+k}(p^{i-j-k}F)}
	\\
	Gr_k W_m A^{i\bullet} \ar[r]_-\sim^-\Res
	&
	\bigoplus_{j\ge \text{max}\{-k,0\}}
	W_m\Omega_{\mathring{Y}^{(2j+k+1)}/R}^{i-j-k} \{-j\},
}
\]
which immediately follows from the definition of $\Res$.
\end{proof}

By the comparison theorem \ref{thm:semistable}, we obtain the following theorem:

\begin{thm}
\label{thm:weight-spectral-sequence}
(1)
There exists the spectral sequences:
\begin{align*}
	E_1^{-k,h+k}=
	\bigoplus_{j\ge\max \{-k,0\}}
	H_{\crys}^{h-2j-k} (\mathring{Y}^{(2j+k+1)}&/W_m(R))(-j-k)\\
	&\Rightarrow
	H^h_{\logcrys} (Y/W_m(R,\bN)).
\end{align*}

(2)
Set
\renewcommand{\arraystretch}{1.5}
\[
\begin{array}{ll}
	W\Lambda^\bullet := \varprojlim_m W_m\Lambda^\bullet,
	&
	WA^\bullet := \varprojlim_m W_m A^\bullet,
	\\
	\Psi^{(n)} := \varprojlim_m \Psi_m^{(n)} : W\Lambda^\bullet \to W\Lambda^\bullet,
	&
	\widetilde{\Phi}^{(n;\bullet *)} := \varprojlim_m \widetilde{\Phi}_m^{(n;\bullet *)} : WA^{\bullet *}\to WA^{\bullet *}
\end{array}
\]
\renewcommand{\arraystretch}{1}
Then there exists the spectral sequence:
\begin{align*}
E_1^{-k,h+k}=
	\bigoplus_{j\ge\max \{-k,0\}}H_{\crys}^{h-2j-k}(\mathring{Y}^{(2j+k+1)}&/W(R))(-j-k)\\
	&\Rightarrow H^h_{\logcrys}(Y/W(R,\bN)).
\end{align*}
We call this spectral sequence the $p$-adic weight spectral sequence.
\end{thm}

\subsection{Gysin map}
In this subsection we describe Gysin maps defined on the de Rham complexes,
the de Rham-Witt complexes and the crystalline cohomology, and their relation.

Let $X$ be a smooth scheme over a scheme $S$ and $D$ be a smooth divisor of $X/S$.
The Gysin map of the de Rham complexes
$G_{D/X}^{\dR}:\Omega_{D/S}^\bullet \{ -1\} \to \Omega_{X/S}^\bullet [1]$
in the derived category of sheaves on $X$ is equal to the boundary morphism of the exact sequence (cf. \cite{Mo} \S 4.1)
\[
	0
	\to
	\Omega_{X/S}^\bullet
	\to
	\Lambda_{(X,D)/S}^\bullet
	\xrightarrow{\Res}
	\Omega_{D/S}^\bullet \{-1\}
	\to
	0.
\]

Next, We recall the Gysin map of the crystalline cohomology.
Note that the Gysin map in the crystalline cohomology is originally defined by
Berthelot \cite{B}, but the construction in \cite{NS} is convenient for our purpose.
Let $(S,\mathcal{I},\gamma)$ be a pd-scheme such that $p$ is nilpotent in $S$.
Set $S_0:=\Spec (\mathcal{O}_S/\mathcal{I})$.
Let $X$ be a smooth scheme over $S_0$ and $D$ a smooth divisor on $X$ over $S_0$.
We denote by $a$ the natural closed immersion $D\hookrightarrow X$ .
Let $a_{\Zar}:(D_{\Zar},\mathcal{O}_D)\to (X_{\Zar},\mathcal{O}_X)$
(resp. $a_{\crys}:((D/S)_{\crys},\mathcal{O}_{D/S})\to ((X/S)_{\crys},\mathcal{O}_{X/S})$)
be the induced morphism of Zariski ringed topoi (resp. crystalline ringed topoi).
The Gysin map of the crystalline cohomology is defined as follows.

Choose an open covering $X=\bigcup_{i\in I_0}X^{i_0}$ such that
there exist a smooth scheme $Y^{i_0}$ with smooth divisor $Z^{i_0}$ on $Y^{i_0}$ over $S$ and a cartesian diagram
\[
\xymatrix{
	X^{i_0} \ar[r]
	&
	Y^{i_0}
	\\
	D|_{X^{i_0}} \ar[r] \ar[u]
	&
	Z^{i_0}. \ar[u]
}
\]

Fix a total order on $I_0$ and let $I$ be a category
whose objects are $\underline{i} = (i_0,\ldots,i_r)'$s $(i_0 < i_1 < \ldots < i_r, r\in \mathbb{Z}_{\ge 0})$.
Set $\{\underline{i} \}:=\{i_0,\ldots,i_r \}$.
For two objects $\underline{i},\underline{i}'\in I$, a morphism from $\underline{i}'$ to $\underline{i}$ is the inclusion $\{\underline{i}' \}\hookrightarrow \{\underline{i} \}$.
By abuse of notation, we sometimes write simply $\underline{i}$ instead of $\{\underline{i} \}$.

Set $D^{i_0}:=D|_{X^{i_0}}$.
For an object $\underline{i} = (i_0,\ldots,i_r)$, we set $X^{\underline{i}}:=\bigcap_{s=1}^r X^{i_s},D^{\underline{i}}:=\bigcap_{s=1}^r D^{i_s}$.
Then $(X^\bullet, D^\bullet)$ is a diagram of log schemes, i.e., a contravariant functor
\[
	I^{\text{op}}
	\to
	\text{LogSch}
\]
over $(X,D)$.
By \cite{NS} (2.4.0.2), there exists a closed immersion $(X^\bullet,D^\bullet)\hookrightarrow (Y^\bullet,Z^\bullet)$ to a diagram of smooth schemes with smooth divisor over $S$.
Let $a^\bullet: D^\bullet \hookrightarrow X^\bullet$ and $b^\bullet: Z^\bullet \hookrightarrow Y^\bullet$ be diagrams of the natural closed immersions.
By using Poincar\'e residue isomorphism, there is the following exact sequence
\[
	0
	\to
	\Omega_{Y^\bullet/S}^\bullet
	\to
	\Lambda_{(Y^\bullet,Z^\bullet)/S}^\bullet
	\xrightarrow{\Res}
	b_{\Zar *}^\bullet \Omega_{Z^\bullet/S}^\bullet \{-1\}
	\to
	0.
\]

Let $L_{X^\bullet/S}$ (resp. $L_{D^\bullet/S}$) be the linearization functor (\cite{BO} Construction 6.9)
with respect to the diagram $X^\bullet \hookrightarrow Y^\bullet$
(resp. $D^\bullet \hookrightarrow Z^\bullet)$
of closed immersions of schemes.
Let $L_{(X^\bullet, D^\bullet)/S}$ be the log linearization functor (\cite{NS} \S 2.2)
with respect to the diagram $(X^\bullet,D^\bullet)\to (Y^\bullet, Z^\bullet)$ of closed immersions of log schemes.
Let $Q_{X/S}:(X/S)_{\text{Rcrys}} \to (X/S)_{\crys},Q_{X^\bullet/S}:(X^\bullet/S)_{\text{Rcrys}} \to (X^\bullet/S)_{\crys}$ be the natural morphisms from the restricted crystalline topos to the crystalline topos (\cite{B} IV 2.1). Then we have morphisms
\begin{align*}
	Q_{X/S}&:((X/S)_{\text{Rcrys}},Q_{X/S}^*\mathcal{O}_{X/S})
	\to
	((X/S)_{\crys},\mathcal{O}_{X/S})
	\\
	Q_{X^\bullet/S}&:((X^\bullet/S)_{\text{Rcrys}},Q_{X^\bullet/S}^*\mathcal{O}_{X^\bullet/S})
	\to
	((X^\bullet/S)_{\crys},\mathcal{O}_{X^\bullet/S})
\end{align*}
of ringed topoi (\cite{B} IV (2.1.1)).

The following diagram is commutative by \cite{NS} Corollary 2.2.12:

\[
\xymatrix{
	\left( \frac{\mathcal{O}_{\overline{Z}^\bullet}\text{-modules}}{\text{HPD differential operators}} \right) \ar[d]_{L_{D^\bullet/S}} \ar[r]^{b_{\Zar *}^\bullet}
	&
	\left(\frac{\mathcal{O}_{\overline{Y}^\bullet}\text{-modules}}{\text{HPD differential operators}} \right) \ar[d]_{L_{X^\bullet/S}}
	\\
	(\text{Crystals of } \mathcal{O}_{D^\bullet/S}\text{-modules}) \ar[r]^{a_{\crys *}^\bullet}
	&
	(\text{Crystals of } \mathcal{O}_{X^\bullet/S}\text{-modules}),
}
\]
where $\overline{Z}^\bullet$ (resp. $\overline{Y}^\bullet$) is the pd-envelope of the closed immersion $D^\bullet\hookrightarrow Z^\bullet$
(resp. $X^\bullet\hookrightarrow Y^\bullet$) over $(S,I,\gamma)$.
Hence we have the following exact sequence:
\begin{align*}
	0
	\to
	Q_{X^\bullet/S}^*L_{X^\bullet/S} (\Omega_{Y^\bullet/S}^\bullet)
	\to
	Q_{X^\bullet/S}^*L_{X^\bullet/S} & (\Lambda_{(Y^\bullet,Z^\bullet)/S}^\bullet)
	\to
	\\
	& Q_{X^\bullet/S}^* a_{\crys *}^\bullet L_{D^\bullet/S} (\Omega_{Z^\bullet/S}^\bullet)\{-1\} \to 0
\end{align*}

Let $\theta_{X/S,\crys}:(X^\bullet/S)_{\crys}\to (X/S)_{\crys}$ and $\theta_{X/S,\Rcrys}:(X^\bullet/S)_{\Rcrys}\to (X/S)_{\Rcrys}$ be natural augmentation morphisms of topoi. Similarly, we have augmentation morphisms
\[
	\theta_{D/S,\crys},
	\theta_{D/S,\Rcrys},
	\theta_{(X,D)/S,\crys},
	\theta_{(X,D)/S,\Rcrys}.
\]
By \cite{NS} Proposition 1.6.4, we have the equality of functors $Q_{X/S}^* \mathbb{R}\theta_{X/S,\crys *}\simeq \mathbb{R}\theta_{X/S,\Rcrys *}Q_{X^\bullet/S}^*$.
There is an isomorphism $\mathbb{R}a_{\crys *}\mathbb{R}\theta_{D/S,\crys *} = \mathbb{R}\theta_{X/S,\crys *} \mathbb{R}a^\bullet_{\crys *}$ by \cite{NS} (1.6.0.13).
Since $a$ and $a^\bullet$ are closed immersions, we see
\[
	\mathbb{R}a_{\crys *}=a_{\crys *},\
	\mathbb{R}a_{\crys *}^\bullet=a_{\crys *}^\bullet
\]
(\cite{B} III Corollaire 2.3.2).
So we have the following triangle
\begin{align*}
	Q_{X/S}^* \mathbb{R}\theta_{X/S,\crys *}L_{X^\bullet/S}(\Omega_{Y^\bullet/S}^\bullet )
	&\to
	Q_{X/S}^*
	\mathbb{R}\theta_{X/S,\crys *}L_{X^\bullet/S}(\Lambda_{(Y^\bullet,Z^\bullet)/S}^\bullet)\\
	&\to
	Q_{X/S}^*a_{\crys *}\mathbb{R}\theta_{D/S,\crys *} L_{D^\bullet/S}(\Omega_{Z^\bullet/S}^\bullet)\{-1\}
	\xrightarrow{+}.
\end{align*}

Let $\epsilon:(X,D)\to X$ and $\epsilon^\bullet:(X^\bullet,D^\bullet)\to X^\bullet$ be the canonical morphisms of log schemes.
By the cohomological descent (\cite{NS} Lemma 1.5.1),
we have the natural isomorphisms (in derived categories)
\begin{align*}
	\mathcal{O}_{X/S}
	&\simeq
	\mathbb{R}\theta_{X/S,\crys *}\mathcal{O}_{X^\bullet/S},
	\\
	\mathcal{O}_{(X,D)/S}
	&\simeq
	\mathbb{R}\theta_{(X,D)/S,\crys *}\mathcal{O}_{(X^\bullet,D^\bullet)/S},
	\\
	\mathcal{O}_{D/S}
	&\simeq
	\mathbb{R}\theta_{D/S,\crys *}\mathcal{O}_{D^\bullet/S}.
\end{align*}
By \cite{NS} Proposition 2.2.7, we have isomorphisms
\begin{align*}
	\mathcal{O}_{X^\bullet/S}
	&\simeq
	L_{X^\bullet/S}(\Omega_{Y^\bullet/S}^\bullet),
	\\
	\mathcal{O}_{(X^\bullet,D^\bullet)/S}
	&\simeq
	L_{(X^\bullet,D^\bullet)/S}(\Lambda_{(Y^\bullet,Z^\bullet)/S}^\bullet),
	\\
	\mathcal{O}_{D^\bullet/S}
	&\simeq
	L_{D^\bullet/S}(\Omega_{D^\bullet/S}^\bullet).
\end{align*}
We also have isomorphisms
\begin{align*}
	\mathbb{R}\theta_{X/S,\crys *}L_{X^\bullet/S} (\Lambda_{(Y^\bullet,Z^\bullet)/S}^\bullet)
	&\simeq
	\mathbb{R}\theta_{X/S,\crys *}\mathbb{R}\epsilon_*^\bullet L_{(X^\bullet,D^\bullet)/S} (\Lambda_{(Y^\bullet,Z^\bullet)/S}^\bullet)
	\\
	&\simeq
	\mathbb{R}\epsilon_* \mathbb{R}\theta_{(X,D)/S,\crys *} L_{(X^\bullet,D^\bullet)/S} (\Lambda_{(Y^\bullet,Z^\bullet)/S}^\bullet).
\end{align*}

Hence we have the following triangle
\begin{align*}
	Q_{X/S}^* (\mathcal{O}_{X/S})
	\to
	Q_{X/S}^* \mathbb{R}\epsilon_* (\mathcal{O}_{(X,D)/S})
	\to
	Q_{X/S}^*a_{\crys *}(\mathcal{O}_{D/S})\{-1\}\xrightarrow{+}.
\end{align*}

From this triangle, we have the following boundary morphism
\begin{align*}
	G_{D/X}^{\crys}:Q_{X/S}^* a_{\crys *}(\mathcal{O}_{D/S})\{ -1 \}
	\to
	Q_{X/S}^*(\mathcal{O}_{X/S})[1]
\end{align*}
in $D^+(Q_{X/S}^*(\mathcal{O}_{X/S}))$.
Applying the global section functor, we obtain a morphism
\[
	G_{D/X}^{\crys}:
	\mathbb{R}\Gamma ((X/S)_{\Rcrys}, Q_{X/S}^* a_{\crys *}(\mathcal{O}_{D/S}))\{-1\}
	\to
	\mathbb{R}\Gamma ((X/S)_{\Rcrys}, Q_{X/S}^*(\mathcal{O}_{X/S}))[1].
\]
By \cite{B} V Proposition 1.3.1. (1.3.3), the left hand is identified to
\begin{align*}
	\mathbb{R}\Gamma ((X/S)_{\Rcrys}, Q_{X/S}^* a_{\crys *}(\mathcal{O}_{D/S}))\{ -1 \}
	&\simeq
	\mathbb{R}\Gamma ((X/S)_{\crys}, a_{\crys *}(\mathcal{O}_{D/S}))\{ -1 \} \\
	&\simeq
	\mathbb{R}\Gamma((D/S)_{\crys}, \mathcal{O}_{D/S})\{ -1 \}
\end{align*}
and the right hand is identified to
\begin{align*}
	\mathbb{R}\Gamma ((X/S)_{\Rcrys}, Q_{X/S}^*(\mathcal{O}_{X/S}))[1]
	\simeq
	\mathbb{R}\Gamma((X/S)_{\crys}, \mathcal{O}_{X/S})[1].
\end{align*}

Therefore we have the Gysin map
\[
	G_{D/X}^{\crys} :
	\mathbb{R}\Gamma((D/S)_{\crys}, \mathcal{O}_{D/S})\{-1\}
	\to
	\mathbb{R}\Gamma((X/S)_{\crys}, \mathcal{O}_{X/S})[1].
\]
The Gysin map $G_{D/X}^{\crys}$ is independent of
the choice of the open covering $X=\bigcup_{i\in I_0}X^{i_0}$
and the diagram of embeddings
$(X^\bullet,D^\bullet)\hookrightarrow (Y^\bullet,Z^\bullet)$ (\cite{NS} Proposition 2.8.2).

We define the Gysin map of the de Rham-Witt complex
\[
	G_{D/X,m}^{\dRW} : W_m\Omega_{D/S}^\bullet\{-1\}
	\to
	W_m\Omega_{X/S}^\bullet [1]
\]
by the boundary morphism of the exact sequence
\[
	0
	\to
	W_m\Omega_{X/S}^\bullet
	\to
	W_m\Lambda_{(X,D)/S}^\bullet
	\xrightarrow{\Res}
	W_m\Omega_{D/S}^\bullet \{-1\}
	\to
	0.
\]
Since restriction maps are surjective, we also have the exact sequence:
\[
	0
	\to
	W\Omega_{X/S}^\bullet
	\to
	W\Lambda_{(X,D)/S}^\bullet
	\xrightarrow{\Res}
	W\Omega_{D/S}^\bullet \{-1\}
	\to
	0.
\]
Similarly we define $G_{D/X}^{\dRW}:W\Omega_{D/S}^\bullet \{-1\}\to W\Omega_{X/S}^\bullet [1]$.

We consider the compatibility of these Gysin maps.
Let $S$ be a scheme over $\mathbb{Z}_{(p)}$ in which $p$ is nilpotent,
$X$ a smooth scheme over $S$ and $D$ a smooth divisor of $X$ over $S$.
We imitate the method in \cite{NS}  \S 2.4 to make a simplicial log Frobenius lift.

Take an affine open covering $X=\bigcup_{i\in I_0}X^i$ of $X$ such that
there exists an \'{e}tale morphism $X^i\to \mathbb{A}_S^{k_i}$,
and $D^i:=D\cap X^i = \emptyset$ or $D^i$ is defined by the image of
$T_1\in \mathcal{O}_{\mathbb{A}_S^{k_i}}$ in $\mathcal{O}_{X^i}$.
Then each $X^i$ (resp. $D^i$) has a canonical Frobenius lift
(in the sense of \cite{LZDRW} \S 3.1) $\{X^i_m\}_m$ (resp. $\{D^i_m\}_m$)
and there is a morphism $\{D^i_m\}_m\to \{X^i_m\}_m$ of systems which is compatible with the structure of Frobenius lifts.
For $\underline{i}=( i_0,\ldots,i_r) \in I$,
we set $X^{\underline{i}}:=\bigcap_{\alpha = 0}^r X^{i_\alpha}$
and $D^{\underline{i}}:=\bigcap_{\alpha = 0}^r D^{i_\alpha}$.
Let $X^{(i_\alpha,\underline{i})}_m$ be the open subscheme of $X^{i_\alpha}_m$ defined by
the image of $X^{\underline{i}} \to X_m^{i_\alpha}$.
It is easy to see that the induced morphism
$X^{\underline{i}} \to X^{(i_\alpha,\underline{i})}_m$ is a closed immersion.
Set $D^{(i_\alpha,\underline{i})}_m:=D^{i_\alpha}_m\cap X^{(i_\alpha,\underline{i})}_m$ and $X^{\prime \underline{i}}_m:=\times_{W_m(S),\alpha = 0}^r X^{(i_\alpha,\underline{i})}_m$.
The closed immersion $X^{\underline{i}} \hookrightarrow X^{(i_\alpha,\underline{i})}_m$ induce
a closed diagonal immersion $X^{\underline{i}} \hookrightarrow X^{\prime \underline{i}}_m$.
We denote by $b:X^{\prime \prime \underline{i}}_m \to X^{\prime \underline{i}}_m$ the blow up of
$X^{\prime \underline{i}}_m$ along $D^{\prime \underline{i}}_m:=\times_{W_m(S),\alpha = 0}^r D^{(i_\alpha,\underline{i})}_m$.
We consider the complement $X^{\underline{i}}_m$ of the strict transform of
\[
	\bigcup_{\beta = 0}^r
	(X^{(i_0,\underline{i})}_m\times \cdots \times X^{(i_{\beta-1},\underline{i})}_m\times D^{(i_{\beta},\underline{i})}_m\times X^{(i_{\beta+1},\underline{i})}_m\times \cdots \times X^{(i_r,\underline{i})}_m)
\]
in $X^{\prime \prime \underline{i}}_m$ ,where fibered products $\times$ mean $\times_{W_m(S)}$, fibered products over $W_m(S)$.
Let $D^{\underline{i}}_m=X^{\underline{i}}_m \cap b^{-1}(D^{\prime \underline{i}}_m)$ be the exceptional divisor on $X^{\underline{i}}_m$.
Then $D^{\underline{i}}_m$ is a smooth divisor on $X^{\underline{i}}_m$ by \cite{NS} Theorem 2.4.2.
Considering the strict transform of the image of $X^{\underline{i}}$ of the diagonal embedding in $X^{\prime \underline{i}}_m$,
we have a closed immersion $X^{\underline{i}}\hookrightarrow X^{\underline{i}}_m$.
Moreover, we have $D^{\underline{i}}_m \times_{X^{\underline{i}}_m} X^{\underline{i}}\simeq D^{\underline{i}}$.

We interpret \cite{NS} Theorem 2.4.2 in our situation. We consider the case $D^{(i_\alpha,\underline{i})}_m \neq \emptyset$ for all $0\le \alpha \le r$.
Then the closed immersion $D^{(i_\alpha,\underline{i})}_m\hookrightarrow X^{(i_\alpha,\underline{i})}_m$ is defined by a global section $x^{(i_\alpha,\underline{i})}_m$ of $X^{(i_\alpha,\underline{i})}_m$, which corresponds to the image of $T_1$ of $\mathbb{A}_{W_m(S)}^{k_{i_\alpha}}$ under the map
$X^{(i_\alpha,\underline{i})}_m\to X^{i_\alpha}_m\to \mathbb{A}_{W_m(S)}^{k_{i_\alpha}}$.
Let $\mathcal{A}^{\underline{i}}_m :=\boxtimes_{\alpha =0}^r \mathcal{O}_{X^{(i_\alpha,\underline{i})}_m}$ be the structure sheaf of $X^{\prime \underline{i}}_m$.
Then $X^{\underline{i}}_m$ is the spectrum over $W_m(S)$ of the following sheaf of algebras
\[
	\mathcal{B}^{\underline{i}}_{m,\alpha} :=
	\mathcal{A}^{\underline{i}}_m [u_{m, 1}^{\pm 1}, \ldots, u_{m,  r}^{\pm 1}]
	/(x^{(i_\alpha,\underline{i})}_m-u_{m,\alpha} x^{(i_0,\underline{i})}_m \mid 1 \le \alpha \le r),
\]
where $u_{m,\alpha}$ are independent indeterminants and $D^{\underline{i}}_m$ corresponds to the equation $x^{(i_0,\underline{i})}_m=0$. $\{X^{\underline{i}}_m\}_m$ has a natural structure of Frobenius lift.
In fact, the natural morphism $X^{i_\alpha}_{m-1}\to X^{i_\alpha}_{m}$ induces $X^{(i_\alpha,\underline{i})}_{m-1}\to X^{(i_\alpha,\underline{i})}_{m}$ and it maps $x^{(i_\alpha,\underline{i})}_{m}$ to $x^{(i_\alpha,\underline{i})}_{m-1}$.
Hence we obtain a natural map $\mathcal{B}^{\underline{i}}_{m} \to \mathcal{B}^{\underline{i}}_{m-1}$ and it satisfies $W_{m-1}(\mathcal{O}_S)\otimes_{W_{m}(\mathcal{O}_S)}\mathcal{B}^{\underline{i}}_{m} \simeq \mathcal{B}^{\underline{i}}_{m-1}$.
Since the following diagram commutes
\[
\xymatrix{
	X^{i_\alpha}_{m-1} \ar[d]_{\Phi_{m}^{i_\alpha}}
	&
	\ar[l] X^{i_\alpha}_{1}=X^{i_\alpha}
	&
	\ar[l] X^{i_\alpha}/pX^{i_\alpha} \ar[d]^{\text{Frob}}
	\\
	X^{i_\alpha}_{m}
	&
	\ar[l] X^{i_\alpha}_{1}=X^{i_\alpha}
	&
	\ar[l] X^{i_\alpha}/pX^{i_\alpha}
}
\]
and the absolute Frobenius map $\text{Frob}$ and horizontal arrows are homeomorphisms, $\Phi_{m}^{i_\alpha}$ is also a homeomorphism.
It induces a map $\Phi_{m}^{(i_\alpha,\underline{i})}:X^{(i_\alpha,\underline{i})}_{m-1}\to X^{(i_\alpha,\underline{i})}_{m}$.
The family $\{\Phi_{m}^{(i_\alpha,\underline{i})}\}_\alpha$ defines a map $\mathcal{A}^{\underline{i}}_m\to \mathcal{A}^{\underline{i}}_{m-1}$.
Since it maps $x^{(i_\alpha,\underline{i})}_m$ to $\left (x^{(i_\alpha,\underline{i})}_{m-1} \right)^p$,
we can extend this map to $\mathcal{B}^{\underline{i}}_m\to \mathcal{B}^{\underline{i}}_{m-1}$ by $u_{m,\alpha} \mapsto u_{m-1,\alpha}^p$.

For $\Delta_m$, we use the following commutative diagram:
\[
\xymatrix{
	X^{\underline{i}} \ar[d] \ar[r]_-{\mathbf{w}_0} \ar@/^1pc/[rr]
	&
	W_m(X^{\underline{i}})\ar[d]
	&
	X^{(i_\alpha,\underline{i})}_m \ar[d]
	\\
	X^{i_\alpha}\ar[r]^-{\mathbf{w}_0}
	&
	W_m(X^{i_\alpha})\ar[r]^{\Delta_m}
	&
	X^{i_\alpha}_m.
}
\]

Since $X^{\underline{i}} \to X^{i_\alpha}$ is an open immersion, the morphism $W(X^{\underline{i}}) \to W(X^{i_\alpha})$ is also an open immersion.
We also know $X^{\underline{i}}$ and $W_m(X^{\underline{i}})$ have the same underlying topological space. Since $X^{(i_\alpha,\underline{i})}_m \to X^{i_\alpha}_m$ is also open,
there is a unique map $\Delta_m:W_m(X^{\underline{i}})\to X^{(i_\alpha,\underline{i})}_m$ which makes the diagram commutative.
These maps define $\Delta_m:W_m(X^{\underline{i}})\to X^{\prime \underline{i}}_m$.
We can define $\Delta_m:W_m(X^{\underline{i}})\to X^{\underline{i}}_m$
using this map and sending $u_{m,\alpha}$ to $[u_{1,\alpha}]$.

We have the following cartesian diagram:
\[
\xymatrix{
	\ar @{} [dr] |{\Box}
	X^\bullet \ar[r]
	&
	X_m^\bullet
	\\
	D^\bullet \ar[r] \ar[u]
	&
	D_m^\bullet \ar[u]
}
\]

Let $\overline{X}_m^\bullet$
(resp. $\overline{D}_m^\bullet$)
be the pd-envelope of the closed immersion $X^\bullet \hookrightarrow X_m^\bullet$
(resp. $D^\bullet \hookrightarrow D_m^\bullet$).
By \cite{NS} Lemma 2.2.16 (2), the natural morphism $\overline{D}_m^\bullet \to \overline{X}_m^\bullet \times_{X_m^\bullet} D_m^\bullet$ is an isomorphism.

Let $a_{\Zar}:D_{\Zar}\to X_{\Zar}$ be the canonical morphism of Zariski topoi and $\theta_{X,\Zar}:X_{\Zar}^\bullet \to X_{\Zar}, \theta_{D,\Zar}:D_{\Zar}^\bullet \to D_{\Zar}$ be the augmentation morphism.
Then the following commutative diagram shows the compatibility of Gysin maps:
\[
\xymatrix{
	\mathbb{R}a_{\Zar *} \mathbb{R}u_{D/W_m(S)*}\mathcal{O}_{D/W_m(S)}\{-1\} \ar[d]_\sim \ar[r]^{G_{D/X}^{\crys}}
	&
	\mathbb{R}u_{X/W_m(S)*}\mathcal{O}_{X/W_m(S)}[1] \ar[d]^\sim
	\\
	\mathbb{R}a_{\Zar *}\mathbb{R}\bar{u}_{D/W_m(S)*}Q_{D/W_m(S)}^*\mathcal{O}_{D/W_m(S)}\{-1\} \ar[d]_\sim \ar[r]
	&
	\mathbb{R}\bar{u}_{X/W_m(S)*}Q_{X/W_m}^*\mathcal{O}_{X/W_m}[1] \ar[d]^\sim
	\\
	\mathbb{R}a_{\Zar *} \mathbb{R} \theta_{D,\Zar *}
	(\Omega_{\overline{D}_m^\bullet/W_m(S)}^\bullet) \{-1\}
	\ar[r]^-{G_{\overline{D}_m^\bullet/\overline{X}_m^\bullet}^{\dR}}
	\ar[d]_\sim
	&
	\mathbb{R}\theta_{X,\Zar *}
	(\Omega_{\overline{X}_m^\bullet/W_m(S)}^\bullet)[1] \ar[d]^\sim
	\\
	\mathbb{R}a_{\Zar *} \mathbb{R}\theta_{D,\Zar *} (W_m\Omega_{D^\bullet/S}^\bullet) \{-1\} \ar[r]^{G_{D^\bullet/X^\bullet}^{\dRW}}
	&
	\mathbb{R}\theta_{X,\Zar *} (W_m\Omega_{X^\bullet/S}^\bullet)[1],
}
\]
where
$
\Omega_{\overline{D}_m^\bullet/W_m(S)}^\bullet
:=
\mathcal{O}_{\overline{D}_m^\bullet} \otimes_{\mathcal{O}_{D_m^\bullet}} \Omega_{D_m^\bullet/W_m(S)}^\bullet,
\Omega_{\overline{X}_m^\bullet/W_m(S)}^\bullet
:=
\mathcal{O}_{\overline{X}_m^\bullet} \otimes_{\mathcal{O}_{X_m^\bullet}} \Omega_{X_m^\bullet/W_m(S)}^\bullet
$.

Finally, we consider the relation between the boundary map of $E_1$-term of the $p$-adic weight spectral sequence and Gysin maps.
Let $Y$ be a stricly semistable log scheme over $S=\Spec (R,\bN)$.
We use the convension of \S \ref{subsect:comparison-semistable} and \S \ref{section:Poincare}.
Let $G_{q}^{\dRW} : W_m\Omega_{\mathring{Y}^{(j+1)}/R}^\bullet \{-1\}
\to W_m\Omega_{\mathring{Y}^{(j)}/R}^\bullet [1]$ be
the Gysin map corresponding to different immersions $\iota^{(q)}:\mathring{Y}^{(j+1)}\to \mathring{Y}^{(j)}$.
We set $G^{\dRW}:=\sum_{q=1}^{j+1}(-1)^{q+1}G_{q}^{\dRW}$,
and let $d^1$ be the boundary morphism of the exact sequence
\[
	0
	\to
	Gr_jW_m\widetilde{\Lambda}^\bullet
	\to
	(P_{j+1}/P_{j-1})W_m\widetilde{\Lambda}^\bullet
	\to
	Gr_jW_m\widetilde{\Lambda}^\bullet
	\to
	0.
\]

\begin{prop}
\label{Gys2} (cf. \cite{Mo} Proposition 4.11)
The following diagram is commutative:
\[
\xymatrixcolsep{4pc}
\xymatrix{
	Gr_{j+1}W_m\widetilde{\Lambda}^\bullet \ar[r]^{d^1} \ar[d]_{\Res}
	&
	Gr_j W_m\widetilde{\Lambda}^\bullet [1] \ar[d]^{\Res [1]}
	\\
	W_m\Omega_{\mathring{Y}^{(j+1)}/R}^\bullet \{-j-1\} \ar[r]^{G^{\dRW}\{-j\}}
	&
	W_m\Omega_{\mathring{Y}^{(j)}/R}^\bullet [1]\{-j\}.
}
\]
\end{prop}

\begin{proof}
Let $J=\{\alpha_1,\ldots,\alpha_{j+1}\}$ be a subset of $[1,d]$
and $J_q=\{\alpha_1,\ldots,\widehat{\alpha_q},\ldots,\alpha_{j+1} \}$.
The residue morphism $W_m\Omega_{\mathring{Y}_{J_q}/R}^{\bullet}\{-j\} \to Gr_jW_m\widetilde{\Lambda}^\bullet$ naturally extends to a morphism
$W_m\Lambda_{(\mathring{Y}_{J_q}, \mathring{Y}_J)/R}^{\bullet} \{-j\} \to (P_{j+1}/P_{j-1})W_m\widetilde{\Lambda}^\bullet$.
The commutativity follows from the following commutative diagram with exact rows:
\[
\xymatrix{
	0 \ar[r]
	&
	W_m\Omega_{\mathring{Y}_{J_q}}^\bullet \{-j\} \ar[d] \ar[r]
	&
	W_m\Lambda_{(\mathring{Y}_{J_q}, \mathring{Y}_J)}^{\bullet} \{-j\} \ar[d] \ar[r]^{(-1)^{q+1} \Res}
	&
	W_m\Omega_{\mathring{Y}_J}^\bullet \{-j-1\} \ar[r] \ar[d]
	&
	0
	\\
	0 \ar[r]
	&
	Gr_jW_m\widetilde{\Lambda}^\bullet \ar[r]
	&
	(P_{j+1}/P_{j-1})W_m\widetilde{\Lambda}^\bullet \ar[r]
	&
	Gr_{j+1}W_m\widetilde{\Lambda}^\bullet \ar[r]
	&
	0.
}
\]
\end{proof}

\begin{prop}
\label{prop:boudarymorphism} (cf. \cite{Nak} Theorem 10.1)

Let $\star$ be a positive integer or nothing.
Under the residue isomorphism,
\[
	d^1:\mathbb{H}_{\Zar}^h (\mathring{Y}, Gr_{k}W_\star A^\bullet)
	\to
	\mathbb{H}_{\Zar}^h (\mathring{Y}, Gr_{k-1}W_\star A^\bullet)
\]
is identified with the following morphism:
\begin{align*}
	&\sum_{j\ge \max \{-k,0 \}}
	[(-1)^j G^{\crys}\{-2j-k+1\}+(-1)^{j+k}\rho_\star\{-2j-k\}]:\\
	&\bigoplus\limits_{j\ge \max \{-k,0 \}}
	H_{\crys}^{h-2j-k} (\mathring{Y}^{(2j+k+1)}/W_\star(R))(-j-k)
	\to
	\\
	&\bigoplus\limits_{j\ge \max \{-k+1,0 \}}
	H_{\crys}^{h-2j-k+2} (\mathring{Y}^{(2j+k)}/W_\star(R))(-j-k+1),
\end{align*}
where $\rho_\star$ is the morphism defined in Lemma \ref{lem:Gys1}.
\end{prop}

\begin{proof}
We can copy the proof of \cite{Nak} Theorem 10.1 using Proposition \ref{lem:Gys1} and Proposition \ref{Gys2}.
\end{proof}

\subsection{Degeneration of weight spectral sequence}

In this section, we prove that the weight spectral sequence degenerates up to torsion
if the base scheme is a spectrum of a (not necessarily perfect) field using the method of \cite{Nak}.
Let $Y$ be a proper strictly semistable log scheme over a field $k$ of characteristic $p > 0$.

Let $s=(\Spec k, \bN\oplus k^*)$ be a log point with structure morphism defined by
$\bN \oplus k^* \ni (a,u) \mapsto 0$ for $a \neq 0$ and $(0,u) \mapsto u$.
By \cite{Na} Lemma 2.2,
there is a subring $A_1$ of $k$ which is finitely generated over $\mathbb{F}_p$
and a proper strictly semistable log scheme $\mathfrak{Y}$ over $s_1=(\Spec A_1,\bN\oplus A_1^*)$
with the structure morphism defined by $\bN \oplus A_1^* \ni (a,u) \mapsto 0$
for $a \neq 0$ and $(0,u) \mapsto u$ such that $\mathfrak{Y} \times_{s_1} s=Y$.
We can assume $A_1$ is smooth over $\mathbb{F}_p$.
Lift $A_1$ to a $p$-adically complete formally smooth algebra $A$ over $W(\mathbb{F}_p)=\mathbb{Z}_p$.
Let $S:=(\text{Spf}A,\bN\oplus A^*)$ be a $p$-adically log formal scheme over $\text{Spf}(\mathbb{Z}_p, \mathbb{Z}_p^*)$
such that the log structure of $S$ is induced by $\bN\oplus A^*\ni (a,u)\mapsto 0$ for $a\neq 0$
and $(0,u)\mapsto u$. $S$ has the pd-ideal $p\mathcal{O}_S$
and it defines the exact closed immersion $s_1 \hookrightarrow S$.
For an affine log formal open subscheme $T$ of $S$,
let $T_1 := T \otimes_{\mathbb{Z}_p} \mathbb{F}_p$ be its reduction.
We fix a lift of Frobenius $F_T : T \to T$  of $T_1$.
Set $\mathfrak{Y}_{T_1} := \mathfrak{Y} \times_{S_1} T_1$.
If $t$ is a closed point of $T_1$, set
$\mathfrak{Y}_t := \mathfrak{Y}_{T_1} \times_{T_1} t$.
In this situation, the canonical inclusion $A_1 \hookrightarrow k$ factors
$\mathcal{O}_{T_1}$. Let $\mathcal{O}_T\to W(k_t)$ (resp. $\mathcal{O}_T\to W(k))$
be the composition of the map $\mathcal{O}_T \to W(\mathcal{O}_{T_1})$ from
\cite{ICrys} (0.1.3.20) with the natural surjection $W(\mathcal{O}_{T_1}) \to W(k_t)$
(resp. the natural inclusion $W(\mathcal{O}_{T_1}) \hookrightarrow W(k)$).
We consider $W(k_t)$ and $W(k)$ as $\mathcal{O}_T$-algebra via these maps.

\begin{prop}
\label{prop:crysnoncanobc}
(cf. \cite{Nak} Proposition 3.2)
There exists an affine log formal open subscheme $T$ of $S$
such that the canonical morphism
\[
	H_{\logcrys}^h (\mathfrak{Y}_{T_1}/T) \otimes_{\mathcal{O}_T} W(k)
	\to
	H_{\logcrys}^h (Y/W(k,\bN))
\]
is an isomorphism.
\end{prop}

\begin{proof}
For an affine log formal open subscheme $T$ of $S$, we find
\[
	\mathbb{R}\Gamma_{\logcrys}(\mathfrak{Y}_{T_1}/T_n) \otimes_{\mathcal{O}_{T_n}}^{\mathbb{L}} W_n(k)
	\simeq
	\mathbb{R}\Gamma_{\logcrys} (Y/W_n(k,\bN))
\]
by the base change theorem (\cite{KatoLog} (6.10)).

Let $(P^\bullet,d^\bullet)$ be a strictly perfect complex (Definition \ref {defn:strictly-perfect-cpx} (2))
which represents $\mathbb{R}\Gamma_{\logcrys}(\mathfrak{Y}_{T_1}/T)$. Then
\begin{align*}
	\mathbb{R}\Gamma_{\logcrys}(\mathfrak{Y}_{T_1}/T_n)
	\otimes_{\mathcal{O}_{T_n}}^{\mathbb{L}}W_n(k)
	&\simeq
	\mathbb{R}\Gamma_{\logcrys}(\mathfrak{Y}_{T_1}/T)
	\otimes_{\mathcal{O}_{T}}^{\mathbb{L}} \mathcal{O}_{T_n}
	\otimes_{\mathcal{O}_{T_n}}^{\mathbb{L}} W_n(k)\\
	&\simeq P^\bullet \otimes_{\mathcal{O}_{T}} W_n(k).
\end{align*}
Since $P^\bullet \otimes_{\mathcal{O}_{T}} W_n(k)$ satisfies the Mittag-Leffler condition
\begin{align*}
	\mathbb{R}\Gamma_{\logcrys}(Y/W(k,\bN))
	&= \mathbb{R}\varprojlim_n \mathbb{R}\Gamma_{\logcrys}(Y/W_n(k,\bN))\\
	& \simeq \mathbb{R}\varprojlim_n (\mathbb{R}\Gamma_{\logcrys}(\mathfrak{Y}_{T_1}/T_n)\otimes_{\mathcal{O}_{T_n}}^{\mathbb{L}}W_n(k))\\
	& \simeq \mathbb{R}\varprojlim_n (P^\bullet \otimes_{\mathcal{O}_T}W_n(k))\\
	& \simeq \varprojlim_n (P^\bullet \otimes_{\mathcal{O}_T} W_n(k))\\
	& \simeq P^\bullet \otimes_{\mathcal{O}_T} W(k).
\end{align*}
By \cite{Nak} Lemma 3.1, we can suppose $\text{Tor}_1^{\mathcal{O}_T}(L|_T,W(k))=0$
for $L=H^j(P^\bullet)$ and $\text{Im}(d^j)$ for any $j$ by shrinking $T$ if necessary.
Then we get
\begin{align*}
	H_{\logcrys}^h (\mathfrak{Y}_{T_1}/T) \otimes_{\mathcal{O}_T} W(k)
	&= H^h (P^\bullet) \otimes_{\mathcal{O}_T} W(k) \\
	&= H^h (P^\bullet \otimes_{\mathcal{O}_T} W(k))\\
	&= H_{\logcrys}^h (Y/W(k,\bN)).
\end{align*}
\end{proof}

\begin{thm}
\label{thm:e2degeneration}
(\cite{Nak} Proposition 3.5, Theorem 3.6)
Let $Y$ be a semistable log scheme over any field of characteristic $p>0$
and $K$ be the fraction field of $W(k)$.
The $p$-adic weight spectral sequence (Theorem \ref{thm:weight-spectral-sequence} (2))
degenerates at $E_2$ after tensoring with $K$.
\end{thm}

\begin{proof}
By \cite{Nak} Corollary 3.4 and Proposition \ref{prop:crysnoncanobc}, there exists an affine log formal scheme $T$ of $S$ such that for any closed point $t\in T$ the canonical morphisms
\begin{align*}
	H_{\logcrys}^h(\mathfrak{Y}_{T_1}/T)\otimes_{\mathcal{O}_T} W(k_t)&\to H_{\logcrys}^h(\mathfrak{Y}_t/W(k_t,\bN)),\\
	H_{\logcrys}^h(\mathfrak{Y}_{T_1}/T)\otimes_{\mathcal{O}_T} W(k)&\to H_{\logcrys}^h(Y/W(k,\bN))
\end{align*}
are isomorphisms.
By Deligne's remark (\cite{IReport} (3.10)),
we can assume there exists a finitely generated $\mathbb{Z}_p$-module $M$
such that $H_{\logcrys}^h(\mathfrak{Y}_{T_1}/T) \simeq M\otimes_{\mathbb{Z}_p}\mathcal{O}_T$ by shrinking $T$ if necessary.

By Corollary 3.4 of \cite{Nak} and Proposition \ref{prop:crysnoncanobc},
there exists isomorphisms
\begin{align*}
	H_{\crys}^h (\mathring{Y}^{(j)}/W(k))
	&\simeq
	H_{\crys}^h (\mathring{\mathfrak{Y}}_{T_1}^{(j)}/\mathring{T}) \otimes_{\mathcal{O}_T} W(k),\\
	H_{\crys}^h (\mathring{\mathfrak{Y}}_t^{(j)}/W(k_t))
	&\simeq
	H_{\crys}^h (\mathring{\mathfrak{Y}}_{T_1}^{(j)}/\mathring{T}) \otimes_{\mathcal{O}_T} W(k_t)
\end{align*}
for all $j$ and for all closed points $t$ of $T$ by shrinking $T$ if necessary. Set
\[
	F^{-k,h+k}:=
	\bigoplus_{j\ge \max \{-k,0\}} H_{\crys}^{h-2j-k} (\mathring{\mathfrak{Y}}_{T_1}^{(2j+k+1)}/\mathring{T})
\]
and
\[
	G^{-k,h+k}:=
	\ker (F^{-k,h+k} \to F^{-k+1,h+k})/
	\text{image}(F^{-k-1,h+k} \to F^{-k,h+k}),
\]
where the morphisms $F^{-k,h+k} \to F^{-k+1,h+k}$ and $F^{-k-1,h+k} \to F^{-k,h+k}$ are
the sums of the induced morphisms of closed immersions
and Gysin maps as in Proposition \ref{prop:boudarymorphism}.

By the base change of Gysin maps of crystalline cohomology (\cite{B} VI Theorem 4.3.12)
and by \cite{Nak} Lemma 3.1, we obtain
\begin{align*}
	E_2^{-k,h+k} (Y/W(k,\bN))
	&= G^{-k,h+k} \otimes_{\mathcal{O}_T} W(k)\\
	E_2^{-k,h+k}(\mathfrak{Y}_t/W(k_t,\bN))
	&= G^{-k,h+k} \otimes_{\mathcal{O}_T} W(k_t)
\end{align*}
for all $k,h$ by shrinking $T$ if necessary.
Using Deligne's remark, we can assume that there exists a finitely generated $\mathbb{Z}_p$-module $M^{-k,h+k}$
such that $M^{-k,h+k}\otimes_{\mathbb{Z}_p}\mathcal{O}_T\simeq G^{-k,h+k}$.
Let $K_t$ be the fraction field of $W(k_t)$. We have
\begin{align*}
	\dim_{K} (E_2^{-k,h+k} (Y/W(k,\bN)) \otimes_{W(k)} K)
	&= \dim_{\mathbb{F}_p} (M^{-k,h+k} \otimes_{\mathbb{Z}_p} \mathbb{F}_p)\\
	&= \dim_{K_t}(E_2^{-k,h+k} (\mathfrak{Y}_t/W(k_t,\bN)) \otimes_{W(k_t)} K_t),\\
	\dim_{K} (H_{\logcrys}^h (Y/W(k,\bN)) \otimes_{W(k)} K)
	&= \dim_{\mathbb{F}_p}(M \otimes_{\mathbb{Z}_p} \mathbb{F}_p)\\
	&= \dim_{K_t} (H_{\logcrys}^h (\mathfrak{Y}_t/W(k_t,\bN)) \otimes_{W(k_t)} K_t).
\end{align*}
By the purity of the weight of the crystalline cohomology (\cite{CL} Th\'eor\`eme 1.2),
this theorem is true for $\mathfrak{Y}_t/W(k_t,\bN)$
because $t$ is the spectrum of a finite field.
By the above calculation of dimensions,
we see this theorem is true for any field of characteristic $p>0$.
\end{proof}

\section{Weight spectral sequence and its degeneration for open smooth varieties}
\label{sect:ssq-open-smooth}

Let $R$ be a Noetherian $\mathbb{Z}_{(p)}$-algebra in which $p$ is nilpotent.
Let $X$ be a proper smooth scheme over $R$ and $D$ an SNCD on $X$ over $R$.
We consider the log scheme $(X, D)$ with respect to the Zariski topology.
By Theorem \ref{thm:comparison-theorem-NCD},
we have a canonical isomorphism
\[
	H_{\logcrys}^h ((X,D)/W(R))
	\simeq
	\mathbb{H}_{\Zar}^h (X,W\Lambda_{(X,D)/R}^\bullet).
\]
Let $D_1,\ldots ,D_d$ be the irreducible components of $X$.
For a subset $J = \{\alpha_1,\ldots, \alpha_j\}$ of $[1,d]$,
let $D_J$ be $D_{\alpha_1}\cap \cdots \cap D_{\alpha_j}$.
We set $D^{(j)}$ by $\coprod_{|J|=j}D_J$ for all nonnegative number $j$.
Then we can show that the canonical morphism
\[
	W\Omega_{D^{(j)}/R}^{\bullet-j}(-j)
	\to
	Gr_j W\Lambda_{(X,D)/R}^\bullet
\]
is an isomorphism as \S \ref{section:Poincare}.
We also call the map $Gr_j W\Lambda_{(X,D)/R}^\bullet \xrightarrow{\sim} W\Omega_{D^{(j)}/R}^{\bullet-j}(-j)$ the Poincar\'e residue map.
Using this, we obtain the following spectral sequence
\[
	E_{1}^{-k,h+k}
	=
	H_{\crys}^{h-k} (D^{(k)}/W(R))(-k)
	\Rightarrow
	H_{\logcrys}^h ((X,D)/W(R)),
\]
which we also call the $p$-adic weight spectral sequence.

\begin{thm}
When $R=k$ is a field, the $p$-adic weight spectral sequence degenerates at $E_2$  after tensoring with the fraction field of $W(k)$.
\end{thm}

\begin{proof}
The proof is the same as that of Theorem \ref{thm:e2degeneration} (cf. \cite{Nak} Theorem 5.2).
\end{proof}

\section{Overconvergent log de Rham-Witt complex in SNCD case}
\label{sect:overconv}

In this section, we extend the overconvergent de Rham-Witt complex of \cite{DLZODRW} to
log schemes associated to schemes with simple normal crossing divisor over a perfect field.
In this section, we work on the Zariski topology when we consider log structures and log de Rham-Witt complexes (See \S \ref{subsect:topology-of-log-structures}).

\subsection{Overconvergent log de Rham-Witt complex}

Let $k$ be a perfect field of positive characteristic $p$
and $K = W(k)[1/p]$ its fraction field.
Let $A = k[T_1,\ldots,T_n]$ be a polynomial ring.
We consider the pre-log ring $(A,\bN^d),\bN^d\ni e_i\mapsto T_i \in A$ for $d\le n$.
Recall that an element $\omega$ of $W\Lambda_{(A,\bN^d)/k}^\bullet$
is uniquely written as a convergent sum (Proposition \ref{prop:expression-basic-witt-sncd})
\[
	\omega = \sum_{k,\mathcal{P}} \epsilon (\xi_{k,\mathcal{P}},k,\mathcal{P}).
\]
In this section, we only consider the case that $J$ is empty.
Therefore we write $\epsilon (\xi_{k,\mathcal{P}},k,\mathcal{P})$
for $\epsilon (\xi_{k,\mathcal{P}, \emptyset},k,\mathcal{P}, \emptyset)$.

For a positive real number $\epsilon$ we define the Gauss norm $\gamma_\epsilon$ by
\[
	\gamma_\epsilon(\omega):=
	\inf_{k,\mathcal{P}} \{ \ord_p \xi_{k,\mathcal{P}} - \epsilon |k^+| \}
\]
where $|k^+|=(k^+)_1+\ldots + (k^+)_n$.
This is equal to
$\inf_{k,\mathcal{P}} \{ \ord_V \xi_{k,\mathcal{P}} - \epsilon |k^+| \}$
(see \cite{DLZODRW} (0.3)) because $\ord_p \xi = \ord_V \xi$ for $\xi \in W(k)$.

If $\gamma_\epsilon(\omega)>-\infty$, we say that $\omega$ has radius of convergence $\epsilon$.
We call $\omega$ overconvergent if there is an $\epsilon > 0$
such that $\omega$ has radius of convergence $\epsilon$. We find
\[
	\gamma_\epsilon (\omega_1+\omega_2)
	\ge
	\min(\gamma_\epsilon(\omega_1),\gamma_\epsilon(\omega_2))
\]
and overconvergent elements form
a sub differential algebra $W^\dagger\Lambda_{(A,\bN^d)/k}$ of $W\Lambda_{(A,\bN^d)/k}$
(cf. \cite{DLZODRW} pp. 200).

\begin{prop}
\label{prop:indep}
(cf. \cite{DLZODRW} Proposition 0.7, Proposition 0.9)

Let $\phi : (k[S_1,\ldots,S_n],\bN^d) \to (k[T_1,\ldots,T_m],\bN^{d'})$
be a morphism of pre-log rings over $k$. The map
\[
	\phi_* :
	W\Lambda_{(k[S_1,\ldots,S_n],\bN^d)/k}^\bullet
	\to
	W\Lambda_{(k[T_1,\ldots,T_m],\bN^{d'})/k}^\bullet
\]
induces
\[
	\phi^\dagger_* :
	W^\dagger\Lambda_{(k[S_1,\ldots,S_n],\bN^d)/k}^\bullet
	\to
	W^\dagger\Lambda_{(k[T_1,\ldots,T_m],\bN^{d'})/k}^\bullet.
\]
Moreover, $\phi^\dagger_*$ is surjective
when both $k[S_1,\ldots,S_n] \to k[T_1,\ldots,T_m]$ and $\bN^d \to \bN^{d'}$ are surjective.
\end{prop}

\begin{proof}
Let $\omega=\sum_{k,\mathcal{P}}\epsilon (\xi_{k,\mathcal{P}},k,\mathcal{P})$
be any element of $W^\dagger\Lambda_{(k[S_1,\ldots,S_n],\bN^d)/k}^\bullet$.
Since $\omega$ is overconvergent, there are $\epsilon > 0$ and $C\in \mathbb{R}$
such that $\ord_p \xi_{k,\mathcal{P}}-\epsilon |k^+| \ge C$ for all $k$ and $\mathcal{P}$.
For any subset $J$ of $[1,d]$,
we set $\omega_J:=\sum_{k,\mathcal{P}, I_{-\infty}=J}\epsilon (\xi_{k,\mathcal{P}},k,\mathcal{P})$.
Then we see that $\omega_J$ can be written as a form
\[
	\omega_J=
	\left( \prod_{i\in J}d\log X_i \right)
	\cdot \sum_{k',\mathcal{P}'} e(\zeta_{k',\mathcal{P}'}^J\ ,k',\mathcal{P}').
\]
Here $X_i := [T_i]$, $k' : [1,n]\setminus J \to \mathbb{Z}_{\ge 0}[1/p]$ runs over all  weights without log poles,
$\mathcal{P}'$ runs over all partitions of $\Supp k'$,
$\zeta_{k',\mathcal{P}'}^J\in {}^{V^{u(k')}}W(k)$
and $e(\zeta_{k',\mathcal{P}'}^J\ ,k',\mathcal{P}')$ is a basic Witt differential (in the sense of \cite{LZDRW}).
We see that $\omega=\sum_{J\subset [1,d]}\omega_J$
and that all coefficients $\zeta_{k',\mathcal{P}'}^J$ satisfy
$\ord_p \zeta_{k',\mathcal{P}'}^J-\epsilon |k'| \ge C$.
Hence we obtain $\bar{\omega}_{J} := \sum_{k',\mathcal{P}'} e(\zeta_{k',\mathcal{P}'}^J\ ,k',\mathcal{P}') \in W^\dagger \Omega_{k[S_1,\ldots,S_n]/k}^\bullet$.

By \cite{DLZODRW} Proposition 0.9,
we obtain $\phi(\bar{\omega}_{J}) \in W^\dagger \Omega_{k[T_1,\ldots,T_m]/k}^\bullet$.
We see that
\[
	\phi(\omega_{J}) =
	\phi \left( \prod_{i\in J}d\log X_i\right) \phi(\bar{\omega}_{J})
	\in
	W^\dagger\Lambda_{(k[T_1,\ldots,T_m],\bN^{d'})/k}^\bullet
\]
because $\phi(\prod_{i\in J}d\log X_i) \in W^\dagger\Lambda_{(k[T_1,\ldots,T_m],\bN^{d'})/k}^\bullet$
and $W^\dagger\Lambda_{(k[T_1,\ldots,T_m],\bN^{d'})/k}^\bullet$ is a ring.
This shows $\phi(\omega)\in W^\dagger\Lambda_{(k[T_1,\ldots,T_m],\bN^{d'})/k}^\bullet$.

We prove the last statement. If $\phi$ is surjective, we can construct a map
\[
	\psi:
	(k[T_1,\ldots,T_m],\bN^{d'})
	\to
	(k[S_1,\ldots,S_n],\bN^d)
\]
of pre-log rings such that $\phi \circ \psi = \text{id}$.
Then for any $\eta \in W^\dagger\Lambda_{(k[S_1,\ldots,S_n],\bN^d)/k}^\bullet$,
the element $\psi_*(\eta)$ belongs to $W^\dagger\Lambda_{(k[T_1,\ldots,T_m],\bN^{d'})/k}^\bullet$ and it satisfies $\phi_*\psi_*(\eta)=\eta$.
\end{proof}

Let $(B,P,\alpha)$ be a pre-log ring such that $B$ is a finitely generated $k$-algebra.
Then we can find a commutative diagram
\begin{align}
\label{diag:logringpres}
\xymatrix{
	\bN^d \ar[r] \ar[d]
	&
	A=k[T_1,\ldots,T_n] \ar[d]
	\\
	P \ar[r]^\alpha
	&
	B,
}
\end{align}
where the top morphism is given by $e_i\to T_i$ and the both vertical morphisms are surjective.
It induces a map between log de Rham-Witt complexes
$\lambda : W\Lambda_{(A,\bN^d)/k}^\bullet \to W\Lambda_{(B,P)/k}^\bullet$.

\begin{defn}
We define $W^\dagger \Lambda_{(B,P)/k}^\bullet$ as the image of $W^\dagger \Lambda_{(A,\bN^d)/k}^\bullet$ under the map $\lambda$.
We call $W^\dagger \Lambda_{(B,P)/k}^\bullet$
the overconvergent log de Rham-Witt complex for the pre-log ring $(B, P)$ over $k$.
\end{defn}

By Proposition \ref{prop:indep} (cf. \cite{DLZODRW} Definition 1.1),
this definition is independent of the choice of
the above diagram (\ref{diag:logringpres})
and the correspondence $(B,P)\mapsto W^\dagger \Lambda_{(B,P)/k}^\bullet$ is functorial.
Our definition of the overconvergent log de Rham-Witt complex is
an extension of the overconvergent de Rham-Witt complex of Davis-Langer-Zink,
i.e., $W^\dagger \Lambda_{(B,\{*\})/k}^\bullet \simeq W^\dagger \Omega_{B/k}^\bullet$.

\subsection{Comparison with log Monsky-Washnitzer cohomology}
\label{comparison-log-MW}

Let $k$ be a perfect field of char $p > 0$.
We consider a finitely generated, smooth algebra $\widetilde{B}$ over Witt ring $W(k)$ and $\mathfrak{X}:=\Spec \widetilde{B}$.
We assume there are (global) coordinates $\widetilde{t}_1,\ldots,\widetilde{t}_n$ of $\mathfrak{X}$,
i.e., the morphism $\mathfrak{X}\to \mathbb{A}_{W(k)}^n$ defined by $\widetilde{t}_1,\ldots,\widetilde{t}_n$ is \'etale.
Let $B=\widetilde{B}\otimes_{W(k)}k$ be the reduction of $\widetilde{B}$ to $k$
and $t_1,\ldots,t_n$ be images of $\widetilde{t}_1,\ldots,\widetilde{t}_n$ in $B$, and $X=\Spec B$.
We denote by $\mathcal{D}$ the divisor of $\mathfrak{X}$
which is defined by the equation $\widetilde{t}_1\cdots \widetilde{t}_d=0$ and $D$ its reduction to $X$.
Let $\widetilde{B}^\dagger$ be the weak completion of $\widetilde{B}$
with respect to $(p)\subset W(k)$ (in sense of \cite{MW} Definition 1.1).
Let $(B,\bN^d)$ (resp. $(\widetilde{B},\bN^d), (\widetilde{B}^\dagger,\bN^d)$)
be the pre-log ring defined by $e_i\mapsto t_i$ (resp. $e_i\mapsto \widetilde{t}_i$).

\begin{defn}
\label{def:Frob-endmorphism}
(cf. \cite{Tsu} \S3)
An endomorphism $\phi$ of $\widetilde{B}^\dagger$ is called Frobenius
if the following three conditions are satisfied:

(1) $\phi$ is compatible with the Frobenius map $F$ on $W(k)$,

(2) Its reduction to $B\simeq \widetilde{B}^\dagger/p\widetilde{B}^\dagger$
coincides with the absolute Frobenius on $B$,

(3) $\phi$ satisfies the relation $\phi(\widetilde{t}_i)=\widetilde{t}_i^p\cdot u_i, u_i\in 1+p\widetilde{B}^\dagger$ for $1\le i\le d$.
\end{defn}

By \cite{Chi} Lemma 3.3.1,
there exists a Frobenius endomorphism on $\widetilde{B}^\dagger$ in this situation.
Tsuzuki defined the logarithmic Monsky-Washnitzer cohomology $H_{\logMW}^*((X,D)/K)$
and proved that it depends only on $X$ and $D$ (\cite{Tsu} (3.3), Proposition 3.3.1).

To prove the comparison theorem between the logarithmic overconvergent de Rham-Witt cohomology and the logarithmic Monsky-Washnitzer cohomology,
we have to extend the overconvergent Witt lift of \cite{DLZODRW} \S 3.

By \cite{DLZODRW} Proposition 3.2, the map $t_\phi:\widetilde{B}^\dagger \to W(B)$ defined in \cite{ICrys} (0.1.3.20) has the image in $W^\dagger (B)$.
The map $s_\phi: \widetilde{B}^\dagger \to W(\widetilde{B}^\dagger)$ defined in \cite{ICrys} (0.1.3.16) maps $\widetilde{t}_i$ to
the unique element whose ghost components are $(\widetilde{t}_i,\phi(\widetilde{t}_i),\phi^2(\widetilde{t}_i),\ldots \ )$.
It easily follows by induction that $\phi^j(\widetilde{t}_i)$ is written as a form
$\widetilde{t}_i^{p^j}\cdot \beta_{i,j}$ where $\beta_{i,j}\in \widetilde{B}^\dagger$. Then
\[
	(\widetilde{t}_i,\phi(\widetilde{t}_i),\phi^2(\widetilde{t}_i),\ldots \ )
	=(\widetilde{t}_i,\widetilde{t}_i^p,\widetilde{t}_i^{p^2},\ldots \ )
	\cdot (1,\beta_{i,1},\beta_{i,2},\ldots \ ).
\]
Since $\beta_{i,j+1}=u_i^{p^j}\phi(\beta_{i,j})$ and $u_i^{p^j} \equiv 1 \mod p^j$,
we have $\beta_{i,j+1} \equiv \phi(\beta_{i,j}) \mod p^j$.
Using \cite{HBig} Lemma 1.1 with $S=\{p^m ; m \ge 1\}$,
we find there is a unique element $\nu_\phi(\widetilde{t}_i)$ of $W(\widetilde{B}^\dagger)$
whose ghost components are
\[
	(1,\beta_{i,1},\beta_{i,2},\ldots \ ).
\]
Let $\lambda_\phi(\widetilde{t}_i)\in W(B)$ be the image of $\nu_\phi(\widetilde{t}_i)$ in $W(B)$ via the projection map.

Take a presentation from a polynomial algebra $\widetilde{A}=W(k)[\widetilde{T}_1,\ldots,\widetilde{T}_N]\to \widetilde{B}$
such that $\widetilde{T}_i$ is mapped to $\widetilde{t}_i$ for $1\le i\le n$
and lift the Frobenius $\phi$ on $\widetilde{B}^\dagger$ to a Frobenius $F$ on $\widetilde{A}^\dagger$.

Following the notation used in \cite{DLZODRW} Proposition 3.1,
we define a pseudovaluation (cf. \cite{DLZOWV} Definition 1.4) $\mu_\epsilon$ on $\widetilde{A}^\dagger$ by
\[
	\mu_\epsilon \left(\sum_{k\in \bN^N} c_k \widetilde{T}_1^{k_1} \cdots \widetilde{T}_N^{k_N}\right)
	=\inf_{k,c_k\neq 0} \{ \ord_p c_k -\epsilon|k| \}, c_k\in W(k), |k|=k_1+\cdots +k_N.
\]
and define $W^\dagger(\widetilde{A}^\dagger)\subset W(\widetilde{A}^\dagger)$ by
\[
	W^\dagger(\widetilde{A}^\dagger):=
	\{ (a_0,a_1,\ldots \ )\in W(\widetilde{A}^\dagger) \mid \exists \epsilon > 0, \exists C\in \mathbb{R}, m+\mu_{\epsilon/p^m}(a_m)\ge C \ \text{for all}\ m \}.
\]

\begin{lem}
\label{lem:daggercalc}
$\nu_F(\widetilde{T}_i)\in W^\dagger (\widetilde{A}^\dagger)$.
\end{lem}
\begin{proof}
We find  $\mu_\epsilon(\widetilde{T}_i)=-\epsilon$ by definition. By the argument of \cite{DLZODRW} Proposition 3.1, we have $\mu_{\epsilon/p^j}(F^j (\widetilde{T}_i))\ge -\epsilon.$
Let $\alpha_{i,j}:=\mathbf{w}_j(\nu_F(\widetilde{T}_i))\in \widetilde{A}^\dagger$.
Then $F^j(\widetilde{T}_i)=\widetilde{T}_i^{p^j} \cdot \alpha_{i,j}$ and we find $\mu_{\epsilon/p^j}(F^j(\widetilde{T}_i))=\mu_{\epsilon/p^j}(\alpha_{i,j})-(\epsilon / p^j)\cdot p^j = \mu_{\epsilon/p^j}(\alpha_{i,j}) - \epsilon$.
Hence we get $\mu_{\epsilon/p^j}(\mathbf{w}_j(\nu_F(\widetilde{T}_i)))=\mu_{\epsilon/p^j}(\alpha_{i,j})\ge 0$.
By the proof of \cite{DLZODRW} Proposition 3.1,
it is equivalent to that $j+\mu_{\epsilon/p^j}(\nu_F(\widetilde{T}_i))\ge 0$.
This means $\nu_F(\widetilde{T}_i)\in W^\dagger (\widetilde{A}^\dagger)$.
\end{proof}

\begin{lem}
\label{lem:daggermaps}
(cf. \cite{DLZODRW} Proposition 3.1)
The projection map $\text{pr}:W(\widetilde{A}^\dagger)\to W(A)$ induces a map $W^\dagger(\widetilde{A}^\dagger)\to W^\dagger(A)$.
\end{lem}

\begin{proof}
We define a pseudovaluation $\mu_\epsilon$ on $A$ by
\[
	\mu_\epsilon \left(\sum_{k\in \bN^N} c_k T_1^{k_1} \cdots T_N^{k_N}\right)
	=\min_{k,c_k\neq 0} \{ -\epsilon|k| \},\ c_k\in k,\ |k|=k_1+\cdots +k_N.
\]
For $k:[1,N]\to \mathbb{Z}\left[\frac{1}{p}\right]_{\ge 0}$ and
$\xi={}^{V^{u(k)}}\eta \in {}^{V^{u(k)}}W(k)$,
we set $\xi X^k := {}^{V^{u(k)}}(\eta X^{p^{u(k)}k})$.
Any element $\alpha=(a_0,a_1,a_2,\ldots)$ of $W(A)$ has a unique expression
\[
	\alpha=\sum_k \xi_k X^k,\ k:[1,N]\to \mathbb{Z}\left[\frac{1}{p}\right]_{\ge 0}, \xi_k \in {}^{V^{u(k)}}W(k),
\]
where $u(k)$ is the denominator of $k$.
The Gauss norm $\gamma_\epsilon$ on $W(A)$ is defined by
\[
	\gamma_\epsilon(\alpha)=
	\inf_k \{ \ord_p \xi_k - \epsilon|k| \}.
\]
By \cite{DLZOWV} Proposition 2.18 and (2.2),
$\gamma_\epsilon (\alpha) = \inf_m \{ m + \mu_{\epsilon/p^m}(a_m) \}.$
Hence the projection map $\text{pr}$ maps any element of $W^\dagger(\widetilde{A}^\dagger)$ to an element of $W^\dagger(A)$.
\end{proof}

\begin{lem}
If $(1,a_1,a_2,\ldots \ )$ is in $W^\dagger(\widetilde{A}^\dagger)$,
$(a_1,a_2,\ldots \ )$ is in $W^\dagger(\widetilde{A}^\dagger)$.
\end{lem}

\begin{proof}
Since $(1,a_1,a_2,\ldots \ ) \in W^\dagger(\widetilde{A}^\dagger)$,
there exist $\epsilon' > 0$ and $C'\in \mathbb{R}$ such that
$m+\mu_{\epsilon'/p^m}(a_m)\ge C'$ for all $m \ge 1$.
By setting $\epsilon := \epsilon'/p$ and $C := C'-1$, we find an inequality $(m-1)+\mu_{\epsilon/p^{m-1}}(a_m)\ge C$ for all $m \ge 1$
which shows $(a_1,a_2,\ldots \ ) \in W^\dagger(\widetilde{A}^\dagger)$.
\end{proof}

We obtain
\[
	W^\dagger (\widetilde{A}^\dagger)\cap (1+{}^V W(\widetilde{A}^\dagger)) \subset 1+{}^V W^\dagger (\widetilde{A}^\dagger).
\]
By this argument and by Lemma \ref{lem:daggercalc}, we have $\nu_F(\widetilde{T}_i) \in 1+{}^V W^\dagger(\widetilde{A}^\dagger)$.
The element $\lambda_F(\widetilde{T}_i)$ belongs to $1+{}^V W^\dagger(A)$ by Lemma \ref{lem:daggermaps}.
Using the functoriality, $\lambda_\phi(\widetilde{t}_i)$ belongs to $1+{}^V W^\dagger(B)$.

As a result, we obtain the following diagram (cf. \cite{ICrys} (1.3.18)):
\begin{align}
\label{diag:logwittlift}
\xymatrixcolsep{3pc}
\xymatrix{
	\bN^d \ar[r]^-{(\text{id},\lambda_\phi)} \ar[d]
	&
	\bN^d\oplus (1+{}^V W^\dagger(B)) \ar[d]
	\\
	\widetilde{B}^\dagger \ar[r]^{t_\phi}
	&
	W^\dagger(B),
}
\end{align}
where the right vertical arrow is induced by $\bN^d\to W^\dagger(B);e_i\mapsto [t_i]$ and the natural inclusion $1 + {}^V W^\dagger(B) \hookrightarrow W^\dagger(B)$.
One sees $t_\phi(\widetilde{t}_i)=[t_i]\cdot \lambda_\phi (\widetilde{t}_i)$.

Now we can construct a map from the differential complex with logarithmic poles $\Omega_{\widetilde{B}^\dagger/W(k)}^\bullet(\mathcal{D})
:=\widetilde{B}^\dagger\otimes _{\widetilde{B}}\Omega_{\widetilde{B}/W(k)}^\bullet(\mathcal{D})
=\widetilde{B}^\dagger\otimes _{\widetilde{B}}\Lambda_{(\widetilde{B},\bN^d)/W(k)}^\bullet$
defined in \cite{Tsu} \S 3 to the logarithmic overconvergent de Rham-Witt complex $W^\dagger \Omega_{B/k}^\bullet(D):=W^\dagger \Lambda_{(B,\bN^d)/k}^\bullet$:
\[
	\sigma:
	\Omega_{\widetilde{B}^\dagger/W(k)}^\bullet(\mathcal{D})
	\to
	W^\dagger \Omega_{B/k}^\bullet(D),
\]
which is induced by the diagram (\ref{diag:logwittlift}). We see $\sigma(d\log \widetilde{t}_i)= d\log [t_i] + d\lambda_\phi(\widetilde{t}_i)/\lambda_\phi(\widetilde{t}_i)$ and
$\sigma(d\widetilde{t}_i)= d t_\phi(\widetilde{t}_i)$. Note that $\lambda_\phi(\widetilde{t}_i)\in 1 + {}^V W^\dagger(B)\subset W^\dagger(B)^\times$.

We define a filtration $\{ P_j \}$ on the de Rham complex and the de Rham-Witt complex by
\begin{align*}
	P_j(\Omega_{\widetilde{B}/W(k)}^i(\mathcal{D})):
	&=\text{image}(\Omega_{\widetilde{B}/W(k)}^j(\mathcal{D})\otimes_{\widetilde{B}} \Omega_{\widetilde{B}/W(k)}^{i-j}\to \Omega_{\widetilde{B}/W(k)}^i(\mathcal{D})),\\
	P_j(\Omega_{\widetilde{B}^\dagger/W(k)}^i(\mathcal{D})):
	&=\text{image}(\Omega_{\widetilde{B}^\dagger/W(k)}^j(\mathcal{D})\otimes_{\widetilde{B}^\dagger} \Omega_{\widetilde{B}^\dagger/W(k)}^{i-j}\to \Omega_{\widetilde{B}^\dagger/W(k)}^i(\mathcal{D})),\\
	P_j(W \Omega_{B/k}^i(D)):
	&=\text{image}(W \Omega_{B/k}^j(D)\otimes_{W(B)} W \Omega_{B/k}^{i-j}\to W \Omega_{B/k}^i(D)),\\
	P_j(W^\dagger \Omega_{B/k}^i(D)):
	&=P_j(W \Omega_{B/k}^i(D))\cap W^\dagger \Omega_{B/k}^i(D).
\end{align*}
Since
\[
	\text{image}(W^\dagger \Omega_{B/k}^j(D)\otimes_{W^\dagger(B)}W^\dagger \Omega_{B/k}^{i-j}
	\to
	W^\dagger \Omega_{B/k}^i(D))\subset P_j(W^\dagger \Omega_{B/k}^i(D)),
\]
$\sigma$ induces
\[
	\sigma :
	P_j\Omega_{\widetilde{B}^\dagger/W(k)}^\bullet(\mathcal{D})
	\to
	P_j W^\dagger \Omega_{B/k}^\bullet(D).
\]
The canonical morphism $Gr_jW^\dagger \Omega_{B/k}^\bullet (D)\to Gr_jW \Omega_{B/k}^\bullet (D)$ is injective.
For a subset $J=\{\alpha_1,\ldots,\alpha_{j}\}$ of $[1,d]$,
put $\widetilde{B}_J:=\widetilde{B}/(\widetilde{t}_{\alpha_1},\ldots, \widetilde{t}_{\alpha_j})$.

\begin{lem}
There are residue isomorphisms of de Rham complexes:
\begin{align*}
	\Res &: Gr_j(\Omega_{\widetilde{B}/W(k)}^\bullet(\mathcal{D}))
	\to
	\bigoplus_{|J|=j}\Omega_{\widetilde{B}_J/W(k)}^{\bullet -j},
	\\
	\Res &: Gr_j(\Omega_{\widetilde{B}^\dagger/W(k)}^\bullet(\mathcal{D}))
	\to
	\bigoplus_{|J|=j}\Omega_{\widetilde{B}^\dagger_J/W(k)}^{\bullet -j}.
\end{align*}
\end{lem}

\begin{proof}
The first claim is \cite{De} II Proposition 3.6.
Since $\widetilde{B}\to \widetilde{B}^\dagger$ is flat (\cite{Me} Proposition 3)
and $\widetilde{B}^\dagger \otimes_{\widetilde{B}}\widetilde{B}_J \simeq \widetilde{B}_J^\dagger$, one sees
\[
	\widetilde{B}^\dagger \otimes_{\widetilde{B}} Gr_j(\Omega_{\widetilde{B}/W(k)}^\bullet(\mathcal{D}))
	\simeq
	Gr_j(\Omega_{\widetilde{B}^\dagger/W(k)}^\bullet(\mathcal{D}))
\]
and therefore
\begin{align*}
	\bigoplus_{|J|=j}\Omega_{\widetilde{B}_J^\dagger/W(k)}^{\bullet-j}
	&\simeq \widetilde{B}^\dagger \otimes_{\widetilde{B}}\left(\bigoplus_{|J|=j} \Omega_{\widetilde{B}_J/W(k)}^{\bullet-j}\right)\\
	&\simeq \widetilde{B}^\dagger \otimes_{\widetilde{B}} (Gr_j(\Omega_{\widetilde{B}/W(k)}^\bullet(\mathcal{D}))\\
	&\simeq Gr_j(\Omega_{\widetilde{B}^\dagger/W(k)}^\bullet(\mathcal{D})).
\end{align*}
\end{proof}

We prove the de Rham-Witt version.
Note that we have an isomorphism
\[
	\Res : Gr_j W \Omega_{B/k}^\bullet(D)
	\simeq
	\bigoplus_{|J|=j}W \Omega_{B_J/k}^{\bullet-j}
\]
by a similar proof to that of Lemma \ref{lem:resisom}.

\begin{lem}
\label{lem:residue-number-of-log-poles}
Let $j$ be an integer such that $0 \le j \le d$.
Let
\[
\xymatrix{
	\bN^d \ar[r] \ar@{=}[d]
	&
	A=k[T_1,\ldots,T_N] \ar@{->>}[d]^{q}
	\\
	\bN^d \ar[r]
	&
	B
}
\]
be a presentation of the pre-log ring $(B,\bN^d)$.
We denote by $\phi$ the natural morphism
\[
	W_m\Omega_{A/k}^\bullet(D) := W_m\Lambda_{(A,\bN^d)/k}^\bullet
	\to
	W_m\Omega_{B/k}^\bullet(D).
\]
Let $\omega = \sum_{k,\mathcal{P},|I_{-\infty}| > j} \epsilon (\xi_{k,\mathcal{P}},k,\mathcal{P})$
be an element of $W_m\Omega_{A/k}^\bullet(D)$.

Then we have $\phi(\omega)=0$ if $\phi(\omega) \in P_j W_m\Omega_{B/k}^\bullet(D)$.
\end{lem}

\begin{proof}
By Proposition \ref{prop:exact-sequence-Kahler-diffs} (2), we have
\[
	\ker \phi =
	W_m(I) W_m\Omega_{A/k}^\bullet (D) + dW_m(I) W_m \Omega_{A/k}^{\bullet-1} (D) \subset W_m \Omega_{A/k}^\bullet (D)
\]
where $I=\ker (q:A \to B)$.

When $J = \{\alpha_1,\ldots, \alpha_r\}$ is a subset of $[1,d]$,
we set
\[
	A_J = k[T_1,\ldots,\widehat{T_{\alpha_1}},\ldots,\widehat{T_{\alpha_r}},\ldots,T_N]
\]
and $B_J := B/(t_{\alpha_1},\ldots, t_{\alpha_r})$.
The map $q$ induces a surjective map $q_J:A_J \to B_J$.
We define $I_J := \ker q_J$.

Let $\omega = \sum_{k,\mathcal{P},|I_{-\infty}| > j} \epsilon (\xi_{k,\mathcal{P}},k,\mathcal{P})$ be an element of $W_m\Omega_{A/k}^\bullet(D)$
such that $\phi(\omega) \in P_j W_m\Omega_{B/k}^\bullet(D)$.

By the construction of the (log) basic Witt differentials and
Proposition \ref{prop:expression-basic-witt-sncd} and \cite{LZDRW} Proposition 2.17,
 we obtain an isomorphism
\[
	\bigoplus_{r=0}^d \bigoplus_{|J|=r} W_m \Omega_{A_J/k}^\bullet
	\xrightarrow{\sim}
	W_m\Omega_{A/k}^\bullet (D).
\]
We see $\omega$ is uniquely written as the sum
$\sum_{|J|>j} d\log X_J \cdot \omega_J$ where $\omega_J \in W_m \Omega_{A_J/k}^\bullet$ via this isomorphism.
For $0 \le s \le d$, we set $\omega_s:=\sum_{|J|=s} d\log X_J \cdot \omega_J$.

Let $r$ be an integer such that $j+1 \le r \le d$ and $\omega_{r+1}=\omega_{r+2}=\cdots=\omega_{d}=0$.
It follows that $\omega = \omega_{j+1}+\cdots+\omega_r$
and $\omega \in P_r W_m \Omega_{A/k}^\bullet (D)$.

We have the following commutative diagram:
\[
\xymatrix{
	Gr_r W_m \Omega_{A/k}^\bullet(D) \ar[r]^{\phi_r} \ar[d]_{\sim}^{\Res}
	&
	Gr_r W_m \Omega_{B/k}^\bullet(D) \ar[d]_{\sim}^{\Res}
	\\
	\bigoplus_{|J|=r} W_m \Omega_{A_J/k}^\bullet \ar[r]^{\bigoplus_{|J|=r} \phi_J}
	&
	\bigoplus_{|J|=r} W_m \Omega_{B_J/k}^\bullet.
}
\]
Here $\phi_r$ is the induced morphism of $\phi$.
We have $\phi_r(\omega)=0$ because
$\phi(\omega) \in P_j W_m\Omega_{B/k}^\bullet(D)$ and $r \ge j+1$.
The image of $\omega$ is mapped to
\[
	(\omega_J)_{|J|=r} \in \bigoplus_{|J|=r} W_m \Omega_{A_J/k}^\bullet
\]
by the residue isomorphism.

For any $J$ satisfying $|J|=r$, we obtain $\phi_J(\omega_J)=0$.
By \cite{LZGMconn}, $\ker \phi_J$ is equal to
\[
	W_m(I_J) W_m \Omega_{A_J/k}^\bullet + dW_m(I_J) W_m \Omega_{A_J/k}^{\bullet-1} \subset W_m \Omega_{A_J/k}^\bullet.
\]
We have $I_J = A_J \cap I$ via the canonical inclusion $A_J \subset A$.
Then we see $d\log X_J \cdot \omega_J \in W_m\Omega_{A/k}^\bullet (D)$ belongs to
\[
	\ker \phi =
	W_m(I) W_m\Omega_{A/k}^\bullet (D) + dW_m(I) W_m \Omega_{A/k}^{\bullet-1} (D) \subset W_m \Omega_{A/k}^\bullet (D).
\]
It shows $\phi(\omega_r)=\sum_{|J|=r} \phi(d\log X_J \cdot \omega_J)=0$.
If we consider $\omega':=\omega-\omega_r$ instead of $\omega$,
we find that $\omega'$ has an expression of the form $\sum_{k,\mathcal{P},|I_{-\infty}| > j} \epsilon (\xi_{k,\mathcal{P}},k,\mathcal{P})$
and that $\phi(\omega') \in P_j W_m\Omega_{B/k}^\bullet(D)$.
It is clear that $\omega'_r=\omega'_{r+1}=\cdots=\omega'_{d}=0$.

By descending induction on $r$, we find $\omega$ is a sum of elements of $\ker \phi$.
\end{proof}

\begin{lem}
The residue isomorphism $\Res$ induces the residue isomorphism of overconvergent de Rham-Witt complexes:
\begin{align*}
	\Res : Gr_j W^\dagger \Omega_{B/k}^\bullet(D)
	\to
	\bigoplus_{|J|=j}W^\dagger \Omega_{B_J/k}^{\bullet-j}.
\end{align*}
\end{lem}

\begin{proof}
Take a presentation
\[
\xymatrix{
	\bN^d \ar[r] \ar@{=}[d]
	&
	A=k[T_1,\ldots,T_N] \ar@{->>}[d]^{q}
	\\
	\bN^d \ar[r]
	&
	B
}
\]
and consider the following commutative diagram:
\[
\xymatrix{
	Gr_j W^\dagger \Omega_{B/k}^\bullet(D) \ar[r]^\subset
	&
	Gr_j W\Omega_{B/k}^\bullet(D) \ar[r]^-{\Res}_{\sim}
	&
	\bigoplus_{|J|=j}W\Omega_{{B_J}/k}^{\bullet-j}
	\\
	Gr_j W^\dagger \Omega_{A/k}^\bullet(D) \ar[r]^\subset \ar[u]
	&
	Gr_j W\Omega_{A/k}^\bullet(D) \ar[r]^-{\Res}_{\sim} \ar[u]
	&
	\bigoplus_{|J|=j}W\Omega_{{A_J}/k}^{\bullet-j}. \ar[u]
}
\]
We prove that the map $Gr_j W^\dagger \Omega_{A/k}^\bullet(D)\to Gr_j W^\dagger \Omega_{B/k}^\bullet(D)$ is surjective.
It suffices to show that the map $P_j W^\dagger \Omega_{A/k}^\bullet(D) \to P_j W^\dagger \Omega_{B/k}^\bullet(D)$
induced by $q : W\Omega_{A/k}^\bullet(D)\to W\Omega_{B/k}^\bullet(D)$ is surjective.

Let $\bar{\omega} \in P_j W^\dagger \Omega_{B/k}^\bullet(D)$.
Since $W^\dagger \Omega_{A/k}^\bullet(D) \to W^\dagger \Omega_{B/k}^\bullet(D)$ is surjective,
there exists an element $\omega =
\sum_{k,\mathcal{P}} \epsilon (\xi_{k,\mathcal{P}},k,\mathcal{P})
\in W^\dagger \Omega_{A/k}^\bullet(D)$
such that $q(\omega)=\bar{\omega}$.
Set $\omega_1 = \sum_{k,\mathcal{P},|I_{-\infty}| \le j} \epsilon (\xi_{k,\mathcal{P}},k,\mathcal{P})$
and
$\omega_2 = \sum_{k,\mathcal{P},|I_{-\infty}| > j} \epsilon (\xi_{k,\mathcal{P}},k,\mathcal{P})$.
Then we see $\omega_1 \in P_j W^\dagger \Omega_{A/k}^\bullet(D)$
and $\omega = \omega_1 + \omega_2$.
By Lemma \ref{lem:residue-number-of-log-poles}, we get $q(\omega_2)=0$.
Hence we have $q(\omega_1) = q(\omega) = \bar{\omega}$.
This implies $P_j W^\dagger \Omega_{A/k}^\bullet(D) \to P_j W^\dagger \Omega_{B/k}^\bullet(D)$
is surjective.
Thus, we can assume $B$ is a polynomial ring $A=k[T_1,\ldots,T_N]$.

In this case, the elements $\omega$ of $Gr_j W \Omega_{A/k}^\bullet(D)$ is in $Gr_j W^\dagger \Omega_{A/k}^\bullet(D)$
if and only if it can be written as a overconvergent sum of log basic Witt differentials $\epsilon (\xi,k,\mathcal{P})$
such that $|I_\infty| = j$.
Hence
\[
	\Res:Gr_j W\Omega_{A/k}^\bullet(D) \simeq \bigoplus_{|J|=j}W\Omega_{{A_J}/k}^{\bullet-j}
\]
induces an isomorphism
\[
	\Res:Gr_j W^\dagger \Omega_{A/k}^\bullet(D) \simeq \bigoplus_{|J|=j}W^\dagger \Omega_{{A_J}/k}^{\bullet-j}.
\]
\end{proof}

\begin{thm}
\label{thm:logcomparison}
Let $\tau = 4d\lfloor \log_p \dim B \rfloor$. Then the kernel and cokernel of the homomorphism
\[
	\sigma_*:H^i(\Omega_{\widetilde{B}^\dagger/W(k)}^\bullet(\mathcal{D}))
	\to
	H^i(W^\dagger \Omega_{B/k}^\bullet(D))
\]
induced by $\sigma$ are annihilated by $p^{\tau}$. In particular, $\sigma_*$ is an isomorphism if $\dim B < p$.

There is a rational isomorphism
\[
H_{\logMW}^*((X,D)/K)\simeq H^*(W^\dagger \Omega_{B/k}^\bullet(D)\otimes_{W(k)}K)
\]
between log Monsky-Washnitzer cohomology and logarithmic overconvergent de Rham-Witt cohomology.
\end{thm}
\begin{proof}
Consider the following commutative diagram:
\[
\xymatrix{
	Gr_j(\Omega_{\widetilde{B}^\dagger/W(k)}^\bullet (\mathcal{D}))\ar[d]^{\Res}_{\sim} \ar[r]
	&
	Gr_j(W^\dagger \Omega_{B/k}^\bullet (D))\ar[d]^{\Res}_{\sim}
	\\
	\bigoplus_{|J|=j}\Omega_{\widetilde{B}_J^\dagger/W(k)}^{\bullet-j} \ar[r]
	&
	\bigoplus_{|J|=j}W^\dagger \Omega_{{B_J}/k}^{\bullet-j}.
}
\]
Let $\kappa:=\lfloor \log_p \dim B \rfloor$. By \cite{DLZODRW} Proposition 3.24,
the kernel and cokernel of
\[
	H^i(\Omega_{\widetilde{B}_J^\dagger/W(k)}^\bullet)
	\to
	H^i(W^\dagger \Omega_{{B_J}/k}^\bullet)
\]
are annihilated by $p^{2\kappa(J)}$ where $\kappa(J)=\lfloor \log_p \dim B_J \rfloor$.
Thus the kernel and cokernel of
\[
	H^i(Gr_j(\Omega_{\widetilde{B}^\dagger/W(k)}^\bullet (\mathcal{D})))
	\to
	H^i(Gr_j(W^\dagger \Omega_{B/k}^\bullet (D)))
\]
are annihilated by $p^{2\kappa}$. Consider the following exact sequences:
\[
\xymatrix{
	0 \ar[r]
	&
	P_{j-1}\Omega_{\widetilde{B}^\dagger/W(k)}^\bullet(\mathcal{D}) \ar[r] \ar[d]
	&
	P_j\Omega_{\widetilde{B}^\dagger/W(k)}^\bullet (\mathcal{D}) \ar[r] \ar[d]
	&
	Gr_j\Omega_{\widetilde{B}^\dagger/W(k)}^\bullet (\mathcal{D}) \ar[r] \ar[d]
	&
	0
	\\
	0 \ar[r]
	&
	P_{j-1}W^\dagger \Omega_{B/k}^\bullet (D) \ar[r]
	&
	P_jW^\dagger \Omega_{B/k}^\bullet (D) \ar[r]
	&
	Gr_j W^\dagger \Omega_{B/k}^\bullet (D) \ar[r]
	&
	0.
}
\]
It induces a long exact sequences of cohomology of chain complexes:
\[
\xymatrix@C=1em{
	\ar[r]
	&
	H^r(P_{j-1}\Omega_{\widetilde{B}^\dagger/W(k)}^\bullet (\mathcal{D})) \ar[r] \ar[d]
	&
	H^r(P_j\Omega_{\widetilde{B}^\dagger/W(k)}^\bullet (\mathcal{D})) \ar[r] \ar[d]
	&
	H^r(Gr_j\Omega_{\widetilde{B}^\dagger/W(k)}^\bullet (\mathcal{D})) \ar[r] \ar[d]
	&
	\\
	\ar[r]
	&
	H^r(P_{j-1}W^\dagger \Omega_{B/k}^\bullet (D)) \ar[r]
	&
	H^r(P_j W^\dagger \Omega_{B/k}^\bullet (D)) \ar[r]
	&
	H^r(Gr_j W^\dagger \Omega_{B/k}^\bullet (D)) \ar[r]
	&
	.
}
\]

By diagram chase and induction, we find that the kernel and cokernel of
\[
	H^r(P_j\Omega_{\widetilde{B}^\dagger/W(k)}^\bullet (\mathcal{D}))
	\to
	H^r(P_j W^\dagger \Omega_{B/k}^\bullet (D))
\]
are annihilated by $p^{\alpha(j)}$
where $\alpha(j)=4j\lfloor \log_p \dim B \rfloor$.
Since $P_d\Omega_{\widetilde{B}^\dagger/W(k)}^\bullet (\mathcal{D})=\Omega_{\widetilde{B}^\dagger/W(k)}^\bullet (\mathcal{D})$
and $P_d W^\dagger \Omega_{B/k}^\bullet (D) =W^\dagger \Omega_{B/k}^\bullet (D)$, we get the claim.
\end{proof}

\begin{prop}
\label{prop:affine-comparison-open}
Let $\displaystyle C:=B\left[ \frac{1}{t_1\cdots t_d}\right]$ and $Y:=\Spec C$.
Let
\[
	\theta :
	W\Omega_{B/k}^\bullet(D)
	\to
	W\Omega_{C/k}^\bullet,\
	d\log [t_i]
	\mapsto
	\frac{[t_1 \cdots \widehat{t_i} \cdots t_d]}{[t_1  \cdots t_d]} d[t_i]
\]
be the canonical morphism induced by the universal property of $W\Omega_{B/k}^\bullet(D)$. Then:

(1) $\theta$ induces a morphism
\[
	\theta^\dagger:
	W^\dagger\Omega_{B/k}^\bullet(D)
	\to
	W^\dagger\Omega_{C/k}^\bullet.
\]

(2) $\theta^\dagger$ induces an isomorphism of cohomology groups
\[
	\theta^\dagger:
	H^* (W^\dagger\Omega_{B/k}^\bullet(D)\otimes_{W(k)}K)
	\to
	H^* (W^\dagger\Omega_{C/k}^\bullet\otimes_{W(k)}K).
\]
\end{prop}

\begin{proof}
(1) Choose a presentation
\[
\xymatrix{
	\bN^d \ar[r] \ar@{=}[d]
	&
	k[T_1,\ldots,T_N] \ar@{->>}[d]^{\lambda}
	\\
	\bN^d \ar[r]
	&
	B.
}
\]
Then we have a presentation $\lambda' : k[T_1,\ldots,T_N,S]\to C$ induced by $\lambda$ and $S\mapsto 1/(t_1\cdots t_d)$.
We obtain two surjective morphisms
\begin{align*}
	\tau:
	W \Lambda_{(k[T_1,\ldots,T_N],\bN^d)/k}^\bullet
	&\to W \Omega_{B/k}^\bullet(D),\\
	\tau':W\Omega_{k[T_1,\ldots,T_N,S]/k}^\bullet
	&\to W \Omega_{C/k}^\bullet.
\end{align*}

Let $\omega$ be any element of $W^\dagger \Lambda_{(k[T_1,\ldots,T_N],\bN^d)/k}^\bullet$.
Then as in the proof of Proposition \ref{prop:indep}, $\omega$ can be written as $\omega=\sum_{J\subset [1,d]} ( \prod_{i\in J}d\log X_i ) \bar{\omega}_J$
where $\bar{\omega}_J \in W^\dagger\Omega_{(k[T_1,\ldots,T_N])/k}^\bullet$.

We set
\[
	\widetilde{\omega}:=
	\sum_{J\subset [1,d]} \left( \left( \prod_{i\in J} (Y\cdot X_1\cdots \widehat{X_i} \cdots X_d \cdot dX_i) \right)
	\cdot \bar{\omega}_J \right)
	\in W\Omega_{k[T_1,\ldots,T_N,S]/k}^\bullet
\]
where $Y=[S]$.
It is easy to see that $\theta (\tau (\omega))=\tau'(\widetilde{\omega})$.

One finds $\widetilde{\omega} \in W^\dagger \Omega_{k[T_1,\ldots,T_N,S]/k}^\bullet$ because we have
\[
	\prod_{i\in J} (Y\cdot X_1\cdots \widehat{X_i}\cdots X_d \cdot dX_i) \in W^\dagger \Omega_{k[T_1,\ldots,T_N,S]/k}^\bullet,\
	\bar{\omega}_J \in W^\dagger \Omega_{k[T_1,\ldots,T_N]/k}^\bullet
\]
and $W^\dagger \Omega_{k[T_1,\ldots,T_N,S]/k}^\bullet$ is a ring.
Therefore $\theta$ induces a morphism
\[
	\theta^\dagger:W^\dagger\Omega_{B/k}^\bullet(D)
	\to
	W^\dagger\Omega_{C/k}^\bullet.
\]

(2) A Frobenius map (in the sence of Definition \ref{def:Frob-endmorphism})
on $\widetilde{B}$ induces a Frobenius map on
$\widetilde{C}:=\widetilde{B}\left[ 1/(\tilde{t}_1\cdots\tilde{t}_d)\right]$.
Hence we have a commutative diagram
\[
\xymatrix{
	\Omega_{\widetilde{B}^\dagger/W(k)}^\bullet(\mathcal{D}) \ar[r] \ar[d]
	&
	W\Omega_{B/k}^\bullet(D)\ar[d]
	\\
	\Omega_{\widetilde{C}^\dagger/W(k)}^\bullet \ar[r]
	&
	W \Omega_{C/k}^\bullet.
}
\]
By (1), it induces the following commutative diagram
\[
\xymatrix{
	\Omega_{\widetilde{B}^\dagger/W(k)}^\bullet(\mathcal{D}) \ar[r] \ar[d]
	&
	W^\dagger \Omega_{B/k}^\bullet(D)\ar[d]
	\\
	\Omega_{\widetilde{C}^\dagger/W(k)}^\bullet \ar[r]
	&
	W^\dagger \Omega_{C/k}^\bullet.
}
\]
Tensoring with $K$ and taking cohomology, we obtain the following commutative diagram
\[
\xymatrix{
	H_{\logMW}^*((X,D)/K) \ar[r] \ar[d]
	&
	H^*(W^\dagger \Omega_{B/k}^\bullet(D)\otimes_{W(k)}K)\ar[d]
	\\
	H_{\MW}^*(Y/K) \ar[r]
	&
	H^*(W^\dagger \Omega_{C/k}^\bullet \otimes_{W(k)}K).
}
\]
The horizontal arrows and the left vertical arrow are isomorphisms
by Theorem \ref{thm:logcomparison} and \cite{DLZODRW} Corollary 3.25 and \cite{Tsu} Theorem 3.5.1.
Hence the right vertical arrow is also an isomorphism.
\end{proof}

\subsection{Sheaf of overconvergent log de Rham-Witt complex}
In this subsection, we define the Zariski sheaf of overconvergent log de Rham-Witt complexes
for smooth schemes with simple normal crossing divisor.

\begin{prop}
\label{prop:OdRW-well-defined-affine}
(cf. \cite{DLZODRW} Proposition 1.2)

Let $X=\Spec B$ be a smooth affine scheme and $D$ a simple normal crossing divisor on $X$.
We assume that there is a global chart $\alpha:\bN^d\to \mathcal{M}_{(X,D)}$.
Then we have a pre-log structure
\[
	\beta=\beta_\alpha :
	\bN^d
	\xrightarrow{\alpha}
	\mathcal{M}_{(X,D)}(X)
	\to
	\mathcal{O}_X(X) = B
\]
of $B$.
Let $D_1,\ldots,D_d$ be the irreducible components of $D$.
We also assume that there is an \'etale morphism
$X\xrightarrow{w}\mathbb{A}_k^n=\Spec k[T_1,\ldots,T_n]$
such that $w(T_i)=\beta(e_i)=:t_i \in B$ and that $D_i$ is defined by $t_i=0$ for $1 \le i \le d$.
We fix a nonnegative integer $r$.

(1) We denote by $f\in B$ an arbitrary element.
Then $\beta$ induces a pre-log structure $\beta_f : \bN^d \xrightarrow{\beta} B \to B_f$ of $B_f$.
The presheaf
\[
	D(f)
	\mapsto
	W^\dagger \Lambda_{(B_f,\bN^d)/k}^r
\]
defines a sheaf on the Zariski topology on $X$.
We denote by $W^\dagger \Lambda_{(X,D)/k,\alpha}$ this sheaf.

(2) The Zariski sheaf $W^\dagger \Lambda_{(X,D)/k, \alpha}^r$ is independent of the choice of charts $\alpha$.
We denote by $W^\dagger \Lambda_{(X,D)/k}^r$ this Zariski sheaf.

(3) The Zariski cohomology of the sheaf $W^\dagger \Lambda_{(X,D)/k}^r$ vanishes in degree $j >0$, i.e.,
\[
	H_{\Zar}^j (X, W^\dagger \Lambda_{(X,D)/k}^r)=0.
\]
\end{prop}

\begin{proof}
Let $\{f_i \}_{i=1}^l$ be a finite family of elements of $B$ such that $f_i$ generate $B$ as an ideal.
For $1 \le i_1 < \cdots < i_s \le l$, we denote by $\mathfrak{U}_{i_1,\ldots,i_s}$ the intersection $D(f_{i_1})\cap \cdots \cap D(f_{i_s})$.
For simplicity, we set $B_{i_1\cdots i_s}:=B_{f_{i_1}\cdots f_{i_s}}$.
We define a \v{C}ech complex $C^\bullet = C^\bullet((X,D),\alpha)$ by
$
C^0:=W^\dagger \Lambda_{(B,\bN^d)/k}^r
$
and
$
C^s:=\bigoplus_{1 \le i_1 < \cdots < i_s \le l} W^\dagger \Lambda_{(B_{i_1\cdots i_s},\bN^d)/k}^r.
$
for $s \ge 1$.
Then we see the filtration $\{P_j\}_j$ which we introduced in \S \ref{comparison-log-MW} induces a filtration on $C^\bullet$:
\[
	P_j C^s =
	\bigoplus_{1 \le i_1 < \cdots < i_s \le l}
	P_j W^\dagger \Lambda_{(B_{i_1\cdots i_s},\bN^d)/k}^r.
\]
Set $Gr_j C^\bullet:= P_j C^\bullet/P_{j-1} C^\bullet$. Then one has an exact sequence of complexes
\[
0 \to P_{j-1} C^\bullet \to P_j C^\bullet \to Gr_j C^\bullet \to 0.
\]
By the Poincar\'e residue map, one obtains a commutative diagram
\[
\xymatrix{
	\bigoplus_{1 \le i_1 < \cdots < i_{s-1} \le l} \bigoplus_{|J|=j}W^\dagger \Omega_{B_{i_1\cdots i_{s-1}J}/k}^{r-j} \ar[r]^-{{\Res}^{-1}}_-\sim \ar[d]
	&
	Gr_j C^{s-1} \ar[d] \\
	\bigoplus_{1 \le i_1 < \cdots < i_{s} \le l} \bigoplus_{|J|=j}W^\dagger \Omega_{B_{i_1\cdots i_{s}J}/k}^{r-j} \ar[r]^-{{\Res}^{-1}}_-\sim
	&
	Gr_j C^{s}.
}
\]
Here $B_{i_1\cdots i_{s}J}$ denotes $B_{f_{i_1}\cdots f_{i_s}}/(t_{\alpha_1},\ldots, t_{\alpha_j})$ of $B$ for $J=\{\alpha_1,\ldots,\alpha_{j}\}\subset [1,d]$.

We see that
\[
	\bigoplus_{1 \le i_1 < \cdots < i_{s} \le l}
	\bigoplus_{|J|=j}W^\dagger \Omega_{B_{i_1\cdots i_{s}J}/k}^{r-j}
	\simeq
	\bigoplus_{|J|=j} \widetilde{C}^{j}(\Spec B_J, W^\dagger \Omega_{B_J/k}^{r-j}),
\]
where $\widetilde{C}^\bullet(\Spec B_J, W^\dagger \Omega_{B_J/k}^{r-j})$ is the \v{C}ech complex
with degree $s$ elements given by
\[
	\widetilde{C}^{s}(\Spec B_J, W^\dagger \Omega_{B_J/k}^{r-j})=\bigoplus_{1 \le i_1 < \cdots < i_{s} \le l} W^\dagger \Omega_{B_{i_1\cdots i_{s}J}/k}^{r-j}.
\]
Hence the Poincar\'e residue map induces an isomorphism
\[
	\Res : Gr_j C^{s} \xrightarrow{\sim} \bigoplus_{|J|=j} \widetilde{C}^{s}(\Spec B_J, W^\dagger \Omega_{B_J/k}^{r-j}).
\]
The boundary morphism of $Gr_j C^\bullet$ is identified to the direct sum of boundary morphisms of
$\{ \widetilde{C}^\bullet(\Spec B_J, W^\dagger \Omega_{B_J/k}^{r-j}) \}_{|J|=j}$.
It follows that $Gr_j C^\bullet$ is exact by \cite{DLZODRW} Proposition 1.6.
We find $P_j C^\bullet$ is exact by induction for all $j$.
As $P_d C^\bullet=C^\bullet$, we get (1).

We prove (2). Let $\alpha':\bN^d\to \mathcal{M}_{(X,D)}$ be another chart.
We have an isomorphism $W\Lambda_{(B,\bN^d,\beta_\alpha)/k}^\bullet \simeq W\Lambda_{(B,\bN^d,\beta_{\alpha'})/k}^\bullet$ by Proposition-Definition \ref{prop-def:chart-welldefined}.
Let $t'_i:=\beta_{\alpha'}(e_i)$.
We set $B'_J:=B/(t'_{\alpha_1}\cdots t'_{\alpha_j})$ for $J=\{\alpha_1,\ldots,\alpha_{j}\}\subset [1,d]$.
Since $B_J\simeq B'_J$, we see
\[
W^\dagger \Omega_{B_{i_1\cdots i_{s}J}/k}^{r-j}
\simeq W^\dagger \Omega_{B_{i'_1\cdots i'_{s}J}/k}^{r-j}
\]
for all $1 \le i_1 < \cdots < i_s \le l$.
Hence we obtain an isomorphism $Gr_j C^\bullet((X,D),\alpha)\simeq Gr_j C^\bullet((X,D),\alpha')$. Using the exact sequence and induction, we see that
\[
C^\bullet((X,D),\alpha)\simeq C^\bullet((X,D),\alpha').
\]
This shows (2).

(3) is deduced from the exactness of the \v{C}ech complex and
Cartan's criterion (\cite{G}, Th\'eor\`eme 5.9.2),
which allows us to compute the Zariski cohomology of the sheaf of abelian groups
$W^\dagger \Lambda_{(X,D)/k}^r$ via \v{C}ech cohomology in our situation.
\end{proof}

\begin{defn}
Let $X$ be a smooth scheme over $k$ and $D$ be a simple normal crossing divisor on $X$.
Then for any point $x$ of $X$,
there is an affine neighbourhood $U$ of $x$ in $X$
such that the log scheme $(U,D|_U)$ admits a chart of the form $\bN^d\to \mathcal{M}_{(U,D|_U)}$
for some $d$ (cf. \cite{KatoLog}, \S \ref{subsect:topology-of-log-structures})
and that there is an \'etale morphism $U\to \mathbb{A}_k^n$ for some $n$.

By Proposition \ref{prop:OdRW-well-defined-affine},
the Zariski sheaves $W^\dagger \Lambda_{(U,D|_U)/k}^r$ glues together to give a Zariski sheaf $W^\dagger \Lambda_{(X,D)/k}^r$.

We call $W^\dagger \Lambda_{(X,D)/k}^\bullet$ the sheaf of overconvergent log de Rham-Witt complexes.
\end{defn}

\subsection{Comparison with rigid cohomology}
We generalize our results to global cases.
Let $X$ be a smooth quasi-projective variety over a perfect field $k$
and $D$ an SNCD of $X$ over $k$.
Let $j:Y:=X\setminus D\hookrightarrow X$ be the canonical open immersion.
We have the overconvergent de Rham-Witt complex $W^\dagger \Omega_{Y/k}^\bullet$ for a smooth variety $Y$ (\cite{DLZODRW} \S 1),
and the overconvergent log de Rham-Witt complex $W^\dagger \Lambda_{(X,D)/k}^\bullet$ for a smooth variety with SNCD.
The canonical morphism
$W\Lambda_{(X,D)/k}^\bullet \to j_*W \Omega_{Y/k}^\bullet$
induces the map
$W^\dagger \Lambda_{(X,D)/k}^\bullet\to j_*W^\dagger \Omega_{Y/k}^\bullet$.
Davis-Langer-Zink defined a map from the rigid cohomology to the overconvergent de Rham-Witt cohomology
\[
	\mathbb{R}\Gamma_{\rig}(Y/K)
	\to
	\mathbb{R}\Gamma_{\Zar}(Y,W^\dagger \Omega_{Y/k}^\bullet) \otimes K
\]
and showed this is a quasi-isomorphism when $Y$ is smooth and quasi-projective over $k$ (\cite{DLZODRW} Theorem 4.40).

\begin{lem}
We have a canonical morphism
\[
	\mathbb{R}\Gamma_{\Zar}(X,W^\dagger \Lambda_{(X,D)/k}^\bullet)
	\to
	\mathbb{R}\Gamma_{\Zar}(Y,W^\dagger \Omega_{Y/k}^\bullet).
\]
\end{lem}

\begin{proof}
Take a quasi-isomorphism $W^\dagger \Omega_{Y/k}^\bullet \to I^\bullet$
to a complex of injective abelian sheaves on $Y$.
Applying $j_*$, we have a natural map
$j_* W^\dagger \Omega_{Y/k}^\bullet \to \mathbb{R}j_* W^\dagger \Omega_{Y/k}^\bullet$.
The morphism we want is the composition
\begin{align*}
	\tau:
	\mathbb{R}\Gamma_{\Zar}(X,W^\dagger \Lambda_{(X,D)/k}^\bullet)
	&\to \mathbb{R}\Gamma_{\Zar}(X,j_* W^\dagger \Omega_{Y/k}^\bullet) \\
	&\to \mathbb{R}\Gamma_{\Zar}(X,\mathbb{R}j_* W^\dagger \Omega_{Y/k}^\bullet) \\
	&\simeq \mathbb{R}\Gamma_{\Zar}(Y,W^\dagger \Omega_{Y/k}^\bullet).
\end{align*}
\end{proof}

By this lemma we have a diagram
\[
\xymatrix{
	&
	\mathbb{R}\Gamma_{\Zar}(X,W^\dagger \Lambda_{(X,D)/k}^\bullet) \otimes K \ar[d] \\
	\mathbb{R}\Gamma_{\rig}(Y/K) \ar[r]
	&
	\mathbb{R}\Gamma_{\Zar}(Y,W^\dagger \Omega_{Y/k}^\bullet)\otimes K.
}
\]
We show that the vertical arrow is a quasi-isomorphism.
Take an open covering $\{X^{i}\}_{i\in I}$ of $X$ by affine schemes $X^{i}=\Spec A^i$
which satisfy the following condition:
There is an \'etale morphism $X^i\to \mathbb{A}_k^{r_i}$
and $D^i:=X^i\cap D$ is defined by $t_1\cdots t_s=0$ for some $s\le r_i$,
where $t_j$ is the image of $T_i$ of $\mathbb{A}_k^{r_i}$ in $A^i$.
$A^i$ has a smooth lifting $\widetilde{A}^i$ over $W(k)$
which has an \'etale morphism $\widetilde{X}^i = \Spec \widetilde{A}^i \to \mathbb{A}_{W(k)}^{r_i}$
such that $\widetilde{D}^i$ defined by $t_1,\ldots,t_s$ is a lifting of $D^i$.

For a subset $\underline{i}\subset I$
we set $X^{\underline{i}}:=\bigcap_{j\in \underline{i}}X^i, D^{\underline{i}}:=X^{\underline{i}}\cap D, Y^{\underline{i}}:=X^{\underline{i}}\setminus D^{\underline{i}}$.
Since $X$ is quasi-projective, each $X^{\underline{i}}$ is a smooth quasi-projective affine scheme and it satisfies the condition indicated above.
We form a simplicial scheme $X^\bullet$ by $X^n:=X' \times_X \ldots \times_X X'$ ($n$-times),
where $X':=\sqcup_{i\in I}X^i$. The simplicial scheme $Y^\bullet$ is defined in the similar fashion.

Consider the following commutative diagram of simplicial schemes:
\[
\xymatrix{
	Y^\bullet \ar[r]^{\theta_Y} \ar[d]_{j^\bullet}
	&
	Y \ar[d]^j
	\\
	X^\bullet \ar[r]^{\theta_X}
	&
	X.
}
\]

This diagram induces the following diagram:
\[
\xymatrix{
	W^\dagger \Lambda_{(X,D)/k}^\bullet \otimes K \ar[rr]^\sim \ar[d]
	&
	&
	\mathbb{R}\theta_{X*}(W^\dagger \Lambda_{(X^\bullet,D^\bullet)/k}^\bullet )\otimes K \ar[d]
	\\
	j_* (W^\dagger \Omega_{Y/k}^\bullet )\otimes K  \ar[d]
	&
	&
	\mathbb{R}\theta_{X*}((j^\bullet)_* W^\dagger \Omega_{Y^\bullet/k}^\bullet )\otimes K \ar[d]
	\\
	\mathbb{R}j_* (W^\dagger \Omega_{Y/k}^\bullet )\otimes K \ar[r]^-\sim
	&
	\mathbb{R}j_* (\mathbb{R}\theta_{Y*}W^\dagger \Omega_{Y^\bullet/k}^\bullet )\otimes K \ar[r]^-\sim
	&
	\mathbb{R}\theta_{X*}(\mathbb{R}(j^\bullet)_* W^\dagger \Omega_{Y^\bullet/k}^\bullet )\otimes K.
}
\]
By Proposition \ref{prop:affine-comparison-open}, we see
\[
	\mathbb{R}\theta_{X*}(W^\dagger \Lambda_{(X^\bullet,D^\bullet)/k}^\bullet )\otimes K
	\to
	\mathbb{R}\theta_{X*}((j^\bullet)_* W^\dagger \Omega_{Y^\bullet/k}^\bullet )\otimes K
\]
is an isomorphism.

Since $Y^\bullet$ is an affine simplicial scheme,
we conclude $\mathbb{R}^q (j^\bullet)_* W^\dagger \Omega_{Y^\bullet/k}^\bullet =0$
for all $q > 0$ and all $r$.
Hence we have $(j^\bullet)_* W^\dagger \Omega_{Y^\bullet/k}^\bullet
\simeq \mathbb{R}(j^\bullet)_* W^\dagger \Omega_{Y^\bullet/k}^\bullet$.

Hence we find the morphism
\[
	W^\dagger \Lambda_{(X,D)/k}^\bullet \otimes K
	\to
	\mathbb{R}j_* (W^\dagger \Omega_{Y/k}^\bullet )\otimes K
\]
is a quasi-isomorphism.
Therefore, we get the following comparison theorem.

\begin{thm}
Let $X$ be a smooth quasi-projective variety over
a perfect field $k$ and $D$ be a simple normal crossing divisor of $X$.
Then we have an isomorphism
\[
	H^*_{\rig}(Y/K)
	\simeq
	\mathbb{H}^*_{\Zar}(X,W^\dagger \Lambda_{(X,D)/k}^\bullet \otimes_{W(k)}K).
\]
\end{thm}

\bibliographystyle{alpha}
\bibliography{ref.bib}

\end{document}